\documentclass[two-sides,11pt]{amsart}
\usepackage{indentfirst,latexsym,bm}
\usepackage{amsfonts}
\usepackage{amssymb}
\usepackage{amsmath}
\allowdisplaybreaks[4]
\usepackage{amsbsy}
\usepackage{dsfont}
\usepackage{amsthm}
\usepackage{hyperref}
\usepackage{amscd}
\usepackage[all]{xy}
\usepackage{chemarrow}
\usepackage{multirow}
\usepackage[titletoc]{appendix}

\begin{document}

\title{Pointed Hopf algebras of dimension $p^2q$ in characteristic $p$}
\author{Rongchuan Xiong}
\address{Department of Mathematics, Changzhou University, Changzhou 213164, China}
\email{rcxiong@foxmail.com}
\makeatletter
\@namedef{subjclassname@2020}{\textup{2020} Mathematics Subject Classification}
\makeatother

\date{}
\maketitle

\newcommand{\tabincell}[2]{\begin{tabular}{@{}#1@{}}#2\end{tabular}}
\newtheorem{question}{Question}
\newtheorem{defi}{Definition}[section]
\newtheorem{conj}{Conjecture}
\newtheorem{thm}[defi]{Theorem}
\newtheorem{lem}[defi]{Lemma}
\newtheorem{pro}[defi]{Proposition}
\newtheorem{cor}[defi]{Corollary}
\newtheorem{rmk}[defi]{Remark}
\newtheorem{example}{Example}[section]

\newcommand{\AdL}{\text{ad}_L\,}
\newcommand{\AdR}{\text{ad}_R\,}
\newcommand{\K}{\mathds{k}}
\newcommand{\A}{\mathcal{A}}
\newcommand{\C}{\mathcal{C}}
\newcommand{\M}{\mathcal{M}}
\newcommand{\E}{\mathcal{E}}
\newcommand{\D}{\mathcal{D}}
\newcommand{\X}{\mathcal{X}}
\newcommand{\G}{\mathbf{G}}
\newcommand{\Z}{\mathbb{Z}}
\newcommand{\cZ}{\mathcal{Z}}
\newcommand{\I}{\mathbb{I}}
\newcommand{\cI}{\mathcal{I}}
\newcommand{\J}{\mathcal{J}}
\newcommand{\BN}{\mathcal{B}}
\newcommand{\Lam}{\lambda}
\newcommand{\Ome}{\omega}
\newcommand{\HYD}{{}^{H}_{H}\mathcal{YD}}
\newcommand{\As}{^\ast}
\newcommand{\N}{\mathds{N}}
\newcommand{\dH}{\mathds{H}}
\newcommand{\Pp}{\mathcal{P}}
\newcommand{\fH}{\mathfrak{H}}
\newcommand{\bH}{\mathbb{H}}
\newcommand{\fA}{\mathfrak{A}}
\newcommand{\cF}{\mathcal{F}}

\newcommand\ad{\operatorname{ad}}
\newcommand\Ob{\operatorname{Ob}}
 \newcommand\Aut{\operatorname{Aut}}
 \newcommand\coker{\operatorname{coker}}
 \newcommand\Ker{\operatorname{Ker}}
\newcommand\im{\operatorname{Im}}
\newcommand\co{\operatorname{co}}
\newcommand\Der{\operatorname{Der}}
\newcommand\diag{\operatorname{diag}}
\newcommand\End{\operatorname{End}}
\newcommand\id{\operatorname{id}}
\newcommand\Char{\operatorname{char}}
\newcommand\gr{\operatorname{gr}}
\newcommand\GK{\operatorname{GKdim}}
\newcommand{\Hom}{\operatorname{Hom}}
\newcommand\ord{\operatorname{ord}}
\newcommand\cR{\mathcal{R}}
\newcommand\cH{\mathcal{H}}

\begin{abstract}
Let $\mathds{k}$ be an algebraically closed field of characteristic $p$. We give the complete classification of pointed Hopf algebras over  $\mathds{k}$ of dimension $p^2q$ for a prime number $q$. The result shows that there are  finitely many isomorphism classes, including 10 classes that are not generated by group-like elements and skew-primitive elements.  In particular, there are many  new examples of  finite-dimensional pointed Hopf algebras.

\end{abstract}
\section{Introduction}
Let $\mathds{k}$ be  an algebraically closed field of characteristic $p$. This work is devoted to the classification of pointed Hopf algebras over $\K$ of a given dimension.  In general, it is a hard and challenging question even for too many small dimensions. Until now, the complete classifications have been done only for connected Hopf algebras of dimension $p^2$ \cite{W1}, $p^3$ \cite{NWW1,NWW2} and pointed ones of dimension $p^2$ \cite{WW}, $8$ \cite{NW}.

There are some results with some properties. G. Henderson  classified cocommutative  connected Hopf algebras  of dimension $p^2$ and $p^3$ \cite{Hen95};  S. Scherotzke classified rank one pointed Hopf algebras that are generated by the first term of the coalgebra filtration \cite{S}; The author classified non-connected pointed Hopf algebras of dimension $16$ that are generated by group-like elements and skew-primitive elements \cite{X23}; N. Hu et al. constructed   examples of pointed Hopf algebras of dimension $p^n$   \cite{HW07,HW11,TH16,THW15};  C. Cibils et al. \cite{CLW} and N. Andruskiewitsch et al.  \cite{AAH19,ABFF,AP} gave examples of those whose diagrams are Nichols algebras of non-diagonal type.   We mention that  the classification of Hopf algebras of some prime dimensions has been completed in \cite{NW18}.

 Now we  study the classification of pointed Hopf algebras over $\K$ of dimension $p^2m$ with $p\nmid m$.  N. Andruskiewitsch and S. Natale \cite{AN01} classified all pointed Hopf algebras of dimension $p^2q$ in characteristic zero. The results showed that there are $4(p-1)$ isomorphism classes and all are generated by group-like elements and skew-primitive elements. As far as we know, it is still an open question to classify those in characteristic zero when $m$ is not a prime number, see e.g. \cite{BG13} for more details.

Our strategy follows  the  Lifting Method \cite{AS98b} and the Hochschild cohomology of coalgebras \cite{SO}, which are proposed to classify pointed Hopf algebras in characteristic zero of dimension $p^3$ and  subsequently used in a lot of papers. Let $H$ be a pointed Hopf algebra and $\gr H$ be the graded Hopf algebra of $H$ associated to the coradical filtration. Then $\gr H\cong R\sharp H_0$ \cite[Theorem 2]{R85}, where $R$ is a strictly graded Hopf algebra in ${}^{H_0}_{H_0}\mathcal{YD}$. The subalgebra of $R$ generated by the space $R(1)$ is the so-called Nichols algebra $\BN(R(1))$ (also called quantum shuffle algebra \cite{Rosso}), which plays a key role in the classification of pointed complex Hopf algebras under the  following
\begin{conj}\label{conj1}   \cite[Conjecture 2.7]{AS02}
Assume that the base field $\K$ has characteristic zero. Any finite-dimensional connected graded braided Hopf algebra $R=\oplus_{i\geq 0}R(i)$ in ${}^{H_0}_{H_0}\mathcal{YD}$
satisfying $\Pp(R)=R(1)$ is generated by $R(1)$.
\end{conj}
As well-known,  Conjecture \ref{conj1}   is true for abelian
groups in characteristic zero \cite{AI19} but it fails in positive characteristic, see e.g. \cite[Example 2.5]{AS02}. Indeed, there are infinitely many examples of pointed Hopf algebras over $\K$ of dimension $p^n$ whose diagrams are not Nichols algebras, see e.g. \cite{W1,NWW1,NWW2,NW, X22}. The
second Hochschild cohomology groups of coalgebras can control the coproducts of the  generators which are not group-like elements and skew-primitive elements, see e.g. \cite[Theorem 2.7]{NW}.

 In this paper, we first classify coradically graded pointed Hopf algebras of a given dimension, including classes that are not generated by group-like elements and skew-primitive elements, then compute the liftings and finally determine the isomorphism classes. The lifting procedure is highly non-trivial and  computationally challenging especially when  the diagrams $R$ are not Nichols algebras or the relations of $R(1)\sharp H_0$ admit non-trivial deformations. Our main result is

 \begin{thm}[Theorem \ref{thm:dimp2q-completeclassification}]
Let $H$ be  a pointed Hopf algebra over $\K$ of dimension $p^2q$. If $H$ is generated by group-like elements and skew-primitive elements, then $H$ is one of the algebras listed in Propositions \ref{pro:p2q-1} and \ref{pro:p2q-rank-2-liftings}; otherwise $H$ is one of the algebras listed in Proposition \ref{pro:p2q-2}.
\end{thm}

The result shows that there are  finitely many classes of pointed Hopf algebras over $\K$ of dimension $p^2q$. The Hopf algebras listed in Proposition \ref{pro:p2q-1} are rank-one pointed Hopf algebras of the second type and the third type in \cite{S}. In  Proposition \ref{pro:p2q-rank-2-liftings},  except for the classes described in  $(27)$--$(28)$, $(33)$ and $(40)$, they are Radford biproducts of restricted universal enveloping algebras of dimension $p^2$ by $\K[\Z_q]$. We mention that restricted universal enveloping algebras may be not Hopf subalgebras of $H$. Contrary to the situation in characteristic zero \cite{AN01},  there are 10 classes that are not generated by group-like elements and skew-primitive elements listed in Proposition \ref{pro:p2q-2}.
See Sec. \ref{subsec:p2q} for details.

Besides, we also give the complete classification of pointed Hopf algebras of dimension $pq$ and $pqr$ for distinct prime numbers $p,q,r$, see Corollaries \ref{cor:classification-dimPQ} and \ref{cor:classification-dimPQR}. Furthermore, we classify pointed Hopf algebras of dimension $p^2m$ with abelian coradical, where $m$ is square-free and $\Char\K=p\nmid m$.
\begin{thm}[Theorem \ref{thm:p2m-withabeliancoradical}]
Let $m$ be square-free and $\Char\K=p\nmid m$. Let $H$ be a pointed Hopf algebra of dimension $p^2m$ with abelian coradical. Then $H$ is isomorphic to one of the following Hopf algebras:
\begin{itemize}
\item[(i)] $\fA^k_{\G(H)}(g,\chi,f)$ for $k\in\I_{1,2}$ and $\fA^3_{\G(H)}(g,f)$ (see Definition \ref{defi:fA});
\item[(ii)]$\dH^k(\D)$ for $k\in\I_{1,8}$ (see Definition \ref{defi:dH});
\item[(iii)] $\fH^k_{\G(H)}(g,\chi)$ for $k\in\I_{1,4}$ (see Theorem \ref{thm:liftingsofcH}).

\end{itemize}
\end{thm}

The Hopf algebras listed in $(i)$ are rank-one pointed Hopf algebras of the second type and the third type in \cite{S}. The Hopf algebras listed in $(ii)$ are also generated by group-like elements and skew-primitive elements. Among them, $\dH^k(\D)$ for $k\in\I_{1,5}$ are Radford biproducts of  connected (braided) Hopf algebra of dimension $p^2$  by $\K[\G(H)]$, see Remark \ref{rmk:dHk-1-5-Radfordbiproduct} for details. The Hopf algebras listed in $(iii)$ are not generated by group-like elements and skew-primitive elements. Among them, $\fH^k_{\G(H)}(g,\chi)$ for $k\in\I_{1,3}$ are Radford biproducts of connected Hopf algebras of dimension $p^2$ classified in \cite[Lemma 7.3]{W1} by $\K[\G(H)]$, see Remark \ref{rmk:fHk-1-3-Radfordbiproduct} for details. To the best of our knowledge,  $\dH^k(\D)$ for $k\in\I_{6,8}$ and $\fH^4_{\G(H)}(g,\chi)$ constitute new examples of pointed Hopf algebras.

The paper is organized as follows: In Sec. \ref{secPre}, we introduce some necessary notations and concepts. In Sec. \ref{secNicholsalgebra}, we study the classification when the diagrams are Nichols algebras of dimension $p$ or $p^2$. In Sec. \ref{secNonNicholsalgebra}, we study the classification when diagrams are not Nichols algebras and have dimension $p^2$.   In Sec. \ref{sec:p2q}, we give the main results on pointed Hopf algebras of dimension $p^2m$. It seems much more hard to classify pointed Hopf algebras of dimension $p^2m$ without the assumptions. Indeed, we are working on the classification of pointed Hopf algebras of dimension $p^2q^2$. It turns out that there are examples whose diagrams are braided (not usual) Hopf algebras of dimension $pq$ that are not Nichols algebras.

\section{Preliminaries}\label{secPre}
\subsection*{Conventions}
We work over an algebraically closed field $\K$.  Denote by $\Char\K$ the characteristic of $\K$,  by $\N$ the set of natural numbers, and by $\Z_n$ the cyclic group of order $n$.  $\K^{\times}=\K-\{0\}$.   Given~$n\geq k\geq 0$,    $\I_{k,n}=\{k,k+1,\ldots,n\}$.

For a Hopf algebra $H$,  $\G(H):=\{h\in H\mid \Delta(h)=h\otimes h,\ \epsilon(h)=1\}$ is the group of \emph{group-like} elements of $H$.   For any $g\in\G(H)$, denote by~$|g|$~or~$\ord(g)$ the  order of $g$. $H$ is said to be pointed if the coradical $H_0=\K[\G(H)]$. Unless otherwise stated, ``pointed" means ``non-cosemisimple pointed" in our context.

For any $g,h\in\G(H)$, $\Pp_{g,h}(H):=\{c\in H\mid \Delta(c)=c\otimes g+h\otimes c\} $ is   the set of $(g,h)$-\emph{skew primitive elements} of $H$. In particular, $\Pp(H):=\Pp_{1,1}(H)$ is the space of primitive elements of $H$.

\subsection{$q$-binomial coefficients}
We follow \cite{R99} to introduce necessary results on $q$-binomial coefficients and refer to \cite{R99,R11} and the reference therein for more details.
For any given~$q\in\K$, $i\leq n\in\N$, denote by~$|q|$~or~$\ord(q)$ the multiplicative order. The $q$-number and ~$q$-factorial are defined by
\begin{align*}
(0)_q=0, \ (n)_q:=1+q+ \cdots+q^{n-1}, \  \text{and}\
(0)_q!=1,\  (n)_q!=(n)_q(n-1)_q!,\text{respectively}.
\end{align*}
For $0\leq i\leq n$, the $q$-binomial coefficient is defined inductively by
\begin{align*}
\left(\begin{array}{cc}0\\0\end{array}\right)_q=1,\quad \left(\begin{array}{cc}n\\i\end{array}\right)_q=q^i\left(\begin{array}{cc}n-1\\i\end{array}\right)_q+\left(\begin{array}{cc}n-1\\i-1\end{array}\right)_q.
\end{align*}
In particular, if $n\geq 1$ and $(n-1)_q!\neq 0$, then
\begin{align*}
 \left(\begin{array}{cc}n\\i\end{array}\right)_q=\frac{(n)_q!}{(i)_q!(n-i)_q!},\quad i\in\I_{1,n}.
\end{align*}

\begin{thm} \cite[Lemma 3]{R99}\label{thm:q-binomial}
Suppose that  $A$ is an algebra over $\K$ and $x,y\in A$ satisfy $xy=qyx$, where  $q\in\K^{\times}$. Then
\begin{align*}
(x+y)^n=\sum_{i=0}^n\left(\begin{array}{cc}n\\i\end{array}\right)_qx^iy^{n-i}.
\end{align*}
In particular, if~$\Char\K=0$, $\ord(q)=n$, or  $\Char\K=p>0$, $p^k\ord(q)=n$, for $k\in\N$, then
\begin{align*}
(x+y)^n=x^n+y^n.
\end{align*}
\end{thm}

\subsection{Yetter-Drinfeld modules and bosonizations}
 Suppose that the Hopf algebra $H$ has bijective antipode. Then the category  ${}^{H}_{H}\mathcal{YD}$ of Yetter-Drinfeld modules over $H$ is rigid braided monoidal, where the braiding $c$ is determined by
\begin{align}\label{equbraidingYDcat}
c:=c_{V,W}:V\otimes W\mapsto W\otimes V,\ v\otimes w\mapsto v_{(-1)}\cdot w\otimes v_{(0)},\ \forall\,v\in V, w\in W.
\end{align}
and the left dual $V\As$ is defined by
\begin{align*}
\langle h\cdot f,v\rangle=\langle f,S(h)v\rangle,\quad f_{(-1)}\langle f_{(0)},v\rangle=S^{-1}(v_{(-1)})\langle f, v_{(0)}\rangle.
\end{align*}

\begin{rmk}\label{rmk:dimV=1}
Suppose that $V:=\K\{v\}$ is an object of dimension one in $\HYD$. By definition, there is an algebra map $\chi:H\rightarrow\K$ and $g\in\G(H)$ satisfying
\begin{align*}
h_{(1)}\chi(h_{(2)})g=gh_{(2)}\chi(h_{(1)}),
\end{align*}
such that~$\delta(v)=g\otimes v$, $h\cdot v=\chi(h)v$. Moreover, $g$ lies in the center $\cZ(\G(H))$ of $\G(H)$. Such a triple $(\G(H),g,\chi)$ is called a YD-triple for convenience.
\end{rmk}

\begin{rmk}\cite[Remark 1.5]{AS02}
Denote by ${}_G^G\mathcal{YD}$  the category ${}_{\K[G]}^{\K[G]}\mathcal{YD}$ for short, where $\K[G]$ is the group algebra of $G$. For $V\in{}_G^G\mathcal{YD}$, set $V_g:=\{v\in V\mid \delta(v)=g\otimes v\}$ for any $g\in G$. Then  as a $G$-comodule, $V:=\oplus_{g\in G}V_g$ is a $G$-graded vector space. Assume in addition that the action of $G$ is diagonalizable, that is, $V=\oplus_{\chi\in \widehat{G}}V^{\chi}$, where $V^{\chi}:=\{v\in V\mid g\cdot v=\chi(g)v,\;\forall g\in G\}$.  Then
\begin{align*}
V=\oplus_{g\in G,\chi\in \widehat{G}}V_{g}^{\chi}, \text{ where }V_{g}^{\chi}=V_g\cap V^{\chi}.
\end{align*}
\end{rmk}

 Let $R$ be a braided Hopf algebra in ${}^{H}_{H}\mathcal{YD}$ and denote $\Delta_R(r)=r^{(1)}\otimes r^{(2)}$ to avoid confusions.  The \emph{bosonization or Radford biproduct} of $R$ by $H$, denoted by $R\sharp H$, is a usual Hopf algebra, whose multiplication and comultiplication are determined by the smash product and smash coproduct, respectively:
\begin{align}
(r\sharp g)(s\sharp h)&=r(g_{(1)}\cdot s)\sharp g_{(2)}h,\quad
\Delta(r\sharp g)=r^{(1)}\sharp (r^{(2)})_{(-1)}g_{(1)}\otimes (r^{(2)})_{(0)}\sharp g_{(2)}.\label{eqSmash}
\end{align}

Clearly, the inclusion $\iota:H\rightarrow R\sharp H, h\mapsto 1\sharp h,\ \forall h\in H$ and the projection
$\pi:R\sharp H\rightarrow H,r\sharp h\mapsto \epsilon_R(r)h,\ \forall r\in R, h\in H$ are both morphisms of Hopf algebras such that $\pi\circ\iota=\id_H$.
Furthermore, $R=(R\sharp H)^{coH}=\{x\in R\sharp H\mid (\id\otimes\pi)\Delta(x)=x\otimes 1\}$.

Conversely, for Hopf algebras $A$ and $H$, if there are Hopf algebra morphisms $\pi:A\rightarrow H$ and $\iota:H\rightarrow A$ such that $\pi\circ\iota=\id_H$, then $A\simeq R\sharp H$, where $R=A^{coH}$ is a braided Hopf algebra in ${}^{H}_{H}\mathcal{YD}$. See \cite[pp. 380-384]{R11} for details.

\subsection{Nichols algebras and Lifting Method}
For $V\in{}^{H}_{H}\mathcal{YD}$, the \emph{Nichols algebra} $\BN(V)$ of $V$ is a strictly   $\N$-graded Hopf algebra $\BN(V)=\oplus_{n\geq 0}\BN^n(V)$ in ${}^{H}_{H}\mathcal{YD}$ that is generated as an algebra by $V$. Namely, $\BN(V)$ is a $\N$-graded Hopf algebra $R=\oplus_{n\geq 0} R(n)$ in ${}^{H}_{H}\mathcal{YD}$ such that
\begin{align*}
R(0)=\mathds{k}, \quad R(1)=V,\quad \Pp(R)=V,\quad
R\;\text{is generated as an algebra by } V.
\end{align*}

Furthermore,  $\BN(V)$ depends only on $(V, c_{V,V})$ and  the same braided vector space  can be realized in $\HYD$ in many ways and for many $H$'s.

\begin{rmk}\label{rmk-1-1-1-1-1}
We say that a braided vector space $(V, c)$ of rank $m$ is of diagonal type, if there is a basis $\{x_i\}_{i\in\I_{1,m}}$ of $V$ and $(q_{i,j})_{i,j\in\I_{1,m}}$ such that $c(x_i\otimes x_j)=q_{ij}x_j\otimes x_i$, see \cite{AA17} for details.    In $\Char\K=p>0$,  if $V$ is of rank 1 with trivial braiding, that is, $V=\K\{x\}$ with $c(x\otimes x)=x\otimes x$, then $\BN(V)\cong\K[x]/(x^p)$.
\end{rmk}

\begin{rmk}
Let $V:=\K\{x_1,x_2\}$ be a braided vector space of diagonal type with the braiding $(q_{i,j})_{i,j\in\I_{1,2}}$ satisfying $q_{i,j}=1$. Then $\BN(V)\cong\K[x,y]/(x^p,y^p)$, where $x,y$ are primitive elements. See \cite{WH,WJrank3,WJrank4} for more details.
\end{rmk}

Following the Lifting Method \cite{AS98b},  we explain how the Nichols algebra enters into the classification program of pointed Hopf algebra.  Let $A$ be a pointed Hopf algebra  and $\{A_n\}_{n=0}^{\infty}$ the coradical filtration of $A$, in which
$A_n=A_0\bigwedge A_{n-1}=\{a\in A\mid \Delta(a)\in A_{0}\otimes A+A\otimes A_{n-1}\}$. Consider the associated graded coalgebra $\gr A$, that is, $\gr A=\oplus_{n=0}^{\infty}\gr^n A$ with $\gr ^0 A=A_0$
and $\gr^n A=A_n/A_{n-1}$. It is well-known that $\gr A$ is a Hopf algebra admitting the Hopf
algebra projection of $\gr A$ onto $A_0$, which is denoted by $\pi: \gr A\twoheadrightarrow A_0$.
Then $\pi$ splits the inclusion $i: A_0\hookrightarrow \gr A$. Therefore, $\gr A\cong R\sharp A_0$, where $R=(\gr A)^{co A_0}$ is a Hopf algebra
in ${}^{A_0}_{A_0}\mathcal{YD}$. Furthermore, $R$ has the following properties:
\begin{itemize}
  \item[(1)] $R=\oplus_{n=0}^{\infty}R(n)$ is a graded Hopf algebra in ${}^{A_0}_{A_0}\mathcal{YD}$, where $R(n)=R\cap H(n)$;
  \item[(2)] $R$ is connected, that is, $R(0)\cong\K$;
  \item[(3)] $\Pp(R)=R(1)$.
\end{itemize}

\begin{defi}
With the notations as above, the algebra $R$ and the part $R(1)$ are called the diagram and the infinitesimal braiding of $A$, respectively.
\end{defi}
Furthermore, the subalgebra of $R$ generated by
$R(1)$ is the Nichols algebra of $R(1)$, which
plays a key role in the classification of pointed complex Hopf algebra under
the following
\begin{conj}  \cite[Conjecture 2.7]{AS02}\label{conj-2}
Let $A$ be a finite-dimensional pointed complex Hopf algebra and keep the notations as above. Then the diagram $R$ is a Nichols algebra.
\end{conj}
\begin{rmk}
To the best of our knowledge, all examples of finite-dimensional Hopf algebras with $\Char\K=0$ support Conjecture \ref{conj-2}. However, it is not always true in positive characteristic, see e.g. \cite[Example 2.5]{AS02}. Furthermore, there are infinite families of pointed Hopf algebras of dimension $p^n$  whose diagrams are not Nichols algebras. See e.g. \cite{NWW1,NWW2,NW,X22} for details.
\end{rmk}

\subsection{The Hochschild cohomology}
We  introduce some basic concepts on the Hochschild cohomology of coalgebras and explain how it enters into the classification program. See \cite{SO,WZZ,W1,NW}) for more details.
\begin{defi}\cite[Sec. 1]{SO}
Let $C$ be a pointed coalgebra and
$(M,\delta_L,\delta_R)$  a $C$-bicomodule. Set $\id_n:=\id_{C^{\otimes n}}$. The \emph{Hochschild cohomology} $H^{\bullet}(M,C)$ of $C$ with coefficients in $M$ is defined by the homology of the complex $(\mathbf{C}^n(M,C), d^n)_{n\in\N}$, where $\mathbf{C}^n(M,C)=\Hom_{\K}(M,C^{\otimes n})$ and
\begin{gather*}
d^n(f)=(\id_1\otimes f)\delta_L+\sum_{i=0}^{n-1}(-1)^{i+1}(\id_{i}\otimes \Delta\otimes \id_{n-i-1})f+(-1)^{n+1}(f\otimes \id_1)\delta_R.
\end{gather*}
\end{defi}
Set $Z^n(M,C):=\Ker\partial^n$, $B^n(M,C):=\im \partial^{n-1}$ and $H^n(M,C):=Z^n(M,C)/B^n(M,C)$. Furthermore,
\begin{pro}\cite[Proposition\,1.4]{SO}\label{proDualCo}
$H^n(M,C)\cong H^n(C^{*},M^{*})$, where $M^{*}$ is endowed with a natural $C^{*}$-bimodule structure.
\end{pro}

For $g,h\in\G(C)$ and $M:={}^{g}\K^h$ with $\delta_L(k)=h\otimes k$ and $\delta_R(k)=k\otimes g$, for $k\in\K$, $H^{\bullet}({}^g\K^h,C)$ can be computed as the homology of the complex $(C^{\otimes n},d^n_{g,h})$, whose differentials are given by
\begin{gather*}
d^0_{g,h}(k)=k(g-h),\quad k\in\K,\\
d^n_{g,h}(x)=h\otimes x+\sum_{i=0}^{n-1}(-1)^{i+1}(\id_{i}\otimes \Delta\otimes \id_{n-i-1})(x)+(-1)^{n+1}x\otimes g,\quad n>0.
\end{gather*}
Set $d^n:=d^n_{1,1}$, $Z^n(\K,C):=Z^n({}^1\K^1,C)$, $ B^n(\K,C):=B^n({}^1\K^1,C)$ and $H^n(\K,C):=H^n({}^1\K^1,C)$ for short.
\begin{rmk}
The differentials in low degrees are explicitly given as follows:
\begin{gather*}
d^1_{g,h}(x)=h\otimes x-\Delta(x)+x\otimes g,\\
d^2_{g,h}(x\otimes y)=h\otimes x\otimes y-\Delta(x)\otimes y+x\otimes \Delta(y)-x\otimes y\otimes g.
\end{gather*}
\end{rmk}

There is a close connection between the Hochschild cohomology of a pointed coalgebra and its coradical filtration.
\begin{thm}\cite[Theorem\,1.2]{SO}\label{thm-SO-thm1.2}
Let $C$ be a pointed coalgebra. Then
\begin{enumerate}
  \item $\Pp_{g,h}(C)/\K(g-h)\cong H^1({}^g\K^h,C)$.
  \item $C_n^{g,h}/C_{n-1}\cong \ker[H^2({}^g\K^h,C_{n-1})\rightarrow H^2({}^g\K^h,C)]$ for any $n>1$, where  $C_n^{g,h}=\{x\in C\mid \Delta(x)=x\otimes g+h\otimes x+\omega,\;\omega\in C_{n-1}\otimes C_{n-1}\}$.
\end{enumerate}
\end{thm}

As an application, we have the following result (see also \cite[Lemma 2.3]{WZZ}).

\begin{pro}\cite[Corollary\,1.3]{SO}\label{pronongraded}
Let $C$ be a pointed coalgebra, $D$ a subcoalgebra of $C$ and $g,h\in\G(C)$.  Suppose that there is $n>0$ such that $D_n=C_n$. Then the differential $d^1_{g,h}$ induces an injective map
\begin{gather*}
d^1_{g,h}:C^{g,h}_{n+1}/D^{g,h}_{n+1}\hookrightarrow H^2({}^g\K^h,D).
\end{gather*}
In particular, if $H^2({}^g\K^h,D)=0$, then $D_{n+1}^{g,h}=C_{n+1}^{g,h}$.
\end{pro}


Consider a strictly graded pointed coalgebra $C=\sum_{i=0}^nC(n)$, then $H^{\bullet}({}^g\K^h,C)$ admits a bi-grading structure with the homological grading and the Adams grading, that is, $H^i({}^g\K^h,C)=\oplus_{j=0}^nH^{i,j}({}^g\K^h,C)$,
where $H^{i,j}({}^g\K^h,C)\subset H^{i}({}^g\K^h,C)$ consists of homogeneous elements whose Adams grading is $j$.

Now we introduce the following braided version of Proposition \ref{pronongraded}  due to \cite[Theorem 2.7]{NW18}:
\begin{thm}\label{thm:non-primitive-genetrators-Hoch}
Let $R=\oplus_{n\geq 0}R(n)$ be a strictly  graded  Hopf algebra in ${}_G^G\mathcal{YD}$ and $S$ be a graded Hopf subalgebra of $R$ in ${}_G^{G}\mathcal{YD}$. Suppose that  there exists $n>1\in\N$ satisfying $R(k)=S(k)$ for $k\in\I_{0,n}$. Then $d^1$ induces an injective map in ${}_{G}^{G}\mathcal{YD}$
\begin{gather*}
d^1:R(n+1)/S(n+1)\hookrightarrow H^{2,n+1}(\K,S).
\end{gather*}
\end{thm}
\begin{proof}
For any $r\in R(n+1)$, by assumptions,
\begin{align*}
-d^1(r)=\Delta_R(r)-r\otimes 1-1\otimes r\in&\sum_{k=1}^{n}R(k)\otimes R(n+1-k)\\&=\sum_{k=1}^nS(k)\otimes S(n+1-k)\subset S\otimes S.
\end{align*}
Furthermore, from the coassociativity of $\Delta_R$, we have $d^2|_S(d^1(r))=0$ and hence $d^1(r)\in Z^2(\K,S)$. Consequently, $d^1$ is well-defined in ${}_G^G\mathcal{YD}$.

Now assume that $d^1(r)=0$ in $H^2(\K,S)$. Then there exists $s\in S$ such that $d^1(r)=d^1(s)$, that is, $\Delta(r-s)=(r-s)\otimes 1+1\otimes (r-s)$. Hence $r-s\in\Pp(R)=R(1)=S(1)\subset S$ and so $r\in S(n+1)$. This completes the proof.
\end{proof}

\subsection{Useful results in positive characteristic}Now we introduce some useful results, which shall be used hereafter. Denote $(\AdL x)(y) := [x, y]$ and $(x)(\AdR y)=[x, y]$.
\begin{pro}\cite[pp.186--187]{J}\label{proJ}
Let $A$ be an associative algebra with $\Char\K=p>0$ and $a, b\in A$. Then
\begin{gather*}
(\AdL a)^p(b)=[a^p,b],\quad
(\AdL a)^{p-1}(b)=\sum_{i=0}^{p-1}a^{i}ba^{p-1-i};\\
(a)(\AdR b)^{p}=[a,b^p],\quad
(a)(\AdR b)^{p-1}=\sum_{i=0}^{p-1}b^{p-1-i}ab^i.
\end{gather*}
Furthermore,
\begin{align*}
(a + b)^p=a^p+b^p+\sum_{k=1}^{p-1}s_k(a,b),
\end{align*}
where $ks_k(a,b)$ is the coefficient of $\lambda^{k-1}$ in $(a)(\AdR \lambda a+b)^{p-1}$, $\lambda$ an indeterminate.
\end{pro}
\begin{rmk}
$s_1(a,b)=(a)(\AdR b)^{p-1}=\sum_{i=0}^{p-1}b^iab^{p-1-i}.$
\end{rmk}

\begin{lem}\label{lemHHHH}
Let $H$ be a pointed Hopf algebra over $\K$, $x\in H^{g,h}$,and $y\in H^{1,k}$ for some $g,h,k\in\G(H)$. Assume that $g,h,k,x$ generate a commutative Hopf subalgebra $A$ of $H$ and $d^1_{1,k}(y)\in A$. Then
\begin{gather*}
[d^1_{g,h}(x), \Delta(y)]=d^1_{g,hk}([x,y])+[h,y]\otimes x+xk\otimes [g,y].
\end{gather*}
\end{lem}
\begin{proof}
It follows by a direct computation that
\begin{align*}
[d^1_{g,h}(x),\Delta(y)]&=[x\otimes g+h\otimes x-\Delta(x),\Delta(y)]\\
&=-\Delta([x,y])+[x\otimes g,\Delta(y)]+[h\otimes x,\Delta(y)]\\
&=-\Delta([x,y])+[h,y]\otimes x+[x,y]\otimes g+hk\otimes [x,y]+xk\otimes [g,y]\\
&=d^1_{g,hk}([x,y])+[h,y]\otimes x+xk\otimes [g,y].
\end{align*}
\end{proof}

The following lemma is due to \cite[Corollary 4.10]{S} and \cite{WW} (also \cite[Lemma\,5.1(1)]{NW}).
\begin{lem}\label{pqlem1}
Let $A$ be an associative algebra over $\K$ with generators $g$, $x$, subject to  the relations $g^n=1, gx-xg=  g(1-g)$. Assume that $\Char\K=p>0$. Then
\begin{description}
  \item[(1)] $g^ix=xg^i+ig^{i}-ig^{i+1}$. In particular, $g^px=xg^p$.
  \item[(2)] $(g)(\AdR x)^{p-1}=g-g^p$, $(g)(\AdR x)^{p}=[g,x]$.
  \item[(3)] $(\AdL x)^{p-1}(g)=g-g^p$, $[x^p, g]=(\AdL x)^p(g)=[x,g]$.
\end{description}
\end{lem}
\begin{proof}
\begin{description}
  \item[(1)] Note that $(g)(\AdR x)=[g,x]=g-g^2$. Then $(\AdR x):\K\langle g\rangle\rightarrow \K\langle g\rangle$ is a derivation. It follows that $[g^i,x]=ig^{i-1}[g,x]=ig^{i}-ig^{i+1}$.
  \item[(2)]This is \cite[Lemma\,5.1(1)]{NW} when $n=p^k$ for $k>0$. As stated in the proof of \cite[Lemma\,4.0.1(1)]{NW}, the linear map $(\AdR x):\K\langle g\rangle\rightarrow \K\langle g\rangle$ is diagonalizable with eigenvalues $0,1,\ldots,n-1$. Then by Fermat's little Theorem, $(\AdR x)^p=\AdR x$. Moreover, after a direct computation, $(g)(\AdR x)^{p-1}=g-g^p$.
  \item[(3)]The proof follows the same line of $(2)$.
\end{description}
\end{proof}

The following result  is useful to determine when a coalgebra   map is one-one.
\begin{pro}\cite[Proposition 4.3.3]{R11}\label{pro:R11-4.3.3}
Let $C, D$ be coalgebras over $\K$ and $f: C \rightarrow D$ is a coalgebra map. Assume that $C$ is pointed. Then the following are equivalent:
\begin{itemize}
  \item[(a)] $f$~is one-one.
  \item[(b)] For any~$g,h\in\G(C)$, $f|_{\Pp_{g,h}(C)}$~is one-one.
  \item[(c)] $f|_{C_1}$~is one-one.
\end{itemize}
\end{pro}

\section{The diagrams are Nichols algebras}\label{secNicholsalgebra}
Let $\Char\K=p>0$ and $p\nmid m$. In this section, we study the classification of pointed Hopf algebras whose diagrams are Nichols algebras of dimension $p$ or $p^2$ to obtain our main results. In particular, we give the complete classification of pointed Hopf algebras of dimension $pq$ and $pqr$ for distinct prime numbers $p,q,r$.

\subsection{The diagram has dimension $p$}We determine finite-dimensional pointed Hopf algebras  whose diagrams have dimension $p$, which were essentially classified in \cite{S} in a different way. As a byproduct,  we give the complete classification of pointed Hopf algebras of dimension $pq$ and $pqr$.
\begin{defi}\label{defi:fA}
Let $(G,g,\chi)$ be a YD triple such that $\chi(g)=1$,
 where $G$ is a group of order $n$ with generators $h_1,\cdots,h_t$ and a fixed  set of defining relations $\cR_t$. Let $f$ be a map from $G$ to $\K$ such that $f(hk)=\chi(k)f(h)+f(k)$ for any $h,k\in G$, with condition: $f$ is the zero map if $g=1$. Let
\begin{itemize}
  \item[(1)] $\fA_G^1(g,\chi,f):=\K\langle h_1,\cdots,h_t,x\rangle/(\cR_t, h_1x-\chi(h_1)xh_1-f(h_1)(1-g),\cdots,h_tx-\chi(h_t)xh_t-f(h_t)(1-g),x^p)$, if $f(h)(1-g^p)=0$ for any $h\in G$ and $f(g)=0$;
  \item[(2)] $\fA_G^2(g,\chi,f):=\K\langle h_1,\cdots,h_t,x\rangle/(\cR_t, h_1x-\chi(h_1)xh_1-f(h_1)(1-g),\cdots,h_tx-\chi(h_t)xh_t-f(h_t)(1-g),x^p-x)$, if $\chi^{p-1}=\epsilon$, $g^{p-1}=1$, $f(g)=0$ and $f(h)^p=f(h)$ for any $h\in G$;
  \item[(3)] $\fA_G^3(g,f):=\K\langle h_1,\cdots,h_t,x\rangle/(\cR_t, h_1x-xh_1-f(h_1)h_1(1-g),\cdots, h_tx-xh_t-f(h_t)h_t(1-g),x^p-f(g)x)$, if $p\mid \ord(g)$, $\chi=\epsilon$, $f(g)= 1$ and $f(h)^p=f(h)$ for any $h\in\ G$.
\end{itemize}

\end{defi}

 They admit a Hopf algebra structure given by $\Delta(h)=h\otimes h$ for $h\in G$ and $\Delta(x)=x\otimes 1+g\otimes x$.

\begin{rmk}\label{rmk:important}
\begin{itemize}
\item[(1)]For any $n\in\N$ and $h\in G$, $f(h^n)=\sum_{i=0}^{n-1}\chi(h)^if(h)$. From which, $f(h^q)=0$ when $\chi(h)$ is a primitive $q$th root of unity. Furthermore,
\begin{align*}
h^nx=\chi(h^n)xh^n+f(h^n)h^n(1-g)=\chi(h)^nxh^n+\sum_{i=0}^{n-1}\chi(h)^if(h)h^n(1-g).
\end{align*}
If  $\chi(h)=1$ and $p\nmid\ord(h)$, then $f(h)=0$.  In particular, the case $f(g)\neq0$ occurs only when $g\neq 1$ and $p\mid\ord(g)$.
\item[(2)]If $\chi=\epsilon$, then $f$ is a morphism of groups from $G$ to $(\K,+)$. Furthermore, if $\chi=\epsilon$ and $p\nmid\ord(G)$, then $f$ must be the zero map (we write $f=0$ for short).
  \item [(3)]Let $\fA$ be one of the Hopf algebras in Definition \ref{defi:fA}. If $g=1$, then $f=0$; otherwise  there is an exact sequence of $\K[G]$-modules:
  \begin{align*}
  \K\{1-g\}\hookrightarrow\Pp_{1,g}(\fA)\twoheadrightarrow\K[x].
  \end{align*}
  For any fixed $h\in G$, if $p\nmid\ord(h)$, then as $\langle h\rangle$-modules, the exact sequence is split and hence we can choose $f(h)=0$. Furthermore, we can take $f=0$ when $p\nmid\ord(G)$.

  \item [(4)] The Hopf algebras $\fA_{G}^i(g,\chi,f)$ for $i\in\I_{1,2}$ and $\fA_G^3(g,f)$ are rank-one pointed Hopf algebras of the second type and third type in \cite{S}, respectively.
  \item [(5)]In what follows, denote $\fA_G^i(g,\chi):=\fA_G^i(g,\chi,0)$ with $i\in\I_{1,2}$ for short.
\end{itemize}
\end{rmk}
\begin{rmk}\label{rmk:fA-dual}
\begin{itemize}
\item[(i)]
It is clear that $\{x^{i}h,\ h\in G,i\in\I_{0,p-1}\}$ is a basis of $\fA_G^k(g,\chi)$ for $k\in\I_{1,2}$. Then one can check easily that $\pi:\fA_G^k(g,\chi)\rightarrow \K[G],\ x^ih\rightarrow h$  is a bialgebra map admitting a bialgebra
section $\iota:\K[G]\rightarrow \fA_G^k(g,\chi)$ such that $\pi\circ\iota=\id$. Therefore, $\fA^k(g,\chi)\cong R\sharp\K[G]$ for $k\in\I_{1,2}$, where $R\cong\K[x]/(x^p)$ or $\K[x]/(x^p-x)$.

\item[(ii)] Here $\K[x]/(x^p)$ and $\K[x]/(x^p-x)$ are usual Hopf algebras of dimension $p$ appeared in \cite{W1}. Furthermore, they are dual Hopf algebras of $\K[T]/(T^p)$ and $\K[X]/(X^p-1)$, respectively, where $\Delta(T)=T\otimes 1+1\otimes T$ and $\Delta(X)=X\otimes X$ (See \cite[Corollary 7.2]{W1}). In particular, up to isomorphism, they are regarded as Hopf subalgebras of the algebra of distributions on $G_a$ and $G_m$, respectively
    (see \cite[7.8 and 7.10]{Ja}).

\item[(iii)]For $k\in\I_{1,2}$, by \cite[2.2]{AG99}, $\fA_G^k(\chi,g)\As\cong R\As\sharp(\K[G])\As$, where $R\As\cong\K[T]/(T^p)$ or $\K[X]/(X^p-X)$ with $T\in\Pp(R\As)$ and $X\in\G(R\As)$, respectively.

\end{itemize}
\end{rmk}

\begin{pro}
$\gr\fA_G^i(g,\chi,f)\cong\K[x]/(x^p)\sharp\K[G]$ for $i\in\I_{1,2}$ and $\gr\fA_G^3(g,f)\cong \K[x]/(x^p)\sharp\K[G]$.
\end{pro}
\begin{proof}
Let $\fA$ be one of the Hopf algebras in Definition \ref{defi:fA}. It is easy to check that $\fA_0=\K[G]$ and $\fA_i=\fA_{i-1}+\K[G]\{x^i\}$ for $i\in\I_{1,p-1}$. Therefore, $\gr\fA\cong R\sharp\K[G]$, where $R\cong\K[x]/(x^p)$.
\end{proof}

\begin{pro}\label{pro:isomorohism-rank-1}
\begin{itemize}
\item[(1)]$\fA_G^1(g,\chi,f)\cong\fA_G^1(g',\chi',f')$  if and only if there exists $F\in\Aut(G)$ such that $F(g)=g'$,  $\chi\cdot F^{-1}=\chi'$ and $\alpha f'F-f+\beta(\chi-\epsilon)=0$ for some $\alpha\in\K^{\times},\beta\in\K$ satisfying $\beta(1-g^p)=0$.
\item[(2)]$\fA_G^2(g,\chi,f)\cong\fA_G^2(g',\chi',f')$   if and only if there exists $F\in\Aut(G)$ such that $F(g)=g'$,  $\chi\cdot F^{-1}=\chi'$ and $\alpha f'F-f+\beta(\chi-\epsilon)=0$ for some $\alpha\in\K^{\times},\beta\in\K$ satisfying $\alpha^p=\alpha$ and $(\beta^p-\beta)(1-g)=0$.
\item[(3)]$\fA_G^3(g,f)\cong\fA_G^3(g',f')$ if and only if there exists $F\in\Aut(G)$ such that $F(g)=g'$ and $f\cdot F^{-1}=f'$.
\end{itemize}
\end{pro}
\begin{proof}
\begin{itemize}
  \item [(1)-(2):] Suppose that  there is a Hopf algebra isomorphism  $\phi:\fA_G^i(g,\chi,f)\rightarrow \fA_G^i(g',\chi',f')$, then by Proposition \ref{pro:R11-4.3.3}, $\phi|_{G}\in\Aut(G)$ and   $\phi(\Pp_{1,g}(\fA_G^i(g,\chi)))=\Pp_{1,g'}(\fA_G^i(g',\chi'))$. Hence we assume that $\phi(g)=g'$  and $\phi(x)=\alpha x'+\beta(1-g')$ for $\alpha\neq 0,\beta\in\K$. If $g'=1$, then we assume that $\beta=0$. Since $\phi(hxh^{-1}-\chi(h)x-f(h)(1-g))=0$, it follows that $\chi'\cdot\phi=\chi$ and $[\alpha f'\phi-f+\beta(\chi-\epsilon)](h)(1-g')=0$. Applying $\phi$ to the relation $x^p=(i-1)x$, we have $(i-1)(\alpha^p-\alpha)=0$ and $\beta^p(1-(g')^p)=(i-1)\beta(1-g')$.

      If $g'=1$, then $g=1$,  $f=0=f'$ and $(i-1)(\alpha^p-\alpha)=0$. If $g'\neq 1$ and $(g')^{p}\neq g'$, then $i=1$ and $\alpha f'\phi-f+\beta(\chi-\epsilon)=0$ satisfying $\beta^p(1-(g')^p)=0$. If $g'\neq 1$ and $(g')^{p}= g'$, then $\alpha f'F-f+\beta(\chi-\epsilon)=0$ for some $\alpha\in\K^{\times},\beta\in\K$ satisfying $(i-1)(\alpha^p-\alpha)=0$ and $\beta^p=(i-1)\beta$.

      Conversely, let $\psi$ be the algebra morphism determined by $\psi(h)=F(h),\psi(x)=\alpha x'+\beta(1-g')$, then it is clear  that $\psi$ is a Hopf algebra isomorphism.

  \item [(3):] It follows directly by \cite[Lemma 4.17]{S}.
\end{itemize}
\end{proof}
\begin{pro}\label{pro:non-isomorphism-rank-one}
As Hopf algebras, $\fA_G^1(g,\chi,f)$, $\fA_G^2(g,\chi,f)$ and  $\fA_G^3(g,f)$ are pairwise non-isomorphic.
\end{pro}
\begin{proof}
We first claim that $\fA_G^1(g,\chi,f)\not\cong\fA_G^2(g,\chi,f)$. Indeed, if  there is a Hopf algebra isomorphism  $\phi:\fA_G^1(g,\chi,f)\rightarrow\fA_G^2(g,\chi,f)$, then $\phi|_{G}\in\Aut(G)$ and   $\phi(\Pp_{1,g}(\fA_G^1(g,\chi,f)))=\Pp_{1,g'}(\fA_G^2(g',\chi',f'))$, which implies that $\phi(g)=g'$ and $\phi(x)=\alpha x'+\beta(1-g')$ for $\alpha\neq 0$. On the other hand, we have $\alpha=0$ when applying $\phi$ to the relations $hx-\chi(h)x=f(h)h(1-g)$ and $x^p=0$ in $\fA_G^1(g,\chi,f)$, a contradiction. Consequently, the claim follows.

Similarly,  $\fA_G^k(g,\chi,f)\not\cong\fA_G^3(g,f)$ for $k\in\I_{1,2}$.
\end{proof}

Now we introduce the following result, which is essentially appeared in \cite{S} in a different way.
\begin{thm}\label{thm:classification-rank-one}
Let $H$ be a finite-dimensional pointed Hopf algebra whose diagram has dimension $p$. Then there exists a tuple $(\G(H),g,\chi,f)$  such that $H$ is isomorphic to $\fA_{\G(H)}^1(g,\chi,f)$, $\fA_{\G(H)}^2(g,\chi,f)$ or $\fA_{\G(H)}^3(g,f)$.
\end{thm}
\begin{proof}
Since $x\in\Pp_{1,g}(H)$, it follows that $hx-\chi(h)xh\in\Pp_{h,hg}(H)\cap H_0$.  If $g=1$, then $\Pp_{h,h}(H)\cap H_0=0$; otherwise $\Pp_{h,hg}(H)\cap H_0=\K\{h(1-g)\}$.  Hence $hx-\chi(h)xh=f(h)h(1-g)$, where $f$ is a map from $\G(H)$ to $\K$ with condition: $f=0$ when $g=1$. Then consider the conjugation of $\K[G]$, we have $f(hk)=\chi(k)f(h)+f(k)$. If $f(g)\neq 0$ and $g\neq 1$, then applying $f$ to the relation $gk=kg$ for any $k\in\G(H)$, we have  $\chi(g)f(k)+f(g)=\chi(k)f(g)+f(k)$, which implies that $\chi=\epsilon$ and thereby $f(hk)=f(h)+f(k)$.

Observe that $gx-xg=f(g)g(1-g)$. Then by Proposition \ref{proJ} and Lemma \ref{pqlem1},
\begin{align*}
\Delta(x^p)&=(x\otimes 1+g\otimes x)^p
      =x^p\otimes 1+g^p\otimes x^p+(g)(\AdR x)^{p-1}\otimes x\\
      &=x^p\otimes 1+g^p\otimes x^p+f(g)^{p-1}(g-g^p)\otimes x,\\
\Delta(x^p-f(g)^{p-1}x)&=(x^p-f(g)^{p-1}x)\otimes 1+g^p\otimes (x^p-f(g)^{p-1}x),
\end{align*}
which implies that $x^p-f(g)^{p-1}x\in\Pp_{1,g^p}(H)\cap H_{p-1}$.

Assume that $g^p=g$. Then  $x^p=\lambda_2x+\lambda_3(1-g^p)$ for $\lambda_2,\lambda_3\in\K$ and $f(g)=0$. By induction, $hx^n=[\chi(h)x+f(h)(1-g)]^nh$ for $n\in\N$ and $h\in\G(H)$. In particular,
\begin{align*}
hx^p&=[\chi(h)x+f(h)(1-g)]^ph=[\chi(h)x-f(h)g)]^ph+f(h)^ph\\
&=\chi(h)^px^ph-f(h)^pg^ph+f(h)^ph=\chi(h)^px^ph+f(h)^p(1-g^p)h.
\end{align*}

The verification of $h(x^p)=(hx)x^{p-1}$ imposes the conditions:
\begin{align*}
\lambda_2(\chi(h)^p-\chi(h))=0,\quad \lambda_3\chi(h)^p+f(h)^p=\lambda_2f(h)+\lambda_3.
\end{align*}

Let $\fA(g,\chi,f,\lambda_2,\lambda_3)$ be the Hopf algebra described as above. Then consider the translation $x\mapsto x+a(1-g)$ satisfying $a^p-\lambda_2a=\lambda_3$, we have $\fA(g,\chi,f,\lambda_2,\lambda_3)\cong \fA(g,\chi,f+a(\epsilon-\chi),\lambda_2,0)$.
If $\lambda_2=0$, then clearly $f+a(\epsilon-\chi)=0$ and so $\fA(g,\chi,f,\lambda_2,\lambda_3)\cong\fA^1_{\G(H)}(g,\chi)$. If $\lambda_2\neq 0$, then $\chi^{p-1}=\epsilon$ and by the translation $x\mapsto b^{-1}x$ satisfying $b^{p-1}=\lambda_2$, we have
$$\fA(g,\chi,f,\lambda_2,\lambda_3)\cong \fA(g,\chi,[f+a(\epsilon-\chi)]b^{-1},1,0)\cong\fA^2_{\G(H)}(g,\chi,[f+a(\epsilon-\chi)]b^{-1}).$$

Assume that $g^p\neq g$. Then  $x^p-f(g)^{p-1}x=\lambda_4(1-g^p)$. Furthermore,
\begin{align*}
hx^p&=[\chi(h)x+f(h)(1-g)]^ph=[\chi(h)x-f(h)g)]^ph+f(h)^ph\\
&=\chi(h)^px^ph-f(h)^pg^ph+(-f(h)g)(\ad_R\chi(h)x)^{p-1}h+f(h)^ph\\
&=\chi(h)^px^ph-f(h)^pg^ph-f(h)\chi(h)^{p-1}(g)(\ad_R x)^{p-1}h+f(h)^ph\\
&=\chi(h)^px^ph-f(h)^pg^ph-f(h)\chi(h)^{p-1}f(g)^{p-1}(g-g^p)h+f(h)^ph\\
&=\chi(h)^px^ph+f(h)^p(1-g^p)h-f(h)\chi(h)^{p-1}f(g)^{p-1}(g-g^p)h\\
&=\chi(h)^pf(g)^{p-1}xh+[\lambda_4\chi(h)^p+f(h)^p](1-g^p)h-f(h)\chi(h)^{p-1}f(g)^{p-1}(g-g^p)h;\\
h(x^p)&=f(g)^{p-1}hx+\lambda_4(1-g^p)h\\
&=\chi(h)f(g)^{p-1}xh+f(g)^{p-1}f(h)(1-g)h+\lambda_4(1-g^p)h.
\end{align*}

If $f(g)=0$, then from $(hx)x^{p-1}=h(x^p)$, we have $[\lambda_4(\chi^p-\epsilon)(h)+f(h)^p](1-g^p)=0$; otherwise $\chi=\epsilon$ and then we have
\begin{align*}
f(g)^{p-1}f(h)=f(h)^p.
\end{align*}

Similar to the last case, we may choose $\lambda_4=0$ by the  translation $x\mapsto x+a(1-g)$ satisfying $a^p-f(g)^{p-1}a=\lambda_4$.

If $f(g)=0$, then $f(h)(1-g^p)=0$ and $H\cong\fA^1_{\G(H)}(g,\chi,f)$; otherwise we have $\chi=\epsilon$, $f(g)^{p-1}f(h)=f(h)^p$ and then by the translation $x\mapsto f(g)^{-1}x$, we can take $f(g)=1$ and thereby $H\cong\fA_{\G(H)}^3(g,f)$.
\end{proof}

\begin{cor}\label{cor:classification-dimPQ}
Set $\Z_q:=\langle g\rangle$ and $\widehat{\Z_q}:=\langle\chi\rangle$. Let $H$ be a pointed Hopf algebra of dimension $pq$ for a prime number $q$. Then $H$ is isomorphic to $\fA_{\Z_q}^1(1,\epsilon)$, $\fA_{\Z_q}^2(1,\epsilon)$, $\fA_{\Z_q}^1(g,\epsilon)$, $\fA_{\Z_q}^2(g,\epsilon)$, $\fA_{\Z_q}^1(1,\chi)$ or $\fA_{\Z_q}^2(1,\chi)$.
\end{cor}
\begin{proof}

Let $R$ be the diagram of $H$ and $V:=R(1)$. We claim that $\G(H)\cong \Z_q$.  By Nichols-Zoeller Theorem \cite{NZ},  $\dim H_0 |\dim H$, which implies that $\dim H_0= p,q$. If $\dim H_0=p$, that is, $H_0\cong\K[\Z_p]$, then there must be an element $x\in V_g^{\epsilon}$ for some $g\in\G(H)$ with trivial braiding. Hence $\K[x]/(x^p)$ is a braided Hopf subalgebra of $R$. Then by \cite{G} or \cite[Proposition 2.16]{Scha01}, $\dim \K[x]/(x^p)=p$ must divide $\dim R=q$, a contradiction.
Consequently, $\dim H_0=q$, that is, $\G(H)\cong \Z_q$.

 Since $\G(H)\cong \Z_q:=\langle g\rangle$, it follows that $\dim R=p$. Then by Theorem \ref{thm:classification-rank-one}, $H\cong\fA^k_{\Z_q}(g^i,\chi^j)$ for $k\in\I_{1,2}$ and some $i,j\in\I_{0,q-1}$ such that  $\xi^{ij}=1$, where $\chi(g)=\xi$ is a primitive $q$th root of unity.  Furthermore,  $\xi^{ij}=1$ yields $i=0$ or $j=0$. If $i=0$, then up to change the character $\chi$, we can take $j\in\I_{0,1}$. If $j=0$, then   we can take $i\in\I_{0,1}$. This completes the proof.
\end{proof}

\begin{cor}\label{cor:classification-dimPQR}
Set $\Z_{qr}:=\langle g\rangle$, $\widehat{\Z_{qr}}:=\langle\chi\rangle$ and $\widehat{\Z_q\rtimes \Z_r}:=\langle\tau\rangle$. Let $H$ be a pointed Hopf algebra of dimension $pqr$ for distinct prime numbers $p,q,r$. Then $H$ is isomorphic to $\fA_{\Z_{qr}}^1(1,\chi^j)$, $\fA_{\Z_{qr}}^2(1,\chi^j)$ for $j\in\{0,1,q,r\}$, $\fA_{\Z_{qr}}^1(g^i,\epsilon)$, $\fA_{\Z_{qr}}^2(g^i,\epsilon)$ for $i\in\{1,q,r\}$,  $\fA_{\Z_q\rtimes \Z_r}^1(1,\tau)$ or $\fA_{\Z_q\rtimes \Z_r}^2(1,\tau)$.
\end{cor}
\begin{proof}
Let $R$ be the diagram of $H$ and $V:=R(1)$.
We claim that $\dim H_0=qr$. Indeed, by Nichols-Zoeller Theorem \cite{NZ},  $\dim H_0 |\dim H$, which implies that $\dim H_0=p,q,r,pq,qr,pr$. If $\dim H_0=pq$, then $\widehat{\G(H)}\cong \Z_{pq}:=\langle\zeta\rangle$ and $\dim R=r$. Hence $\dim R(1)=1$ with a basis $\{x\}$.   Furthermore, $x\in V_g^{\zeta^i},i\in\I_{0,q-1}$ with $g\in\G(H)$, which implies that $c(x\otimes x)=\zeta^i(g) x\otimes x$. Therefore, $R$ contains a braided Hopf subalgebra of dimension $p$ or $q$, a contradiction.   Hence $\dim H_0\neq pq$. Similarly, we have $\dim H_0\neq pr,qr,p,q,r$. Therefore, the claim follows.

Assume that $\G(H)=\Z_{qr}=\langle g\rangle$. Then $\dim R=p$ and hence   $\dim V=1$ with trivial braiding. Let $V:=\K\{x\}$ and $\widehat{\G(H)}=\langle\chi\rangle$, where $\chi(g)=\theta$ is a primitive $qr$th root of unity.  Then by Remark \ref{rmk:dimV=1}, $x\in V_{g^i}^{\chi^j}$ for some $i,j\in\I_{0,qr-1}$ such that $\theta^{ij}=1$, which implies that $qr\mid ij$. Hence $i=0$ or $j=0$. If $i=0$, then by Proposition \ref{pro:isomorohism-rank-1}, we can take $j\in\{0,1,q,r\}$. If $j=0$, then we can take $i\in\{0,1,q,r\}$.

Assume that $\G(H)=\Z_q\rtimes \Z_r$. Then $\dim R=p$ and hence $\dim V=1$ with trivial braiding. Let $V:=\K\{x\}$ and $\widehat{\G(H)}=\langle \tau\rangle$, where $\tau(g)=\xi$ is a primitive $q$th root of unity, and $\tau(h)=1$.  Since the center of $\G(H)$ is trivial,  by Remark \ref{rmk:dimV=1}, $x\in V_{1}^{\tau^i}$ for some $i\in\I_{0,q-1}$. By Proposition \ref{pro:isomorohism-rank-1},  we can take $i\in\I_{0,1}$.

Therefore the corollary follows by Theorem \ref{thm:classification-rank-one}.
\end{proof}

\begin{cor}\cite[Theorem 3.3]{WW}
Let $H$ be a non-connected pointed Hopf algebra of dimension $p^2$ and $g$ a generator of $\Z_p$. Then $H$ is isomorphic to $\fA_{\Z_p}^1(1,\epsilon)$, $\fA_{\Z_p}^2(1,\epsilon)$, $\fA_{\Z_p}^1(g,\epsilon)$ or $\fA_{\Z_p}^3(g,f)$.
\end{cor}
\begin{proof}
Let $R$ be the diagram of $H$. By assumptions, $\dim H_0=p$ and hence $\G(H)\cong\Z_p$ and $R\cong\K[x]/(x^p)$ with Yetter-Drinfeld module structure given by $g\cdot x=x, \delta(x)=g^i\otimes x$ for $i\in\I_{0,p-1}$. By changing the generator,  we may choose $i\in\I_{0,1}$.

Observe that $f=0$ when $f(g)=0$. If $i=0$, then by Theorem \ref{thm:classification-rank-one}, $H$ is isomorphic to $\fA_{\Z_p}^1(1,\epsilon)$ or $\fA_{\Z_p}^2(1,\epsilon)$; otherwise $H$ is isomorphic to $\fA_{\Z_p}^1(g,\epsilon)$ or $\fA_{\Z_p}^3(g,f)$.
\end{proof}
\subsection{Liftings of quantum
planes } We determine liftings of quantum
planes of dimension $p^2$.
\begin{defi}
A QPYD-datum $\D(G,\chi_1,\chi_2,g_1,g_2)$ consists of a group $G$ of order $m$,  characters $\chi_1,\chi_2\in\widehat{G}$ and $ g_1,g_2\in\mathcal{Z}(G)$ satisfying $\chi_1(g_1)=1=\chi_2(g_2)=\chi_1(g_2)\chi_2(g_1)$. We write $\D:=\D(G,\chi_1,\chi_2,g_1,g_2)$ for short.
\end{defi}
\begin{lem}
Let $V:=\K\{x,y\}$ and $\D$ a QPYD-datum. Then $(V,\D)\in{}_G^G\mathcal{YD}$ with $x\in V_{g_1}^{\chi_1}$, $y\in V_{g_2}^{\chi_2}$. Furthermore, $\dim\BN(V,\D)=p^2$.
\end{lem}
\begin{proof}
It follows by the fact that $(V,\D)$ is a quantum plane.
\end{proof}


Now we determine liftings of $\BN(V,\D)\sharp\K[G]$.
For a QPYD datum $\D:=\D(G,\chi_1,\chi_2,g_1,g_2)$,  we set $q_{ij}:=\chi_j(g_i)$ and $[x,y]_c=xy-q_{12}yx$ for short.
\begin{defi}\label{defi:dH}
For $k\in\I_{1,8}$,
\begin{align*}
\dH^k(\D):=T(V)\sharp\K[G]/\cI_k,
\end{align*}
where $\cI_k$ is the ideal given as follows:
\begin{itemize}
   \item $\cI_1:=([x,y]_c,x^p,y^p)$;
   \item $\cI_2:=([x,y]_c,x^p-x,y^p)$, if  $\chi_1^{p-1}=\epsilon$, $g_1^{p-1}=1$;
 \item  $\cI_3:=([x,y],x^p-y,y^p)$, if $\chi_1^{p}=\chi_2$, $g_1^p=g_2$;
   \item $\cI_4:=([x,y]_c, x^p-x,y^p-y)$, if $\chi_1^{p-1}=\epsilon=\chi_2^{p-1}$, $g_1^{p-1}=1=g_2^{p-1}$;
   \item $\cI_5:=( [x,y]-y, x^p-x,y^p)$, if $\chi_1=\epsilon$, $g_1=1$;
    \item $\cI_6:=([x,y]-(1-g_1g_2), x^p,y^p)$, if $\chi_1\chi_2=\epsilon$, $g_1g_2\neq 1$;
    \item $\cI_7:=([x,y]-y-(1-g_2), x^p-x,y^p)$, if  $\chi_1=\epsilon=\chi_2$, $g_1=1$, $g_2\neq 1$;
    \item  $\cI_8:=([x,y], x^p-y,y^p-x)$, if $\chi_1^{p}=\chi_2$, $\chi_2^p=\chi_1$, $g_1^p=g_2$, $g_2^p=g_1$ and $g_1\neq g_2$ or $\chi_1^{p}=\chi_2$, $\chi_2^p=\chi_1$, $g_1=g_2$, $g_1^{p-1}=1$ and $\chi_1\neq \chi_2$.
 \end{itemize}
 If no confusions, we also denote $\dH^k(\chi_1,\chi_2,g_1,g_2):=\dH^k(\D)$ for $k\in\I_{1,8}$.
\end{defi}

\begin{lem}
For $k\in\I_{1,8}$, $\dH^k(\D)$ is a Hopf algebra  with
\begin{align*}
\Delta(x)=x\otimes 1+g_1\otimes x,\quad \Delta(y)=y\otimes 1+g_2\otimes y.
\end{align*}
\end{lem}
\begin{proof}
It follows by direct computations that $\Delta(\cI_k)\subset \dH^k(\D)\otimes\cI_k+\cI_k\otimes\dH^k(\D)$ and $\epsilon(\cI_k)=0$.
\end{proof}

\begin{lem}
For $k\in\I_{1,8}$, $\dim\dH^k(\D)=p^2m$.
\end{lem}
\begin{proof}
 Applying the Diamond Lemma \cite{B}, it suffices to verify the following overlaps:
\begin{gather}
g^{m_g-1}(gx)=(g^{m_g-1}g)x,\quad g^{m_g-1}(gy)=(g^{m_g-1}g)y ,\label{eq:p2q-11-1}\\
g(xx^{p-1})=(gx)x^{p-1},\quad g(yy^{p-1})=(gy)y^{p-1},\label{eq:p2q-11-2}\\
g(xy)=(gx)y,\quad (gh)x=g(hx)\label{eq:p2q-11-3}\\
x(yy^{p-1})=(xy)y^{p-1},\quad x^{p-1}(xy)=(x^{p-1}x)y,\label{eq:p2q-11-4}\\
(xx^{p-1})x=x(x^{p-1}x),\quad (yy^{p-1})y=y(y^{p-1}y),\label{eq:p2q-11-5}
\end{gather}
are resolvable with the order $y<x<g$, where $g\in G$ with order $m_g$. We omit the details to save spaces, since it is tedious but straightforward.
\end{proof}
\begin{rmk}\label{rmk:dHk-1-5-Radfordbiproduct}
 Observe that $\{y^ix^jh,\ h\in G,i,j\in\I_{0,p-1}\}$ is a basis of $\dH^k(\D)$. For $k\in\I_{1,8}-\{6,7\}$, it is easy to see that the projection $\pi:\dH^k(\D)\rightarrow \K[G],\ y^jx^ih\rightarrow h$  is a bialgebra map admitting a bialgebra
section $\iota:\K[G]\rightarrow \dH^k(\D)$ such that $\pi\circ\iota=\id$. Then $\dH^k(\D)\cong R\sharp\K[G]$ for $k\in\I_{1,8}-\{6,7\}$, where $R\cong T(V)/\cI_k$ is a (maybe not usual) connected Hopf algebra of dimension $p^2$. If $\chi_2(g_1)=1$ (e.g. $k\in\{3,5\}$), then $R$ is one of the restricted universal enveloping algebras of 2-dimensional restricted Lie algebras classified in \cite[Proposition A.3]{W1}.
\end{rmk}
\begin{pro}
For $k\in\I_{1,8}$, $\gr\dH^k(\D)=\BN(V,\D)\sharp\K[G]$.
\end{pro}
\begin{proof}
Let $\cF_{(n)}:=\cF_{(n-1)}+\K[G]\{y^ix^j\mid i+j=n,i,j\in\I_{0,p-1}$ with $\cF_{(0)}=\K[G]$.
It is easy to check that the filtration $\{\cF_{(n)},n\in\I_{0,p^2-1}\}$ is a Hopf algebra filtration. Therefore,  $(\dH^k(\D))_0=\K[G]$ and hence $\dH^k(\D)$ is a pointed Hopf algebra with coradical $\K[G]$. Furthermore, $\gr_{\cF}\dH^k(\D)=R\sharp\K[G]$, where $\gr_{\cF}\dH^k(\D)$ is the graded Hopf algebra associated to the filtration  $\{\cF_{(n)},n\in\I_{0,p^2-1}\}$. Clearly, $V\subset\Pp(R)$. Since $\dim R=\dim\BN(V)=p^2$, it follows that $R=\BN(V,\D)$. Therefore, the   filtration  $\{\cF_{(n)},n\in\I_{0,p^2-1}\}$ is the coradical filtration of $\dH^k(\D)$.
\end{proof}

\begin{pro}\label{pro:isomorphism-rank-2}
Let $\D:=\D(G,\chi_1,\chi_2,g_1,g_2)$ and $\D':=\D(G,\chi'_1,\chi'_2,g'_1,g'_2)$.
For $k\in\I_{1,8}$, $\dH^k(\D)\cong\dH^k(\D')$ if and only if there exists $f\in\Aut(G)$ and $\sigma\in S_2$ such that $f(g_{\sigma(i)})=g_i^{\prime}$ and $\chi_{\sigma(i)}\cdot f^{-1}=\chi_{i}^{\prime}$.

\end{pro}
\begin{proof}
Set $x_1:=x$ and $x_2:=y$ for convenience. Suppose that $\psi:\dH^1(\D)\rightarrow \dH^k(\D')$ is an isomorphism of Hopf algebras. Then $\psi|_{G}\in\Aut(G)$ and $\psi(x_i)\in\Pp_{1,\psi(g_i)}(\dH^k(\D'))$.

If $\chi'_1= \chi'_2=\chi$ and $g'_1=g'_2=g$, then   $\psi(g_1)=g=\psi(g_2)$, which implies that $g_1=g_2$. Furthermore, there exist some $\alpha_i,\beta_i\in\K$ for $i\in\I_{1,3}$ such that
\begin{align*}
\psi(x_1)=\alpha_1x'_1+\alpha_2x'_2+\alpha_3(1-g),\quad \psi(x_2)=\beta_1x'_1+\beta_2x'_2+\beta_3(1-g),\quad \alpha_1\beta_2\neq \beta_1\alpha_2.
\end{align*}
If $g=1$, we let $\alpha_3=0=\beta_3$.  Since $hx_ih^{-1}=\chi_i(h)x_i$ for all $h\in G$, we see that
\begin{align*}
\chi\cdot\psi(h)=\chi_i(h),\quad (1-\chi_1(h))\alpha_3=0=(1-\chi_2(h))\beta_3.
\end{align*}
Hence $\chi_1=\chi\cdot\psi|_G=\chi_2$.

If $\chi'_1\neq \chi'_2$ or $g'_1\neq g'_2$, then there is $\sigma\in S_2$ and $\alpha_1,\alpha_2\in\K-\{0\},\beta_1,\beta_2\in\K$   such that
\begin{align*}
\psi(g_{i})=g'_{\sigma(i)},\quad \psi(x_i)=\alpha_ix'_{\sigma(i)}+\beta_i(1-g'_{\sigma(i)}),
\end{align*}
where $\beta_i=0$, if $g'_{\sigma(i)}=1$. Indeed, if $g'_1\neq g'_2$, then $\dim \Pp_{1,g'_i}(\dH(\D'))=1$ and hence the claim follows; if $g'_1=g'_2$ and $\chi'_1\neq \chi'_2$, then using the adjoint action of $\K[G]$, one can check easily that the claim follows.
Consequently, applying $\psi$ to the relations $hx_ih^{-1}=\chi_i(h)x_i$ for all $h\in G$, we have $\chi_i=\chi'_{\sigma(i)}\cdot\psi|_G$.

Conversely, let $f\in\Aut(G)$ and $\sigma\in S_2$ such that $f(g_{\sigma(i)})=g_i^{\prime}$ and $\chi_{\sigma(i)}\cdot f^{-1}=\chi_{i}^{\prime}$. Then one can extend $f$ to a Hopf algebra isomorphism from $\dH^k(\D)$ to $\dH^k(\D')$ by $f(x_i)=x'_{\sigma^{-1}(i)}$.
\end{proof}

\begin{pro}\label{pro:non-isomorphic-rank-2}
Let $\D:=\D(G,\chi_1,\chi_2,g_1,g_2)$,  $\D':=\D(G,\chi'_1,\chi'_2,g'_1,g'_2)$ and $i,j\in\I_{1,8}$.
If $i\neq j$, then $\dH^i(\D)\not\cong\dH^j(\D')$ as Hopf algebras.
\end{pro}
\begin{proof}
We show that $\dH^1(\D)\not\cong\dH^k(\D')$ for $k\in\I_{2,8}$. Set $x_1:=x$ and $x_2:=y$ for convenience. Assume that $\psi:\dH^1(\D)\rightarrow \dH^k(\D')$ for $k\in\I_{2,8}$ is an isomorphism of Hopf algebras. Then $\psi|_{G}\in\Aut(G)$ and $\psi(x_i)\in\Pp_{1,\psi(g_i)}(\dH^k(\D'))$.

If $\chi'_1= \chi'_2=\chi$ and $g'_1=g'_2=g$, then as shown in the proof of Proposition \ref{pro:isomorphism-rank-2}, $g_1=g_2$ and $\chi_1=\chi_2$  and there exist some $\alpha_i,\beta_i\in\K$ for $i\in\I_{1,3}$ such that
\begin{align*}
\psi(x_1)=\alpha_1x'_1+\alpha_2x'_2+\alpha_3(1-g),\quad \psi(x_2)=\beta_1x'_1+\beta_2x'_2+\beta_3(1-g),\quad \alpha_1\beta_2\neq \beta_1\alpha_2,
\end{align*}
with conditions: $\alpha_3=0=\beta_3$ if $g=1$.  Applying $\phi$ to the relations $[x_1,x_2]=0$, $x_1^p=0$ and $x_2^p=0$, we see that
\begin{align*}
[x'_1,x'_2]=0,\quad(\alpha_1x'_1+\alpha_2x'_2)^p+\alpha_3^p(1-g^p)=0=(\beta_1x'_1+\beta_2x'_2)^p+\beta_3^p(1-g^p).
\end{align*}
Then from the defining relations in $\dH^k(\D')$, one can check easily that $\alpha_1\beta_2-\alpha_2\beta_1=0$, a contradiction.

If $\chi'_1\neq \chi'_2$ or $g'_1\neq g'_2$, then there is $\sigma\in S_2$ and $\alpha_1,\alpha_2\in\K-\{0\},\beta_1,\beta_2\in\K$   such that
\begin{align*}
\psi(g_{i})=g'_{\sigma(i)},\quad \psi(x_i)=\alpha_ix'_{\sigma(i)}+\beta_i(1-g'_{\sigma(i)}),
\end{align*}
where $\beta_i=0$, if $g'_{\sigma(i)}=1$. Since $(x'_1)^p\neq 0$, $(x'_2)^p\neq 0$ or $[x'_1,x'_2]_c\neq 0$, from $[x_1,x_2]=0$, $x_1^p=0$ and $x_2^p=0$, we have
$\alpha_1\alpha_2=0$, a contradiction.

Similarly, for any $i\in\I_{2,7}$, one can show that $\dH^{i}(\D)\not\cong\dH^k(\D')$ for all $k\in\I_{i+1,8}$.
\end{proof}

\begin{defi}
 For a QPYD-datum $\D:=\D(G,\chi_1,\chi_2,g_1,g_2)$  with $p\nmid \ord(G)$, let $\dH$ be a Hopf algebra such that $\gr \dH\cong\BN(V,\D)\sharp\K[G]$.
\end{defi}

\begin{lem}\label{lem:Nichols-liftings-hx-chi1hhx}
There exists an epimorphism of Hopf algebras from $T(V,\D)\sharp\K[G]$ to $\dH$.
\end{lem}
\begin{proof}
Set $x_1:=x$ and $x_2=y$ for convenience. Since $x_i\in\Pp_{1,g_i}(\dH)$ for $i\in\I_{1,2}$, it follows that there exist two maps $f_1,f_2$ from $G$ to $\K$ such that $hx_ih^{-1}=\chi_i(h)x_i+f_i(h)(1-g_i)$ with conditions: $f_i=0$ if $g_i=1$.

If $g_i\neq 1$ and $g_1\neq g_2$, then there is an exact sequence of $G$-modules:
\begin{align*}
\K\{1-g_i\}\hookrightarrow\Pp_{1,g_i}(\dH)\twoheadrightarrow\K\{x_i\}.
\end{align*}
Since $p\nmid\ord(G)=m$, $\K[G]$  is semisimple and so the sequence is split, which implies that $\Pp_{1,g_i}(\dH)=\K\{1-g_i\}\oplus\K\{x_i\}$ as $G$-modules. Hence we can choose $f_i=0$.

If $g_1=g_2\neq 1$, then there is an exact sequence of $G$-modules:
\begin{align*}
\K\{1-g_1\}\hookrightarrow\Pp_{1,g_1}(\dH)\twoheadrightarrow\K\{x_1\}\oplus\K\{x_2\}.
\end{align*}
The sequence is also split since $\K[G]$ is semisimple, and hence we can choose $f_1=f_2=0$.
\end{proof}
\begin{rmk}
By Lemma \ref{lem:Nichols-liftings-hx-chi1hhx}, it remains to determine liftings of the relations in $\BN(V,\D)$.
\end{rmk}
\begin{lem}\label{lem:liftings-xy-yx-xp-yp}
\begin{itemize}
\item[(i)] $[x,y]_c\in\Pp_{1,g_1g_2}(\dH)$.
\item[(ii)] $x^p\in\Pp_{1,g_1^p}(\dH)$, $y^p\in\Pp_{1,g_2^p}(\dH)$.
\end{itemize}
\end{lem}
\begin{proof}
Since $x\in\Pp_{1,g_1}(\dH)$ and $y\in\Pp_{1,g_2}(\dH)$, it follows that $xy-q_{12}yx\in\Pp_{1,g_1g_2}(\dH)$. By Lemma \ref{lem:Nichols-liftings-hx-chi1hhx}, we have $g_1x=xg_1$ and $g_2y=yg_2$. Then by Theorem \ref{thm:q-binomial}, $x^p\in\Pp_{1,g_1^p}(\dH)$, $y^p\in\Pp_{1,g_2^p}(\dH)$.
\end{proof}

\begin{pro}\label{pro:liftings-dH-dimP(dH)=2}
Assume that $\dim\Pp(\dH)=2$. Then there exists $k\in\I_{1,8}-\{6,7\}$  such that $\dH\cong\dH^k(\D)$.
\end{pro}
\begin{proof}
By assumption, we have $g_1=g_2=1$ and so $\Pp(\dH)=V$. Then by Lemma \ref{lem:liftings-xy-yx-xp-yp}, $xy-yx,x^p,y^p\in\Pp(\dH)$. Since $\Pp(\dH)=\K\{x,y\}$, there exist $\alpha_1,\alpha_2,\beta_1,\beta_2,\gamma_1,\gamma_2\in\K$ such that
 \begin{align}\label{eq:gg11-relations}
 x^p=\alpha_1 x+\alpha_2 y,\quad y^p=\beta_1 x+\beta_2 y,\quad xy-yx=\gamma_1 x+\gamma_2 y.
\end{align}

Consider the adjoint action of $\K[G]$, we have the following conditions:
\begin{itemize}
  \item $\alpha_1=0$, if $\chi_1^{p-1}\neq \epsilon$;
  \item $\alpha_2=0$, if $\chi_1^p\neq \chi_2$;
  \item $\beta_1=0$, if  $\chi_2^p\neq \chi_1$;
  \item $\beta_2=0$, if $\chi_2^{p-1}\neq \epsilon$;
  \item $\gamma_1=0$, if $\chi_2\neq\epsilon$;
  \item $\gamma_2=0$, if $\chi_1\neq\epsilon$.
\end{itemize}

Let $K$ be the subalgebra of $\dH$ generated by $\Pp(\dH)$. Then Lemma \ref{lem:Nichols-liftings-hx-chi1hhx} and equations \eqref{eq:gg11-relations} imply that $K$ is a Hopf subalgebra of $\dH$. Furthermore, $\dH\cong K\sharp\K[G]$. Observe that $K$ is a usual connected Hopf algebras of dimension $p^2$. Indeed, $K$ is isomorphic to the restricted universal enveloping algebra $U(\Pp(H))$ of $\Pp(H)$. Then by \cite[Proposition A.3]{W1}, $K$ is isomorphic to one of the following Hopf algebras
\begin{enumerate}
\item[(1)] $\K[x,y]/(x^p,y^p)$,
 \item[(2)] $\K[x,y]/(x^p-x,y^p)$,
 \item[(3)] $\K[x,y]/(x^p-y,y^p)$,
 \item[(4)] $\K[x,y]/(x^p-x,y^p-y)$,
  \item[(5)] $\K\langle x,y\rangle/(x^p-x,y^p,[x,y]-y)$.

\end{enumerate}

\textbf{Case 1: } Assume that $\chi_1=\chi_2=\epsilon$. Then by Lemma \ref{lem:Nichols-liftings-hx-chi1hhx}, for any $h\in G$, $hx=xh$ and $hy=yh$. Therefore,  $\dH\cong U(\Pp(H))\otimes \K[G]$ and consequently $\dH\cong\dH^i(\chi_1,\chi_2,g_1,g_2)$ for $i\in\I_{1,5}$.

\textbf{Case 2: }  Assume that $\chi_1=\epsilon$ and $\chi_2\neq\epsilon$ $(\chi_2^p\neq\epsilon)$. Then $\alpha_2=\gamma_1=0=\beta_1$, that is, $x^p=\alpha_1x$, $y^p=\beta_2 y$ and $xy-yx=\gamma_2y$.

By induction, we have
\begin{align*}
x^ny=y(x+\gamma_2)^n,\quad xy^n=y^nx+n\gamma_2y^n.
\end{align*}
In particular, $x^py=yx^p+\gamma_2^py$ and $xy^p=y^px$. Then
\begin{align*}
(x^p)y&=(\alpha_1x)y=\alpha_1yx+\alpha_1\gamma_2y,\quad x^{p-1}(xy)=yx^p+\gamma_2^py=\alpha_1yx+\gamma_2^py,\\
x(y^p)&=\beta_2xy=\beta_2yx+\beta_2\gamma_2y,\quad (xy)y^{p-1}=y^px=\beta_2yx.
\end{align*}
Hence the verification of $(x^p)y=x^{p-1}(xy)$, $x(y^p)=(xy)y^{p-1}$ and $(y^p)y=y(y^p)$ imposes the conditions
\begin{align*}
\alpha_1\gamma_2-\gamma_2^p=0, \quad\beta_2\gamma_2=0.
\end{align*}

If $\gamma_2=0$, then we can take $\alpha_1,\beta_2\in\I_{0,1}$ by rescaling $x,y$. If $\alpha_1=\beta_2\in\I_{0,1}$, then $\dH\cong\dH^1(\chi_1,\chi_2,g_1,g_2)$ or $\dH^4(\chi_1,\chi_2,g_1,g_2)$. If $\alpha_1-1=0=\alpha_2$, then $\dH\cong\dH^2(\chi_1,\chi_2,g_1,g_2)$. If $\alpha_1=0=\alpha_2-1$, then $\chi_2^{p-1}=\epsilon$ and hence $\dH\cong\dH^2(\chi_2,\chi_1,g_2,g_1)$.

If $\gamma_2\neq 0$, then $\beta_2=0$ and $\alpha_1=\gamma_2^{p-1}$. Hence we can take $\alpha_1=\gamma_1=1$ via the linear translation $x\mapsto \gamma_2^{-1}x$ and hence $\dH\cong\dH^5(\chi_1,\chi_2,g_1,g_2)$.

\textbf{Case 3: }  Assume that $\chi_1\neq\epsilon$ and $\chi_2=\epsilon$. Then by swapping $(\chi_1,g_1,x)$ and $(\chi_2,g_2,y)$, it is the last case.

\textbf{Case 4: }  Assume that $\chi_1=\chi_2\neq\epsilon$.  Then  $\gamma_1=\gamma_2=0$, that is, $x^p=\alpha_1x+\alpha_2y$, $y^p=\beta_1x+\beta_2y$, $xy-yx=0$. If $\chi_1^{p-1}\neq\epsilon$, then   $\alpha_i=\beta_i=0$ for $i\in\I_{1,2}$ and hence $\dH\cong\dH^1(\chi_1,\chi_2,g_1,g_2)$.
If $\chi_1^{p-1}=\epsilon$, then one can check easily that for any isomorphism $\psi\in\Aut(U(\Pp(\dH)))$, there is $\phi\in\Aut(\dH)$ (e.g. $\psi\sharp\id$) such that $\phi|_{U(\Pp(H)}=\psi$, and then by \cite[Proposition A.3.]{W1}, $\dH\cong\dH^i(\chi_1,\chi_2,g_1,g_2)$ for $i\in\I_{1,4}$. Indeed, for any $\psi\in\Aut(U(\Pp(\dH)))$, there are some $a_1,a_2,b_1,b_2\in\K$ such that
\begin{align*}
\psi(x)=a_1x+a_2y,\quad \psi(y)=b_1x+b_2y,\quad a_1b_2-a_2b_1\neq 0.
\end{align*}
Since $x,y,x^p,y^p\in\Pp(\dH)^{\chi_1}=\{t\in\Pp(\dH)\mid  hth^{-1}=\chi_1(h)t, \forall h\in G\}$, it follows that $\psi$ is also an isomorphism of Hopf algebras in ${}_{G}^{G}\mathcal{YD}$ and hence $\psi\sharp\id\in\Aut(\dH)$.

\textbf{Case 5: } Assume that $\chi_1\neq\epsilon, \chi_2\neq\epsilon$ and $\chi_1\neq\chi_2$. Then $\gamma_1=0=\gamma_2$. that is, $x^p=\alpha_1x+\alpha_2y$, $y^p=\beta_1x+\beta_2y$, $xy-yx=0$.

If $\chi_1^{p-1}=\epsilon=\chi_2^{p-1}$, then $\chi_1^{p}=\chi_1\neq\chi_2$ and $\chi_2^p=\chi_2\neq\chi_1$. Hence $\alpha_2=\beta_1=0$,  and by rescaling $x,y$, we can take $\alpha_1,\beta_2\in\I_{0,1}$. If $\alpha_1=\beta_2\in\I_{0,1}$, then $\dH\cong\dH^1(\chi_1,\chi_2,g_1,g_2)$ or $\dH^4(\chi_1,\chi_2,g_1,g_2)$. If $\alpha_1-1=0=\alpha_2$, then $\dH\cong\dH^2(\chi_1,\chi_2,g_1,g_2)$. If $\alpha_1=0=\alpha_2-1$, then $\dH\cong\dH^2(\chi_2,\chi_1,g_2,g_1)$.

If $\chi_1^{p-1}=\epsilon$ and $\chi_2^{p-1}\neq\epsilon$, then $\chi_1^p\neq \chi_2$ and $\chi_2^{p}\neq\chi_1$. Indeed, if $\chi_2^p=\chi_1$, then $\chi_2^p=\chi_1^p$ and so $\chi_1=\chi_2$, a contradiction. Therefore, $\alpha_2=0=\beta_2=\beta_1$. Then by rescaling $x$,  we take $\alpha_1\in\I_{0,1}$ and hence $\dH\cong\dH^i(\chi_1,\chi_2,g_1,g_2)$ for $i\in\I_{1,2}$.

If $\chi_1^{p-1}\neq\epsilon$ and $\chi_2^{p-1}=\epsilon$, then by swapping $(\chi_1,g_1,x)$ and $(\chi_2,g_2,y)$, it is the last case.

If $\chi_1^{p-1}\neq\epsilon$ and $\chi_2^{p-1}\neq\epsilon$, then $\alpha_1=\beta_2=0$. If $\chi_1^p\neq\chi_2$ and $\chi_2^p\neq\chi_1$, then $\alpha_2=0=\beta_1$ and hence $\dH\cong\dH^1(\chi_1,\chi_2,g_1,g_2)$.
If $\chi_1^p=\chi_2$ and $\chi_2^p\neq\chi_1$, then $\beta_1=0$ and $\alpha_2\in\I_{0,1}$ by rescaling $x$, which implies that $\dH\cong\dH^i(\chi_1,\chi_2,g_1,g_2)$ for $i\in\{0,3\}$. If $\chi_1^p\neq\chi_2$ and $\chi_2^p=\chi_1$, then by swapping $(\chi_1,g_1,x)$ and $(\chi_2,g_2,y)$, it is the last case. If $\chi_1^p=\chi_2$ and $\chi_2^p=\chi_1$, then $x^p=\alpha_2y$ and $y^p=\beta_1x$ for $\alpha_2,\beta_1\in\K$. If $\alpha_2=0$, then by rescaling $y$, we can take $\beta_1\in\I_{0,1}$ and hence $\dH\cong\dH^k(\chi_1,\chi_2,g_1,g_2)$ for $k\in\{1, 3\}$.
If $\beta_1=0$, then by swapping $x$ and $y$, it is the last case.
If $\alpha_2\neq 0$ and $\beta_1\neq0$, then we can take $\alpha_2=1=\beta_1$ via the linear translation
$x\mapsto a^{-1}x,\ y\mapsto b^{-1}y$ satisfying $a^p=\alpha_2b$ and $b^p=\beta_1a$ and hence $\dH\cong\dH^8(\chi_1,\chi_2,g_1,g_2)$.
\end{proof}

\begin{pro}\label{pro:liftings-dimP(dH)=1}
 Assume that $\dim\Pp(\dH)=1$. Then there is $k\in\I_{1,7}$ such that $\dH\cong\dH^k(\D)$.
\end{pro}
\begin{proof}
By assumption, we may assume that $g_1=1$ and $g_2\neq 1$. By Lemma \ref{lem:liftings-xy-yx-xp-yp}, $xy-yx\in\Pp_{1,g_2}(\dH)$, $x^p\in\Pp(\dH)$ and $y^p\in\Pp_{1,g_2^p}(\dH)$. Since $\Pp(\dH)=\K\{x\}$ and $\Pp_{1,g_2}(\dH)=\K\{1-g_2,y\}$,
 \begin{align*}
 x^p=\lambda_2x, \quad [x,y]=\lambda_3y+\lambda_4(1-g_2)
\end{align*}
for some $\lambda_2,\lambda_3,\lambda_4\in\K$. Consider the adjoint action of $\K[G]$, we have the conditions
\begin{itemize}
  \item $\lambda_2=0$, if $\chi_1^{p-1}\neq\epsilon$;
  \item $\lambda_3=0$, if $\chi_1\neq\epsilon$;
  \item $\lambda_4=0$, if $\chi_1\chi_2\neq\epsilon$.
\end{itemize}

Now we determine liftings of $y^p=0$ in $\gr\dH$. Since $p\nmid\ord(\G(\dH))$, it follows that $g_2^p\neq 1$. Therefore, there are the follows two cases.

\textbf{Case 1: } Suppose that $g_2^{p-1}=1$, that is, $g_2^p=g_2$. Then $y^p=\lambda_5y+\lambda_6(1-g_2^{p})$ for some $\lambda_5,\lambda_6\in\K$. Consider the adjoint action of $\K[G]$, we have conditions:
\begin{itemize}
  \item $\lambda_5=0$, if $\chi_2^{p-1}\neq\epsilon$;
  \item $\lambda_6=0$, if $\chi_2\neq\epsilon$.
\end{itemize}

By induction, for $n>0$,  we have
       \begin{align*}
       xy^n&=y^nx+n\lambda_3y^n+n\lambda_4(1-g_2)y^{n-1},\\
       x^ny&=y(x+\lambda_3)^n+\lambda_4(1-g_2)\sum_{k=0}^{n-1}(x+\lambda_3)^{n-k-1}x^k.
       \end{align*}
       In particular, we have
       \begin{align*}
       xy^p&=y^px,\\
       x^py&=yx^p+\lambda_3^py+\lambda_4(1-g_2)\sum_{k=0}^{p-1}(x+\lambda_3)^{p-k-1}x^k\\
       &=yx^p+\lambda_3^py+\lambda_4(1-g_2)\sum_{k=0}^{p-1}\sum_{i=0}^{p-1-k}\left(\begin{array}{c}p-1-k \\i \end{array}\right)
       x^{p-1-i}\lambda_3^i\\
       &=yx^p+\lambda_3^py+\lambda_4\lambda_3^{p-1}(1-g_2).
       \end{align*}
       Then
       \begin{align*}
       x(y^p)&=\lambda_5xy+\lambda_6(1-g_2^p)x=\lambda_5yx+\lambda_3\lambda_5y+\lambda_4\lambda_5(1-g_2)+\lambda_6(1-g_2^p)x,\\
       (xy)y^{p-1}&=(y^p)x=\lambda_5yx+\lambda_6(1-g_2^p)x,\\
       (x^p)y&=\lambda_2xy=\lambda_2yx+\lambda_2\lambda_3y+\lambda_2\lambda_4(1-g_2),\\
       x^{p-1}(xy)&=y(x^p)+\lambda_3^py+\lambda_4(1-g_2)\sum_{k=0}^{p-1}(x+\lambda_3)^{p-1-k}x^k\\
       &=\lambda_2yx+\lambda_3^py+\lambda_4\lambda_3^{p-1}(1-g_2).
       \end{align*}
      Then the overlaps $x(yy^{p-1})=(xy)y^{p-1}$,$x^{p-1}(xy)=(x^{p-1}x)y$ gives the conditions
      \begin{align}\label{eq:P1g-1}
      (\lambda_2-\lambda_3^{p-1})\lambda_3=0=(\lambda_2-\lambda_3^{p-1})\lambda_4, \quad\lambda_5\lambda_3=0=\lambda_5\lambda_4.
      \end{align}
      It is easy to check that other overlaps are resolvable and give no  conditions.

      Assume that $\lambda_2=0$. Then equations \eqref{eq:P1g-1} imply that $\lambda_3=0$ and we can choose $\lambda_6=0$  via the linear translation $y \mapsto y-a(1-g^{})$ satisfying $a^p-\lambda_5a=\lambda_6$. Furthermore, we can also take $\lambda_4,\lambda_5\in\I_{0,1}$ by rescaling $x,y$. Observe that $\lambda_4\lambda_5=0$.
      If $\lambda_4=0=\lambda_5$, then  $\dH\cong\dH^1(\chi_1,\chi_2,g_1,g_2)$; if $\lambda_4=0=\lambda_5-1$, then $\chi_2^{p-1}=\epsilon$, $g_2^p=g_2$ and hence  $\dH\cong\dH^{2}(\chi_2,\chi_1,g_2,g_1)$; if $\lambda_4-1=0=\lambda_5$, then $\chi_1\chi_2=\epsilon$ and   $\dH\cong\dH^6(\chi_1,\chi_2,g_1,g_2)$

      Assume that $\lambda_2\neq 0$. Then  $\chi_1^{p-1}=\epsilon$ and we can take $\lambda_2=1$ by rescaling $x$. If $\lambda_5\neq 0$, then $\chi_2^{p-1}=\epsilon$ and equations \eqref{eq:P1g-1} imply that $\lambda_3=\lambda_4=0$. By rescaling $y$, we can take $\lambda_5=1$. Furthermore, we can choose $\lambda_6=0$ via the linear translation $y \mapsto y+a(1-g_2)$ satisfying $a^p-a=\lambda_6$. Hence $\dH\cong\dH^4(\chi_1,\chi_2,g_1,g_2)$.
If $\lambda_5=0$, then $\lambda_3^p=\lambda_3$,  $(1-\lambda_3^{p-1})\lambda_4=0$ and hence we can take $\lambda_3\in\I_{0,1}$ by rescaling $x$ and take $\lambda_6=0$ via the linear translation $y \mapsto y-a(1-g)$ satisfying $a^p=\lambda_6$. If $\lambda_3=0$, then $\lambda_4=0$ and hence $\dH\cong\dH^2(\chi_1,\chi_2,g_1,g_2)$. If $\lambda_3=1$, then  we take $\lambda_4\in\I_{0,1}$ and hence $\dH\cong\dH^i(\chi_1,\chi_2,g_1,g_2)$ for $i\in\{5,7\}$.

  \textbf{Case 2: } Suppose that $g_2^{p-1}\neq 1$, that is, $g_2^{p}\neq g_2$.  Then $y^p=\lambda_7(1-g_2^{p})$ for some $\lambda_7\in\K$ with conditions: $\lambda_7=0$ if $\chi_2\neq\epsilon$.

     Similar to the last case, we have
      \begin{align*}
      [x,y^p]=0,\quad
      [x^p,y]=\lambda_3^py+\lambda_3^{p-1}\lambda_4(1-g_2^{}).
      \end{align*}
      Applying the Diamond Lemma, the verification of overlaps \eqref{eq:p2q-11-1}--\eqref{eq:p2q-11-4} amounts to conditions:
      \begin{gather*}
      \lambda_2\lambda_3-\lambda_3^p=0=\lambda_2\lambda_4-\lambda_3^{p-1}\lambda_4.
      \end{gather*}
      We take $\lambda_7=0$ via the linear translation $y \mapsto y+a(1-g^{})$ satisfying $a^p=\lambda_7$.

      If $\lambda_2=0$, then $\lambda_3=0$ and we can take $\lambda_4\in\I_{0,1}$ by rescaling $x$.
      Hence $\dH\cong\dH^i(\chi_1,\chi_2,g_1,g_2)$ for $i\in\{1,6\}$.

      If $\lambda_2\neq 0$, then we can take $\lambda_2=1$ by rescaling $x$. Hence $\lambda_3^p=\lambda_3$ and $(\lambda_3^{p-1}-1)\lambda_4=0$. We can take $\lambda_3\in\I_{0,1}$ by rescaling $x$. If $\lambda_3=0$, then $\lambda_4=0$ and hence $\dH\cong\dH^2(\chi_1,\chi_2,g_1,g_2)$. If $\lambda_3=1$, then we take $\lambda_4\in\I_{0,1}$ by rescaling $y$ and hence $\dH\cong\dH^i(\chi_1,\chi_2,g_1,g_2)$ for $i\in\{5,7\}$.
\end{proof}

\begin{pro}\label{pro:liftings-dimP1g=2}
Assume that $\dim\Pp_{1,g}(\dH)=2$ for some $g\neq 1$. Then there exists $k\in\I_{1,8}-\{5,7\}$ such that $\dH\cong\dH^k(\D)$.
\end{pro}
\begin{proof}
 By assumption, we may assume that $g_1=g_2=g\neq 1$. Then by Lemma \ref{lem:liftings-xy-yx-xp-yp}, $x^p, y^p\in\Pp_{1,g_1^p}(\dH)$ and $[x,y]\in\Pp_{1,g_1g_2}(\dH)$. Observe that $\Pp_{1,g_1}(\dH)=\K\{x\}\oplus\K\{y\}\oplus\K\{1-g_1\}$ and $\Pp_{1,g_1g_2}(\dH)=\K\{1-g_1g_2\}$. Hence
    \begin{gather*}
    x^p=\lambda_3x+\lambda_4y+\lambda_5(1-g_1^p),\\
    y^p=\lambda_6x+\lambda_7y+\lambda_8(1-g_2^p),\quad xy-yx=\lambda_9(1-g_1g_2),
    \end{gather*}
    where $\lambda_3,\ldots,\lambda_9\in\K$ with conditions: $\lambda_3=\lambda_4=\lambda_6=\lambda_7=0$, if $g_1^{p-1}\neq 1$. Consider the adjoint action of $\K[G]$, we have the following conditions:
    \begin{itemize}
      \item $\lambda_3=0$, if $\chi_1^{p-1}\neq\epsilon$;
      \item $\lambda_4=0$, if $\chi_1^p\neq\chi_2$;
      \item $\lambda_5=0$, if $\chi_1\neq\epsilon$;
      \item $\lambda_6=0$, if $\chi_2^p\neq\chi_1$;
      \item $\lambda_7=0$, if $\chi_2^{p-1}\neq\epsilon$;
      \item $\lambda_8=0$, if $\chi_2\neq\epsilon$;
      \item $\lambda_9=0$, if $\chi_1\chi_2\neq\epsilon$.

    \end{itemize}

    By induction, for $n>0$,
    \begin{align*}
    x^ny=yx^n+nx^{n-1}\lambda_9(1-g^2),\quad xy^n=y^nx+ny^{n-1}\lambda_9(1-g^2).
    \end{align*}
    In particular, $x^py=yx^p$ and $xy^p=y^px$. Then
    \begin{gather*}
    x^{p-1}(xy)=yx^p=\lambda_3yx+\lambda_4y^2+\lambda_5y(1-g^p),\\
    (x^p)y=\lambda_3xy+\lambda_4y^2+\lambda_5y(1-g^p)=\lambda_3yx+\lambda_3\lambda_9(1-g^2)+\lambda_4y^2+\lambda_5y(1-g^p),\\
    (xy)y^{p-1}=y^px=\lambda_6x^2+\lambda_7yx+\lambda_8(1-g^p)x,\\
    x(y^p)=\lambda_6x^2+\lambda_7xy+\lambda_8(1-g^p)x=\lambda_6x^2+\lambda_7yx+\lambda_9\lambda_7(1-g^2)+\lambda_8(1-g^p)x.
    \end{gather*}
    Hence the verification of \eqref{eq:p2q-11-4} amounts to
    $$\lambda_3\lambda_9=0,\quad \lambda_7\lambda_9=0.$$
     Similarly, the overlaps \eqref{eq:p2q-11-1}, \eqref{eq:p2q-11-2}, \eqref{eq:p2q-11-3} are resolvable and give no conditions. The verification of \eqref{eq:p2q-11-5} amounts to the condition
\begin{align*}
\lambda_4\lambda_9=0=\lambda_6\lambda_9.
\end{align*}
    We can choose $\lambda_9\in\I_{0,1}$ by rescaling $x,y$.

    If $\lambda_9=0$, then we can take $\lambda_5=0=\lambda_8$ via the linear translation $x \mapsto x+\alpha_1(1-g),y\mapsto y+\alpha_2(1-g)$ satisfying
    \begin{align*}
    \alpha_1^p-\lambda_3\alpha_1-\lambda_4\alpha_2=-\lambda_5,\quad
    \alpha_2^p-\lambda_6\alpha_1-\lambda_7\alpha_2=-\lambda_8.
    \end{align*}
    Indeed, since $\K$ is algebraically closed,   equations have the solutions. If $g_1^{p-1}\neq 1$, then $\lambda_3=\lambda_4=\lambda_6=\lambda_7=0$ and hence $\dH\cong\dH^1((\chi_1,\chi_2,g_1,g_2)$. Now we assume that $g_1^{p-1}=1$.

    Assume that $\chi_1\neq\chi_2$. If $\chi_1^p=\chi_1$, then $\chi_1^p\neq\chi_2$ and $\chi_2^p\neq\chi_1$, which implies that $\lambda_4=0=\lambda_6$. Indeed, if $\chi_2^p=\chi_1$, then $\chi_1^p=\chi_2^p$ and so $\chi_1=\chi_2$, a contradiction. By rescaling $x, y$, we can take $\lambda_3,\lambda_7\in\I_{0,1}$ and hence $\dH\cong\dH^i(\chi_1,\chi_2,g_1,g_2)$ for $i\in\{1,2,4\}$. If $\chi_2^p=\chi_2$, then by exchanging $x$ with $y$, we have $\dH\cong\dH^i(\chi_1,\chi_2,g_1,g_2)$ for $i\in\{1,2,4\}$. If $\chi_1^p\neq \chi_1$ and $\chi_2^p\neq \chi_2$,  then $\lambda_3=0=\lambda_7$. If $\lambda_4=0$, then by rescaling $y$, we can take $\lambda_6\in\I_{0,1}$ and hence $\dH\cong\dH^k(\chi_1,\chi_2,g_1,g_2)$ for $k\in\{1, 3\}$. If $\lambda_6=0$, then by swapping $x$ and $y$, it is the last case.
If $\lambda_4\neq 0$ and $\lambda_6\neq 0$, then we can take $\lambda_4=1=\lambda_6$ by rescaling $x,y$ and hence $\dH\cong\dH^8(\chi_1,\chi_2,g_1,g_2)$.

    Assume that $\chi_1=\chi_2$. If $\chi_1^p\neq\chi_1$, then $\chi_1^p\neq \chi_2$, $\chi_2^p\neq \chi_1$ and $\chi_2^p\neq\chi_2$ and hence $\lambda_3=\lambda_4=\lambda_6=\lambda_7=0$, which implies that $\dH\cong\dH^1(\chi_1,\chi_2,g_1,g_2)$.
    If $\chi_1^p=\chi_1$, then $\chi_1^p=\chi_2^p=\chi_1=\chi_2$.
    It is clear that  $\dH= S(\lambda_3,\lambda_4,\lambda_6,\lambda_7)\sharp\K[G]$, where $S(\lambda_3,\lambda_4,\lambda_6,\lambda_7):=\K[x, y]/(x^p-(\lambda_3x+\lambda_4y),  y^p-(\lambda_6x+\lambda_7y))\in{}_{G}^{G}\mathcal{YD}$ with
    \begin{align*}
    h\cdot x=\chi_1(h)x, \quad h\cdot y=\chi_1(h)y, \quad \delta(x)=g_1\otimes x, \quad \delta(y)=g_1\otimes y;\\
    \Delta_S(x)=x\otimes 1+1\otimes x, \quad\Delta_S(y)=y\otimes 1+1\otimes y.
    \end{align*}
  Let $S:=S(\lambda_4,\lambda_5,\lambda_6,\lambda_7)$ for short. Clearly, $S$ is a usual connected Hopf algebra of dimension $p^2$. Observe that $x,y,x^p,y^p\in\Pp_{1,g_1}^{\chi_1}(\dH)=\{t\in\Pp_{1,g}(\dH)\mid hth^{-1}=\chi_1(h)t,\forall h\in G\}$. Then one can easily check that  for any isomorphism $\psi\in\Aut(S)$, $\psi$ is also an isomorphism of Hopf algebras in ${}_G^G\mathcal{YD}$ and hence there is $\phi\in\Aut(\dH)$ (e.g. $\psi\sharp\id$) such that $\phi|_{S}=\psi$; Then by \cite[Proposition A.3]{W1}, $\dH\cong\dH^i(\chi_1,\chi_2,g_1,g_2)$ for $i\in\I_{1,4}$.

   If $\lambda_9=1$, then $\chi_1\chi_2=\epsilon$, $\lambda_3=0=\lambda_7$ and $\lambda_4=0=\lambda_6$. Furthermore, we can take $\lambda_5=0=\lambda_8$ via the linear translation $x\mapsto x-a(1-g),\ y\mapsto y-b(1-g)$ satisfying $a^p =\lambda_5$ and $b^p =\lambda_8$. Hence  $\dH\cong\dH^6(\chi_1,\chi_2,g_1,g_2)$.
\end{proof}

\begin{pro}\label{pro:liftings-dimP1g=1}
 Assume that $g_1\neq 1, g_2\neq 1,g_1\neq g_2$. Then there exists $k\in\I_{1,8}-\{5,7\}$ such that $\dH\cong\dH^k(\D)$.
\end{pro}
\begin{proof}
By assumption, $g_1g_2^{-1}\neq 1$, $\Pp_{1,g_1}(\dH)=\K\{x\}\oplus\K\{1-g_1\}$ and $\Pp_{1,g_2}(\dH)=\K\{y\}\oplus\K\{1-g_2\}$ and $\Pp_{1,g_1g_2}(\dH)=\K\{1-g_1g_2\}$. By Lemma \ref{lem:liftings-xy-yx-xp-yp}, $x^p\in\Pp_{1,g_1^p}(\dH), y^p\in\Pp_{1,g_2^p}(\dH)$ and $[x,y]_c\in\Pp_{1,g_1g_2}(\dH)$. Then
\begin{align*}
[x,y]_c=\lambda_0(1-g_1g_2),
\end{align*}
for $\lambda_0\in\K$. Consider the adjoint action of $\K[G]$, we have the condition: $\lambda_0=0$ if $\chi_1\chi_2\neq\epsilon$. Furthermore, if $\chi_1\chi_2=\epsilon$, then $\chi_1(g_2)=\chi_2(g_1)=1$.

\textbf{Case 1: } Suppose that $g_1^p=g_1(\neq g_2)$. Then $g_2^p\neq g_1$. Indeed, if $g_2^p=g_1$, then $g_1^p=g_2^p$, that is, $(g_1g_2^{-1})^p=1$, which implies that $p\mid m=|G|$, a contradiction. Therefore, there are  $\lambda_2,\lambda_3,\lambda_4,\lambda_5\in\K$ such that
 \begin{align*}
x^p=\lambda_1x+\lambda_3(1-g_1^p),\quad y^p=\lambda_2y+\lambda_4(1-g_2^p),
\end{align*}
 with conditions:
\begin{itemize}
  \item $\lambda_1=0$,  if $\chi_1^{p-1}\neq\epsilon$;
  \item $\lambda_2=0$,  if $\chi_1^{p-1}\neq\epsilon$ or $g_2^{p-1}\neq 1$;
  \item $\lambda_{i+2}=0$,  if $\chi_i\neq\epsilon$ for $i\in\I_{1,2}$.
\end{itemize}
Furthermore, if $\lambda_3\neq 0$ or $\lambda_4\neq 0$, then $\chi_1(g_2)=\chi_2(g_1)=1$. Then we can take $\lambda_3=0=\lambda_4$ via the linear translation $x\mapsto x+a(1-g_1),y\mapsto b(1-g_2)$ satisfying $a^p-\lambda_1a=\lambda_3,b^p-\lambda_2b=\lambda_4$. We can take $\lambda_1,\lambda_2\in\I_{0,1}$ by rescaling $x,y$.

If $\lambda_1=\lambda_2\in\I_{0,1}$, then $\dH\cong\dH^1(\chi_1,\chi_2,g_1,g_2)$ or $\dH^4(\chi_1,\chi_2,g_1,g_2)$. If $\lambda_1-1=0=\lambda_2$, then $\dH\cong\dH^2(\chi_1,\chi_2,g_1,g_2)$. If $\lambda_1=0=\lambda_2-1$, then $\dH\cong\dH^2(\chi_2,\chi_1,g_2,g_1)$.

\textbf{Case 2: } Suppose that $g_1^p=g_2 (\neq g_1)$. Then $\chi_1(g_2)=1=\chi_2(g_1)$ and  $g_2^p\neq g_2$. Indeed, if $g_2^p=g_2$, then $g_1^p=g_2^p$ and hence $g_1=g_2$, a contradiction. Therefore,
\begin{align*}
x^p=\lambda_1y+\lambda_3(1-g_2),\quad y^p=\lambda_2x+\lambda_4(1-g_2^p),
\end{align*}
for $\lambda_2,\lambda_3,\lambda_4,\lambda_5\in\K$ with conditions:
\begin{itemize}
  \item $\lambda_1=0$,  if $\chi_1^{p}\neq\chi_2$;
  \item $\lambda_2=0$,  if $\chi_2^{p}\neq\chi_1$ or $g_2^{p}\neq g_1$;
  \item $\lambda_{i+2}=0$,  if $\chi_i\neq\epsilon$ for $i\in\I_{1,2}$.
\end{itemize}
 The verification of $[x^p,x]=0$ and $[y^p,y]=0$ amounts to the conditions
 \begin{align*}
 \lambda_0\lambda_1=0=\lambda_0\lambda_2.
 \end{align*}

If $\lambda_2=0$, then we can take $\lambda_3=0=\lambda_4$ via the linear translation $x\mapsto x+a(1-g_1),y\mapsto b(1-g_2)$ satisfying $a^p-\lambda_1b=\lambda_3,b^p=\lambda_4$. Then we can take $\lambda_0,\lambda_1\in\I_{0,1}$ by rescaling $x,y$. If $\lambda_0=0$, then $\lambda_1\in\I_{0,1}$ and hence  $\dH\cong\dH^1(\chi_1,\chi_2,g_1,g_2)$ or $\dH^3(\chi_1,\chi_2,g_1,g_2)$. If $\lambda_0=1$, then $\lambda_1=0$ and hence $\dH\cong\dH^6(\chi_1,\chi_2,g_1,g_2)$.

If $\lambda_2\neq 0$, then $\lambda_0=0$, $\chi_2^p=\chi_1$ and $g_2^p=g_1$. Furthermore, we can take $\lambda_3=0=\lambda_4$ via the linear translation $x\mapsto x+a(1-g_1),y\mapsto y+b(1-g_2)$ satisfying $a^p-\lambda_1b=\lambda_3, b^p-\lambda_2a=\lambda_4$.
Then we can take $\lambda_2=1$ and $\lambda_1\in\I_{0,1}$ by rescaling $x,y$. Therefore, $\dH\cong\dH^3(\chi_2,\chi_1,g_2,g_1)$ or $\dH^8(\chi_1,\chi_2,g_1,g_2)$.

\textbf{Case 3: } Suppose that $g_1^p\neq g_1, g_1^p\neq g_2$. If $g_2^p=g_2$, then by swapping $x$ and $y$, it is the case 1. If $g_2^p=g_1$, then by swapping $x$ and $y$, it is the case 2. Therefore, we may assume that $g_2^p\neq g_1, g_1^p\neq g_2$ and hence there exist $\lambda_3,\lambda_4\in\K$ such that
\begin{align*}
x^p=\lambda_3(1-g_1^p),\quad y^p=\lambda_4(1-g_2^p).
\end{align*}

If $\chi_1\chi_2=\epsilon$, then $\chi_1(g_2)=\chi_2(g_1)=1$ and we can take $\lambda_3=0=\lambda_4$ via the linear translation $x\mapsto x+a(1-g_1), y\mapsto y+b(1-g_2)$ satisfying $a^p=\lambda_3,b^p=\lambda_4$. Therefore, $\dH\cong\dH^1(\chi_1,\chi_2,g_1,g_2)$ or $\dH^6(\chi_1,\chi_2,g_1,g_2)$.

If $\chi_1\chi_2\neq\epsilon$, then $\lambda_0=0$ and $\lambda_3\lambda_4=0$. If $\chi_1=\epsilon$ or $\chi_2=\epsilon$, then $\chi_1(g_2)=\chi_2(g_1)=1$ and we can take $\lambda_3=0=\lambda_4$ via the linear translation $x\mapsto x+a(1-g_1), y\mapsto y+b(1-g_2)$ satisfying $a^p=\lambda_3,b^p=\lambda_4$; otherwise consider the adjoint action of $\K[G]$, we have $\lambda_3=0=\lambda_4$. Consequently, $\dH\cong\dH^1(\chi_1,\chi_2,g_1,g_2)$.
\end{proof}

\begin{thm}\label{thm:liftingsofdH}
There exists $k\in\I_{1,8}$ and a QPYD-datum $\D$ such that $\dH\cong\dH^k(\D)$ as Hopf algebras.
\end{thm}
\begin{proof}
By definition, $\dH\cong\gr\dH$ as coalgebras. According to  the spaces of skew-primitive elements, there are four possibilities: (1) $\dim\Pp(\dH)=2$,  (2) $\dim\Pp(\dH)=1$, (3)  $\dim\Pp_{1,g}(\dH)=2$ for some $g\neq 1\in G$, (4) $\dim\Pp_{1,g_1}(\dH)=1=\dim\Pp_{1,g_2}(\dH)$ with $g_1\neq 1,g_2\neq 1$. Consequently, it follows by Propositions \ref{pro:liftings-dH-dimP(dH)=2}, \ref{pro:liftings-dimP(dH)=1}, \ref{pro:liftings-dimP1g=2} and \ref{pro:liftings-dimP1g=1}.
\end{proof}

\section{The diagrams are not Nichols algebras}\label{secNonNicholsalgebra}
Let $\Char\K=p>0$. In this section, we study pointed Hopf algebras of dimension $p^2m$ whose diagrams are not Nichols algebras and have dimension $p^2$. We first classify the coradically graded ones, then determine the liftings and finally determine the isomorphism classes.
\subsection{The coradically graded pointed Hopf algebras of dimension $p^2m$}

Let $\cH$ be a coradically graded pointed Hopf algebra of dimension $p^2m$ that is not generated by group-like elements and skew-primitive elements, where $\G(\cH)$ is of order $m$ with $(p,m)=1$. We study the structure of $\cH$.
\begin{lem}\cite[4.2~Case (C)]{NW}\label{lem:H-coh-1}
Let $\Char\K=p>0$ and $\BN(V)=\K[x]/(x^p)$ with $V:=\K\{x\}$. Then $\dim H^2(\K,\BN(V))=1$ with a generator $\omega_0(x)$, where
$$\omega_0(x)=\sum_{i=1}^{p-1}\frac{(p-1)!}{i!(p-i)!}x^i\otimes x^{p-i}.$$
\end{lem}

\begin{lem}\label{lem:p2dividedimR}
Let $R:=\oplus_{n=0}^{\infty}R(n)$ be a strictly graded Hopf algebra in ${}_G^G\mathcal{YD}$ for some group algebra $\K[G]$ such that $\BN(R(1))=p$. Assume that $R$ is not a Nichols algebra. Then $R$ contains a Hopf subalgebra $S$ of dimension $p^2$, where  $S\cong\K[x,y]/(x^p,y^p)\in{}_{G}^{G}\mathcal{YD}$ with
\begin{align*}
\Delta_R(x)=x\otimes 1+1\otimes x, \Delta_R(y)=y\otimes 1+1\otimes y+\omega_0(x),
\end{align*}
with $x\in R(1)_{g}^{\chi}$ and $y\in R(p)_{g^p}^{\chi^p}$ for some $g\in\mathcal{Z}(\G(\cH))$ and $\chi\in\widehat{\G(\cH)}$ satisfying $\chi(g)=1$.
\end{lem}
\begin{proof}
Let $V:=R(1)$ for short. By assumption, $\dim V=1$ with trivial braiding. Set $V:=\K\{x\}$. Then $\BN(V)=\K[x]/(x^p)$, which is a usual connected Hopf algebra of dimension $p$. Furthermore, by Remark \ref{rmk:dimV=1}, $\BN(V)\in{}^{G}_{G}\mathcal{YD}$ with $x\in V_{g}^{\chi}$, where $g$ lies in the center of $G$, $\chi\in \widehat{G}$ satisfying $\chi(g)=1$.

By Lemma \ref{lem:H-coh-1},  $\dim H^2(\K,\BN(V))=1$ with a generator $\omega_0(x)$. Since $\omega_0(x)\in\sum_{i=1}^{p-1}R(i)\otimes R(p-i)$ and $\{x^i\}_{i=0}^{p-1}$ is a basis of $\BN(V)$,    the total degree of $\omega_0(x)$ is $p$ and $\BN(V)(p)=0$.
 Then by Theorem \ref{thm:non-primitive-genetrators-Hoch}, there is an isomorphism  in ${}_{G}^{G}\mathcal{YD}$:
\begin{gather*}
d^1:R(p)\rightarrow H^{2,p}(\K,\BN(V)).
\end{gather*}
 Therefore, $\dim R(p)=1$ and there exists a basis $\{y\}$ of $R(p)$ such that $d^1(y)=\omega_0(x)$, that is,
\begin{align*}
\Delta_R(y)=y\otimes 1+1\otimes y+\omega_0(x).
\end{align*}
Furthermore, for any $h\in G$,
\begin{gather*}
\delta(\omega_0)=g^{p}\otimes \omega_0,\quad h\cdot\omega_0=\chi^p(h)\omega_0;\quad
\delta(y)=g^{p}\otimes y,\quad h\cdot y=\chi^{p}(h)y.
\end{gather*}
Then using the braiding formula in ${}_{G}^{G}\mathcal{YD}$, we have $c(a\otimes b)=b\otimes a$ for $a,b\in\{x,y\}$. Furthermore, it follows by a direct computation that $xy-yx\in\Pp(R)=R(1)\cap R(p+1)$ and hence $xy-yx=0$ in $R$. Then
\begin{align*}
\Delta(y^p)&=(y\otimes 1+1\otimes y+\omega_0(x))^p=y^p\otimes 1+1\otimes y^p+\omega_0(x)^p\\
& =y^p\otimes 1+1\otimes y^p+\omega_0(x^p)=y^p\otimes 1+1\otimes y^p,
\end{align*}
which implies that $y^p\in R(1)\cap R(p^2)$ and hence $y^p=0$ in $R$.  Consequently,  $S:=K[x,y]/(x^p,y^p)$ is the Hopf subalgebra of $R$ generated by $x$ and $y$ in ${}_{G}^{G}\mathcal{YD}$.
\end{proof}
\begin{cor}\label{cor:p2dividedimR}
With the notations in Lemma \ref{lem:p2dividedimR},  $p^2|\dim R$. In particular, if $\dim R=p^2$, then $R\cong S$ in ${}_{G}^{G}\mathcal{YD}$.
\end{cor}
\begin{proof}
It follows directly by Lemma \ref{lem:p2dividedimR} and \cite[Proposition 2.16]{Scha01}.
\end{proof}
\begin{rmk}\label{rmk:R-p2m-1}
  The Hopf algebra $S:=\K[x^p,y^p]$ in ${}_{G}^{G}\mathcal{YD}$ is   a usual co-commutative graded connected Hopf algebra of dimension $p^2$ classified in \cite{Hen95} (see also \cite{W1,NW}).
\end{rmk}

\begin{thm}\label{thm:p2m-nonNichols-1}
Let $\cH$ be a coradically graded  pointed Hopf algebra with $\Char\K=p$ of dimension $p^2m$ such that $\dim \cH_0=m$. Suppose that $\cH$ is not generated by group-like elements and skew-primitive elements.  Then
$\cH$ is generated by generators of $\G(\cH)$, and $x, y$, subject to the relations in $\G(\cH)$ and the following:
\begin{gather}
hx=\chi(h)xh,\qquad hy=\chi^p(h)yh, \qquad \forall h\in\G(\cH),\\
[x,y]=0,\\
x^p=0,\qquad y^p=0,
\end{gather}
where $x\in\Pp_{1,g}(\cH)$ for some $g\in\mathcal{Z}(\G(\cH))$ and $\chi\in\widehat{\G(\cH)}$ such that $\chi(g)=1$ and
\begin{align*}
\Delta(y)=y\otimes 1+g^p\otimes y+\sum_{i=1}^{p-1}\frac{(p-1)!}{i!(p-i)!}x^ig^{p-i}\otimes x^{p-i}.
\end{align*}
\end{thm}
\begin{proof}
By assumption, $\cH\cong R\sharp\cH_0$, where $R=\cH^{\co\cH_0}$ is a strictly graded Hopf algebra of dimension $p^2$  in ${}_{\cH_0}^{\cH_0}\mathcal{YD}$. Since $R$ is not a Nichols algebra, it follows that $\dim\BN(R(1))=p$. Consequently, the assertion follows by Lemma \ref{lem:p2dividedimR} and Corollary \ref{cor:p2dividedimR}.
\end{proof}

\subsection{Liftings of $\cH$}
Suppose that $\Char\K=p$ and $p\nmid m$. Let $\mathbb{H}$ be a pointed Hopf algebra with abelian coradical such that $\gr\mathbb{H}\cong\cH$ as Hopf algebras and $\mathbb{H}\cong\cH$ as coalgebras.

By Theorem \ref{thm:p2m-nonNichols-1}, $\mathbb{H}$ is generated by $h_1,h_2,\cdots,h_{t}$, $x$ and $y$, where $h_i, i\in\I_{1,t}$ generate the group $\G(\bH)$ of order $m$, $x\in\Pp_{1,g}(\bH)$ for some $g\in\mathcal{Z}(\G(\bH))$ and
\begin{align*}
\Delta(y)=y\otimes 1+g^p\otimes y+\sum_{i=1}^{p-1}\frac{(p-1)!}{i!(p-i)!}x^ig^{p-i}\otimes x^{p-i}.
\end{align*}

Now we determine defining relations of $\bH$. Set $\omega_{g}(x):=\sum_{i=1}^{p-1}\frac{(p-1)!}{i!(p-i)!}x^ig^{p-i}\otimes x^{p-i}$ for short in what follows.

\begin{lem}\label{lem:lfitingsof-hx-xh}
The relation $hx-\chi(h)h=0$ for $h\in\G(\bH)$ holds in $\bH$.
\end{lem}
\begin{proof}
If $g=1$, then we have $\Pp(\bH)\cong \K\{x\}$ as $\G(\bH)$-modules, otherwise as well-known, there is an exact sequence of $\G(\bH)$-modules:
\begin{align*}
\K\{1-g\}\hookrightarrow\Pp_{1,g}(\bH)\twoheadrightarrow \K\{x\}.
\end{align*}
By the assumption that $(m,p)=1$, $\K[\G(\bH)]$ is semisimple and so the exact sequence is split, which implies that $\Pp_{1,g}(\bH)\cong\K\{1-g\}\oplus\K\{x\}$ as $\G(\bH)$-modules. Consequently, the relation $hx-\chi(h)xh=0$ holds in $\bH$ for $h\in\G(\bH)$.
\end{proof}

\begin{lem}\label{lem:lfitingsof-hy-yh}
The relation $hy-\chi^p(h)yh=0$ for $h\in \G(\bH)$  holds in $\bH$.
\end{lem}
\begin{proof}
By Lemma \ref{lem:lfitingsof-hx-xh}, $hx-\chi(h)xh=0$ for $h\in \G(\bH)$  in $\bH$. Observe that $\Delta(h)\omega_g(x)=\chi^p(h)\omega_g(x)\Delta(h)$. Then it follows by a direct computation that
 \begin{align*}
 \Delta(hy-\chi^p(h)yh)=(hy-\chi^p(h)yh)\otimes h+g^ph\otimes(hy-\chi^p(h)yh).
 \end{align*}
 Therefore, $hy-\chi^p(h)yh\in\Pp_{h,g^ph}(\bH)\cap \bH_{p-1}$.

 Assume that $g^{p-1}\neq 1$. Then $\Pp_{h,g^ph}(\bH)=\K\{h(1-g^p)\}$ and there is a map $\lambda: \G(\bH)\rightarrow\K$ such that $hy-\chi^p(h)yh=\lambda_hh(1-g^p)$. Then consider the conjugation of $\G(\bH)$, we have $\lambda_{hk}=\chi^p(k)\lambda_h+\lambda_k$. Observe that $\G(\bH)$ is abelian. Then from $\lambda_{hk}=\lambda_{kh}$, we have $[1-\chi^p(h)]\lambda_k=[1-\chi^p(k)]\lambda_h$.
 If $\chi^p(h)\neq 1$, then  we can take $\lambda_h=0$ via the linear translation $y\mapsto y-\alpha(1-g^{p})$ satisfying $(1-\chi^p(h))\alpha=\lambda_h$. If $\chi^p(h)=1$, then
 $0=[1,y]=[h^{m_h},y]=m_h\lambda_h(1-g^p)$ and so $\lambda_h=0$.

 Assume that $g^{p-1}=1$. Then $\Pp_{h,g^ph}(\bH)=\K\{h(1-g^p),xh\}$ and there are two map $\lambda, \gamma$ from $\G(\bH)$ to $\K$ such that $hy-\chi^p(h)yh=\lambda_h xh+\gamma_h h(1-g^p)$. If $g=1$, then set $\gamma_h=0$. Then consider the conjugation of $\G(\bH)$, we have $\lambda_{hk}=\chi^p(k)\lambda_h+\chi(h)\lambda_k$ and $\gamma_{hk}=\chi^p(k)\gamma_h+\gamma_k$. From $\lambda_{hk}=\lambda_{kh}$ and $\gamma_{hk}=\gamma_{kh}$, we have $[\chi(h)-\chi^p(h)]\lambda_k=[\chi(k)-\chi^p(k)]\lambda_h$ and $[1-\chi^p(h)]\gamma_k=[1-\chi^p(k)]\gamma_h$.

Let $\xi=\chi(h)$ for short.  By induction, we have
\begin{align*}
h^ny=\xi^{np}yh^n+(n)_{\xi^{1-p}}\xi^{np}\lambda_hxh^n+(n)_{\xi^p}\gamma_h(1-g^p)h^n.
\end{align*}
Then from $[h^{m_h},y]=[1,y]=0$ for $m_h=\ord(h)$, we have
\begin{align*}
(m_h)_{\xi^{1-p}}\lambda_h=0,\quad (m_h)_{\xi^p}\gamma_h=0.
\end{align*}

If $\chi^p(h)=1$, that is, $\xi^p$=1, then we have $\gamma_h=0$ since $(m_h,p)=1$; otherwise,
 we can take
 $\gamma_h=0$ via the linear translation $y\mapsto y-\alpha(1-g^{p})$ satisfying $(1-\chi^p(h))\alpha=\lambda_h$.

 If $\chi^{p-1}(h)=1$, that is, $\xi^{1-p}=1$, then $\lambda_h=0$; otherwise we can take
 $\lambda_h=0$ via the linear translation $y\mapsto y-\alpha x$ satisfying $\alpha\chi(h)(1-\chi^{p-1}(h))=\lambda_h$. One can directly verify the translations as described above are isomorphisms of Hopf algebras.
 \end{proof}

\begin{lem}\label{lem:liftingofxp}
There exists $\lambda\in\I_{0,1}$ such that $x^p=\lambda x$ in $\bH$ with condition: $\lambda=0$ if $g^{p-1}\neq 1$ or $\chi^{p-1}\neq \epsilon$.
\end{lem}
\begin{proof}
 Since $gx=xg$ in $\bH$, by Theorem \ref{thm:q-binomial},
\begin{align*}
\Delta(x^p)=(x\otimes 1+g\otimes x)^p=x^p\otimes 1+g^{p}\otimes x^p.
\end{align*}
Therefore, $x^p\in\Pp_{1,g^{p}}(\bH)$.

Assume that $g^{p-1}\neq 1$. Then $\Pp_{1,g^p}(\bH)=\K\{1-g^p\}$ and  hence $x^p=\lambda_1(1-g^p)$ for some $\lambda_1\in\K$. If $\chi\neq \epsilon$, then consider the adjoint action of $\K[G(\bH)]$, we have $\lambda_1=0$, otherwise we can take $\lambda_1=0$ via   the translation $x\mapsto x-a(1-g)$ satisfying $a^p=\lambda_1$.

Assume that $g^{p-1}=1$. Then $\Pp_{1,g^p}(\bH)=\K\{1-g^p, x\}$ and hence $x^p=\lambda_2x+\lambda_3(1-g^p)$ for some $\lambda_2,\lambda_3\in\K$.
 If $\chi^p\neq \chi$, then consider the adjoint action of $\K[\G(\bH)]$, we have $\lambda_2=0=\lambda_3$. If $\chi^p=\chi\neq \epsilon$, then $\lambda_3=0$. If $\chi=\epsilon$, then we can take $\lambda_3=0$ via the linear translation $x \mapsto x-a(1-g)$ satisfying $a^p-\lambda_2a=\lambda_3$.
By rescaling $x$, we can take $\lambda_2\in\I_{0,1}$.
\end{proof}

\begin{lem}\label{lem:liftingsof[x,y]}
There exist $\lambda,\gamma\in\K$ such that
\begin{align*}
[x,y]=\lambda x+\gamma(1-g^{p+1}),
\end{align*}
  in $\bH$ with the conditions: $\lambda=0$ if $g\neq 1$ or $\chi\neq \epsilon$; $\gamma=0$ if $\chi^{p+1}\neq\epsilon$ or $g^{p+1}=1$. In particular, $\lambda\gamma=0$.
\end{lem}
\begin{proof}
Observe that $\Delta(x)\omega_{g}(x)=\omega_{g}(x)\Delta(x)$. Then
  \begin{gather*}
 \Delta(xy-yx)=(xy-yx)\otimes 1+g^{p+1}\otimes(xy-yx),
  \end{gather*}
that is, $xy-yx\in\Pp_{1,g^{p+1}}(\bH)$.

Assume that $g=1$. Then $\Pp_{1,g^{p+1}}(\bH)=\Pp(\bH)=\K\{x\}$ and so $[x,y]=\lambda x$ for $\lambda\in\K$.
If $\chi^p\neq\epsilon$, that is, $\chi\neq\epsilon$, then consider the adjoint action of $\K[\G(\bH)]$, we have $\lambda=0$.

Assume that $g\neq 1$.  Then  $\Pp_{1,g^{p+1}}(\bH)=\K\{1-g^{p+1}\}$ and so $[x,y]=\gamma(1-g^{p+1})$ for some $\lambda\in\K$. If $\chi^{p+1}\neq\epsilon$, then consider the adjoint action of $\K[\G(\bH)]$, then $\gamma=0$.
\end{proof}

\begin{lem}\label{lem:comm.x-and-y}
Assume the situation described as above, for $n>0$,
\begin{align*}
x^ny&=yx^n+n\lambda x^n+n\gamma x^{n-1}(1-g^{1+p}),\\
xy^n&=(y+\lambda)^nx+n\gamma y^{n-1}(1-g^{p+1}).
\end{align*}
\end{lem}
\begin{proof}
It follows by induction on $n$.
\end{proof}

The following result  generalizes the relative results in \cite[Case C]{NW}.
\begin{lem}\label{lemHJJJJ}
\begin{gather*}
(\omega_{g}(x))(\AdR  y\otimes 1+g^{ p}\otimes y)^{p-1}
=-d^1_{1,g^{ p^2}}[([x,y]x^{p-1})(\AdR y)^{p-2}].
\end{gather*}
\end{lem}
\begin{proof}
Observe that $[g,x]=0=[g,y]$ in $H$. Then
\begin{align*}
  &\omega_{g}(\AdR(y\otimes 1+g^{p}\otimes y))\\
  &=[\sum_{i=1}^{p-1}\frac{(p-1)!}{i!(p-i)!}x^ig^{p-i}\otimes x^{p-i},y\otimes 1+g^{p}\otimes y]\\
  &=[\sum_{i=1}^{p-1}\frac{(p-1)!}{i!(p-i)!}x^ig^{p-i}\otimes x^{p-i},y\otimes 1]+[\sum_{i=1}^{p-1}\frac{(p-1)!}{i!(p-i)!}x^ig^{p-i}\otimes x^{p-i},g^{ p}\otimes y]\\
  &=\sum_{i=1}^{p-1}\frac{(p-1)!}{i!(p-i)!}[x^ig^{p-i},y]\otimes x^{p-i}+\sum_{i=1}^{p-1}\frac{(p-1)!}{i!(p-i)!}x^ig^{2p-i}\otimes [x^{p-i},y]\\
  &=\sum_{i=1}^{p-1}\frac{(p-1)!}{i!(p-i)!}[x^i,y]g^{p-i}\otimes x^{p-i}+\sum_{i=1}^{p-1}\frac{(p-1)!}{i!(p-i)!}x^ig^{2p-i}\otimes [x^{p-i},y]\\
  &=\sum_{i=1}^{p-1}\frac{(p-1)!}{(i-1)!(p-i)!}[x,y]x^{i-1}g^{p-i}\otimes x^{p-i}+\sum_{i=1}^{p-1}\frac{(p-1)!}{i!(p-i-1)!}x^ig^{2p-i}\otimes [x,y]x^{p-i-1}\\
  &=\sum_{i=0}^{p-2}\left(\begin{array}{c}p-1\\i\end{array}\right)[x,y]x^{i}g^{p-i-1}\otimes x^{p-i-1}+\sum_{i=1}^{p-1}\left(\begin{array}{c}p-1 \\i\end{array}\right)x^{i}g^{2p-i}\otimes [x,y]x^{p-i-1}\\
  &=([x,y]\otimes 1)(\Delta(x^{p-1})-x^{p-1}\otimes 1)+(g^{p+1}\otimes [x,y])(\Delta(x^{p-1})-g^{p-1}\otimes x^{p-1})\\
  &=\Delta([x,y]x^{p-1})-[x,y]x^{p-1}\otimes 1-g^{2 p}\otimes [x,y]x^{p-1}\\
  &=-d^1_{1,g^{2 p}}([x,y]x^{p-1}).
  \end{align*}
Therefore, by Lemma \ref{lemHHHH}, we have that
\begin{align*}
\omega_{g}(\AdR y\otimes 1+g^{p}\otimes y)^{p-1}
&=[-d^1_{1,g^{2  p}}([x,y]x^{p-1})](\AdR (y\otimes 1+g^{  p}\otimes y))^{p-2}=\cdots\cdots\\
&=-d^1_{1,g^{  p^2}}[([x,y]x^{p-1})(\AdR y)^{p-2}].
\end{align*}
\end{proof}

\begin{lem}\label{lem:Delta(Y)}
Let $Y:=y^p-([x,y]x^{p-1})(\AdR y)^{p-2}$. Then
\begin{align*}
\Delta(Y)=Y\otimes 1+g^{p^2}\otimes Y+\omega_g(x^p).
\end{align*}
\end{lem}
\begin{proof}
By Proposition \ref{proJ},
\begin{align*}
\Delta(y^p)&=(y\otimes 1+g^p\otimes y+\omega_g(x))^p\\
&=(y\otimes 1+g^p\otimes y)^p+\omega_g(x)^p+\omega_g(x)(\AdR(y\otimes 1+g^p\otimes y))^{p-1}\\
&=y^p\otimes 1+g^{p^2}\otimes y^p+\omega_g(x^p)+\omega_g(x)(\AdR(y\otimes 1+g^p\otimes y))^{p-1}\\
&=y^p\otimes 1+g^{p^2}\otimes y^p+\omega_g(x^p)-d^1_{1,g^{ p^2}}[([x,y]x^{p-1})(\AdR y)^{p-2}].
\end{align*}
By definition of $d^1_{1,g^{p^2}}$, the assertion holds.
\end{proof}

\begin{cor}\label{cor:liftingsyp-2}
Let $Y:=y^p-([x,y]x^{p-1})(\AdR y)^{p-2}$. If  $g^{p-1}=1$ and $\chi^{p-1}=\epsilon$, then $x^p=\lambda x$ for $\lambda\in\I_{0,1}$ and $Y-\lambda^p y\in\Pp_{1,g^{p^2}}(\bH)=\Pp_{1,g}(\bH)$; otherwise $Y\in\Pp_{1,g^{p^2}}(\bH)$.
\end{cor}
\begin{proof}
If  $g^{p-1}=1$ and $\chi^{p-1}=\epsilon$, then by Lemma \ref{lem:liftingofxp}, $x^p=\lambda x$ for $\lambda\in\I_{0,1}$.  Since $\Delta(y)=y\otimes 1+g^p\otimes y+\omega_g(x)$ and $\omega_g(x^p)=\omega_g(\lambda x)$, it follows that $Y-\lambda^py\in\Pp_{1,g^{p^2}}(\bH)=\Pp_{1,g}(\bH)$. Otherwise  $x^p=0$ in $\bH$ and so $\omega_g(x^p)=0$, which implies that $Y\in\Pp_{1,g^{p^2}}(\bH)$.
\end{proof}

\begin{thm}\label{thm:liftingsofcH}
Let $\cR_G$ be the set of defining relations of $G:=\G(\bH)$ and $hx=\chi(h)xh, hy=\chi^p(h)yh$ for $h\in\G(\bH)$. Then $\bH$ is isomorphic to one of the following Hopf algebras:
\begin{itemize}
  \item[(1) ]$\fH_G^1(g,\chi):=\K\langle h_1,\cdots,h_t, x, y\rangle/(\cR_G,   [x, y],  x^p, y^p)$;
  \item[(2) ]$\fH_G^2(g,\chi):=\K\langle h_1,\cdots,h_t, x, y\rangle/(\cR_G,    [x, y],  x^p, y^p-x)$, if $\chi^{p-1}=\epsilon$ and $g^{p-1}=1$;
  \item[(3) ]$\fH_G^3(g,\chi):=\K\langle h_1,\cdots,h_t, x, y\rangle/(\cR_G,   [x, y],  x^p-x, y^p-y)$,  if $\chi^{p-1}=\epsilon$ and $g^{p-1}=1$;
 \item[(4) ]$\fH_G^4(g,\chi):=\K\langle h_1,\cdots,h_t, x, y\rangle/(\cR_G,   [x,y]-1+g^{p+1}, x^p, y^p+(1-g^{p+1})^{p-1}x-  x)$, if $g^{p-1}=1$, $\chi^{p-1}=\epsilon$, $g^{p+1}\neq 1$ and $\chi^{p+1}=\epsilon$;
 \end{itemize}
where $h_i\in\G(\bH)$ for $i\in\I_{1,t}$, $x\in\Pp_{1,g}(\bH)$ for some $g\in\mathcal{Z}(\G(\bH))$ and $\chi\in\widehat{\G(\bH)}$ such that $\chi(g)=1$ and
\begin{align*}
\Delta(y)=y\otimes 1+g^p\otimes y+\sum_{i=1}^{p-1}\frac{(p-1)!}{i!(p-i)!}x^ig^{p-i}\otimes x^{p-i}.
\end{align*}
\end{thm}
\begin{proof}
By Lemmas \ref{lem:lfitingsof-hx-xh}, \ref{lem:lfitingsof-hy-yh} and \ref{lem:liftingofxp}, the following relations hold in $\bH$:
\begin{align*}
hx-\chi(h) xh=0,\quad hy-\chi^p(h) yg=0,\quad x^p=\lambda_1 x,
\end{align*}
for $\lambda_1\in\I_{0,1}$ with conditions: $\lambda_1=0$, if $g^{p-1}\neq 1$ or $\chi^{p-1}\neq \epsilon$.

\textbf{Case 1. } Assume that $g=1$. Then by Lemma \ref{lem:liftingsof[x,y]}, $[x,y]=\lambda_2x$ for $\lambda_2\in\K$ with conditions: $\lambda_2=0$, if $g\neq 1$ or $\chi\neq\epsilon$.

Observe that $([x,y]x^{p-1})(\AdR y)^{p-2}=(\lambda_2x^{p})(\AdR y)^{p-2}=\lambda_1\lambda_2^{p-1}x$. Then by Corollary  \ref{cor:liftingsyp-2},
$y^p-\lambda_1\lambda_2^{p-1}x-\lambda_1^py\in\Pp(\bH)\cap \bH_{p-1}$ and hence $y^p-\lambda_1^py=\lambda_3 x$ for $\lambda_3\in\K$. Consider the adjoint action of $\K[\G(\bH)]$, we have $\lambda_3=0$ if $\chi^{p-1}\neq\epsilon$.

By Lemma \ref{lem:comm.x-and-y},  $x^py=yx^p$,  $xy^p=y^px+\lambda_2^px$.
Then
  the verification of $[x^p,y]=(\ad_L x)^p(y)$ and $[x,y^p]=(x)(\ad_R y)^p$ gives the conditions  $$\lambda_1\lambda_2=0=(\lambda_1-\lambda_2^{p-1})\lambda_2\Rightarrow \lambda_2=0. $$

If $\lambda_1=0$, then we can choose $\lambda_3\in \I_{0,1}$ by rescaling $x,y$ and obtain two classes  described in $(1)$--$(2)$.

If $\lambda_1=1$, then $g^{p-1}=1$,  $\chi^{p-1}=1$ and so we can take $\lambda_3=0$ via the linear translation $y\mapsto y-bx$ satisfying $b^p-b=\lambda_3$, which gives one class of $H$ described in $(3)$.

\textbf{Case 2.} Assume that $g\neq 1$. Then by Lemma \ref{lem:liftingsof[x,y]}, $[x,y]=\lambda_2(1-g^{p+1})$ for $\lambda_2\in\K$ with conditions: $\lambda_2=0$, if $g^{p+1}=1$ or $\chi^{p+1}\neq\epsilon$. Observe that $([x,y]x^{p-1})(\ad y)^{p-2}=(p-1)![x,y]^{p-1}x=(p-1)!\lambda_2^{p-1}(1-g^{p+1})^{p-1}x$.

 \textbf{Case 2a. } If $g^{p-1}=1$ and $\chi^{p-1}=\varepsilon$, then by Corollary \ref{cor:liftingsyp-2}, $y^p-([x,y]x^{p-1})(\text{ad}\;y)^{p-2}-\lambda_1^p y\in\Pp_{1,g}(H)$. Therefore,
      $$y^p-([x,y]x^{p-1})(\text{ad}\;y)^{p-2}-\lambda_1^p y=\nu_1x+\nu_2(1-g).$$

 If $\chi\neq\epsilon$, then consider the adjoint action of $\K[\G(\bH)]$, we have $\nu_2=0$; otherwise we take $\nu_2=0$ via the linear translation $y\mapsto y-a(1-g)$ satisfying $a^p-\lambda_1^pa=\nu_2$.

       Observe that $(p-1)!=-1$. The verification of $[x^p,y]=(\ad_L x)^p(y)$, $[x,y^p]=(x)(\ad_R y)^p$ and $[y,y^p]=0$ amounts to the conditions
$$\lambda_1\lambda_2(1-g^{p+1})=0=(\lambda_2^p-\lambda_2\nu_1)(1-g^{p+1}).$$

      If $\lambda_1=1$, then $\lambda_2=0$. We can take $\nu_1=0$ via the linear translation $y\mapsto y-ax$ satsifying $a^p=\nu_1$,   which gives one class  described in $(3)$.

      If $\lambda_1=0=\lambda_2$, then we take $\nu_1\in\I_{0,1}$ by rescaling $x,y$, which gives two classes   described in $(1)-(2)$.

      If $\lambda_1=0$ and $\lambda_2\neq 0$, then $g^{p+1}\neq 1$ and $\chi^{p+1}=\epsilon$. Furthermore, we take $\lambda_2=1$ via the linear translation $x\mapsto a^{-1}x,y\mapsto a^{-p}y$ satisfying $a^{p+1}=\lambda_2$, and hence $\nu_1=1$, which gives one class  described in $(4)$.

  \textbf{Case 2b. } If $g^{p-1}\neq 1$ or $\chi^{p-1}\neq 1$, then $x^p=0$ and hence $y^p-([x,y]x^{p-1})(\ad y)^{p-2}=\nu(1-g^{ p^2})$.

The verification of $[y,y^p]=0$ amounts to the condition $\lambda_2=0$. Then we can take $\nu=0$ via the linear translation $y\mapsto y-a(1-g^{p})$ satisfying $a^p=\nu$, which gives    one class of $\bH$ described in $(1)$.
\end{proof}

\begin{rmk}\label{rmk:fHk-1-3-Radfordbiproduct}
It is clear that $\{y^ix^jh,\ h\in G,i,j\in\I_{0,p-1}\}$ is a basis of $\fH_G^k(g,\chi)$ for $k\in\I_{1,3}$. One can check easily that $\pi:\fH_G^k(g,\chi)\rightarrow \K[G],\ y^ix^jh\rightarrow h$ is a bialgebra map admitting a bialgebra
section $\iota:\K[G]\rightarrow \fH_G^k(g,\chi)$ such that $\pi\circ\iota=\id$. Then $\fH^k(g,\chi)\cong R\sharp\K[G]$ for $k\in\I_{1,3}$, where $R$ is one of connected Hopf algebras of dimension $p^2$ classified in \cite[Lemma 7.3]{W1}.
\end{rmk}

\begin{rmk}
\begin{itemize}
\item[(i)] As Hopf algebras, $\fH_G^1(1,\epsilon)\cong\K[\G(\bH)]\otimes\K[x,y]/(x^p,y^p)$, $\fH_G^2(1,\epsilon)\cong\K[\G(\bH)]\otimes\K[x,y]/(x^p,y^p-x)$, $\fH_G^3(1,\epsilon)\cong\K[\G(\bH)]\otimes\K[x,y]/(x^p-x,y^p-y)$.

\item[(ii)] The Hopf subalgebra generated by $x, y$  appeared in \cite{W1} as examples of connected Hopf algebras of dimension $p^2$. Among them,  $\K[x,y]/(x^p,y^p)$ and $\K[x,y]/(x^p-x,y^p-y)$ are dual Hopf algebras of $\K[T]/(T^{p^2})$ and $\K[X]/(X^{p^2}-1)$, respectively, where $\Delta(T)=T\otimes 1+1\otimes T$ and $\Delta(X)=X\otimes X$ (see \cite[Corollary 7.5]{W1}). In particular, up to isomorphism, they are regarded as Hopf subalgebras of the algebra of distributions on $G_a$ and $G_m$, respectively (see \cite[7.8 ]{Ja}).
\end{itemize}
\end{rmk}

\begin{pro}\label{pro:non-isomorphic}
The Hopf algebras $\fH_G^1(g,\chi)$, $\fH_G^2(g,\chi)$, $\fH_G^3(g,\chi)$ and $\fH_G^4(g,\chi)$ are pairwise non-isomorphic.
\end{pro}
\begin{proof}
We first show that $\fH_G^1(g,\chi)\not\cong \fH_G^2(g',\chi')$. Suppose that there is a Hopf algebra isomorphism  $\phi:\fH_G^1(g,\chi)\rightarrow \fH_G^2(g',\chi')$, then $\phi|_{G}\in\Aut(G)$ and  by Proposition \ref{pro:R11-4.3.3} $\phi(\Pp_{1,g}(\fH_G^1(g,\chi)))=\Pp_{1,g'}(\fH_G^2(g',\chi'))$. Therefore, $\phi(g)=g'$ and $\phi(x)=\alpha x'+\beta(1-g')$ for some $\alpha\in\K^{\times},\beta\in\K$.

Applying $\phi$ to the relations $hx-\chi(h)x=0$ and $x^p=0$ in $\fH_G^1(g,\chi)$, we have $\beta=0$.

 Consider $\gr \phi:\gr\fH_G^1(g,\chi)\rightarrow\gr\fH_G^2(g',\chi')$ which is induced by $\phi$. Then there exist $\beta_i\in\K$ such that
     \begin{align*}
     (\gr\phi)(y)=\sum \beta_i g_i^{\prime}y^{\prime},\quad g_i^{\prime}\in G.
     \end{align*}
     Then from $\Delta\gr\phi(y)=(\gr\phi\otimes \gr\phi)\Delta(y)$,
     we deduce $\beta_i=0$ if $g_i^{\prime}\neq 1$.
     Hence $\gr\phi(y)=\beta_0y^{\prime}$. Then $\phi(y)=\beta_0y^{\prime}+\omega$, where $\omega$ belongs to the subalgebra generated by the first term of the coradical filtration. From $\Delta\phi(y)=(\phi\otimes \phi)\Delta(y)$, we have
     \begin{align*}
     \beta_0=\alpha^p,\quad \Delta(\omega)=\omega\otimes 1+(g')^p\otimes \omega.
     \end{align*}
      Thus  $\omega=\gamma_1 x^{\prime}+\gamma_2(1-g')$ and $\phi(y)=\alpha^py'+\gamma_1 x'+\gamma_2(1-g')$. From $y^p=0$ in $\fH_G^1(g,\chi)$,  we deduce $\alpha=0$, a contradiction. Hence $\fH_G^1(g,\chi)\not\cong \fH_G^2(g',\chi')$.

Similarly, one can check that $\fH_G^1(g,\chi)\not\cong \fH_G^3(g,\chi)$, $\fH_G^1(g,\chi)\not\cong \fH_G^4(g,\chi)$, $\fH_G^2(g,\chi)\not\cong \fH_G^3(g,\chi)$, $\fH_G^2(g,\chi)\not\cong \fH_G^4(g,\chi)$ and $\fH_G^3(g,\chi)\not\cong \fH_G^4(g,\chi)$.
\end{proof}

\begin{pro}\label{pro:isomorohism}
$\fH_G^i(g,\chi)\cong\fH_G^i(g',\chi')$ for $i\in\I_{1,4}$ if and only if there exists $f\in\Aut(G)$ such that $f(g)=g'$ and $\chi\cdot f^{-1}=\chi'$.
\end{pro}
\begin{proof}
Let $\psi: \fH_G^1(g,\chi)\rightarrow \fH_G^1(g',\chi')$ be a Hopf algebra isomorphism. Then $\psi|_G\in\Aut(G)$ and $\psi(\Pp_{1,g}(\fH_G^1(g,\chi))=\Pp_{1,g'}(\fH_G^1(g',\chi'))$. Therefore,
\begin{align*}
\psi(g)=g',\quad \psi(x)=\alpha x'+\beta(1-g'),\quad \alpha\neq 0,\beta\in\K.
\end{align*}
Since $\psi(hxh^{-1}-\chi(h)x)=0$, it follows that $\chi'\cdot\psi=\chi$ and $\beta(\chi-\epsilon)=0$. Applying $\psi$ to the relation $x^p=0$, we have $\beta=0$.

Consider $\gr \psi:\gr\fH_G^1(g,\chi)\rightarrow\gr\fH_G^1(g',\chi')$ which is induced by $\psi$. Then there exist $\beta_i\in\K$ such that
     \begin{align*}
     (\gr\phi)(y)=\sum \beta_i g_i^{\prime}y^{\prime},\quad g_i^{\prime}\in G.
     \end{align*}
     Then from $\Delta\gr\psi(y)=(\gr\psi\otimes \gr\psi)\Delta(y)$,
     we deduce $\beta_i=0$ if $g_i^{\prime}\neq 1$.
     Hence $\gr\psi(y)=\beta_0y'$. Then $\phi(y)=\beta_0y'+\omega$, where $\omega$ belongs to the subalgebra generated by the first term of the coradical filtration. From $\Delta\psi(y)=(\phi\otimes \psi)\Delta(y)$, we have
     \begin{align*}
     \beta_0=\alpha^p,\quad \Delta(\omega)=\omega\otimes 1+(g')^p\otimes \omega.
     \end{align*}
      Thus  $\omega=\gamma_1 x'+\gamma_2(1-g')$ and $\phi(y)=\alpha^py'+\gamma_1 x'+\gamma_2(1-g')$. From $hy-\chi^p(h)yh=0$ and $y^p=0$ in $\fH_G^1(g,\chi)$, we have $\gamma_2(1-g')=0$ and $(\chi^{p-1}-\epsilon)\gamma_1=0$.

      Conversely, if there exists $f\in\Aut(G)$ such that $f(g)=g'$ and $\chi\cdot f^{-1}=\chi'$. Then $f$ can extend to an isomorphism of Hopf algebras from  $\fH_G^1(g,\chi)$ to $\fH_G^1(g',\chi')$ by $f(x)=x'$ and $f(y)=y'$.

      The proof of the remaining cases follows the same line of the last case.
\end{proof}

\subsection{Classification results: diagrams are not Nichols algebras}We classify pointed Hopf algebras of dimension $p^2m$ whose diagrams are not Nichols algebras and the coradicals have dimension $m$.
\begin{lem}
Let $\fA$ be the Hopf subalgebra of $\bH$ generated by group-like elements and skew-primitive elements. Then there exists a YD-triple $(\G(\bH),g,\chi)$ and $k\in\I_{1,2}$ such that $\fA\cong\fA_{\G(\bH)}^k(g,\chi)$ (see Definition \ref{defi:fA}).
\end{lem}
\begin{proof}
By assumption and Theorem \ref{thm:liftingsofcH},  $\fA$ is a  pointed Hopf algebra of dimension $pm$ whose diagram has dimension $p$ with $p\nmid m$. Then by Theorem \ref{thm:classification-rank-one},  there exists a tuple $(\G(H),g,\chi,f)$ and $k\in\I_{1,2}$  such that $\fA\cong\fA_{\G(H)}^1(g,\chi,f)$. Since $p\nmid m=|\G(\fA)|$, by Remark \ref{rmk:important}(3), we have $f=0$. This completes the proof.
\end{proof}

\begin{pro}\label{pro:Hoch-fA}
$\dim H^2({}^1\K^{g^p},\fA)=1$ with a basis $\{\omega_g(x)\}$.
\end{pro}
\begin{proof}
 We first claim that $\omega_g(x)\neq 0$ in $H^2({}^1\K^{g^p},\fA)$. Indeed, it  is easy to see that $\omega_g(x)\in Z^2({}^1\K^{g^2},\fA)$. Suppose that $\omega_g(x)\in B^2({}^1\K^{g^p},\fA)$. Then there is an element $a\in \fA$ such that $-d^1_{1,g^p}(a)=\omega_g(x)$, that is, $\Delta(a)=a\otimes 1+g^p\otimes a+\omega_g(x)$.
 On  one hand,  $a=\sum_{i=0}^{p-1}g_ix^i$ for some $g_i\in\K[G]$ and so $\Delta(a)\in \sum_{k=0}^{p-1}\fA_k\otimes \fA_{p-1-k}$. On the other hand,  $\omega_g(x)\in \sum_{k=1}^{p-1}\fA_k\otimes \fA_{p-k}$. Therefore, $\omega_g(x)\not\in B^2({}^1\K^{g^p},\fA)$ and so $\dim H^2({}^{1}\K^{g^p},\fA)\geq 1$.

 By Remark \ref{rmk:fA-dual}(iii), $\fA\As\cong R\As\sharp(\K[G])\As$ and as algebras, $R\As\cong\K[\Z_p]$. By Proposition \ref{proDualCo},
  \begin{align*}
 1\leq\dim H^2({}^{1}\K^{g^p},\fA)=\dim H^2(\fA\As,\K)=\dim H^2(R\As\sharp(\K[G])\As,\K).
 \end{align*}

Since $(\K[G])\As$ is semisimple, it follows by \cite[Theorem 3.3 and Eq. (3.6.1)]{St}, $H^2(R\As\sharp(\K[G])\As,\K)\cong H^2(R\As,\K)^{(\K[G])\As}$, where $H^2(R\As,\K)^{(\K[G])\As}$ is the space of $(\K[G])\As)$-invariant, see \cite{St} for details. Then
  \begin{align*}
 \dim H^2(R\As\sharp(\K[G])\As,\K)=\dim H^2(R\As,\K)=\dim H^2(\Z_p,\K)=1.
 \end{align*}
 Therefore, $\dim H^2({}^{1}\K^{g^p},\fA)=1$ with a basis $\omega_g(x)$.
\end{proof}

Let $H$ be the Hopf algebra such that $\gr H\cong\cH$ in Theorem \ref{thm:p2m-nonNichols-1}. Proposition \ref{pro:Hoch-fA} shows that $\Delta(y)$ does not admit non-trivial deformations in the lifting procedure. Namely, there does not exist an element $\omega\neq\omega_g(x)$ in $H^2({}^1\K^{g^p},\fA)$ such that $\Delta_H(y)=y\otimes 1+g^p\otimes y+\omega$.

\begin{thm}\label{thm:H-group-algebra-of-dim-m}
Let $\Char\K=p$. Let $H$ be a pointed Hopf algebra with abelian coradical of dimension $p^2m$ such that the diagram has dimension $p^2$. Suppose that $H$ is not generated by group-like elements and skew-primitive elements.  Then
$H$ is isomorphic to $\fH_{\G(H)}^k(g,\chi)$ for some $k\in\I_{1,4}$ in Theorem \ref{thm:liftingsofcH}. Furthermore, $\fH_{\G(H)}^i(g,\chi)\cong\fH_{\G(H)}^i(g',\chi')$ for $i\in\I_{1,4}$ if and only if there exists $f\in\Aut(\G(H))$ such that $f(g)=g'$ and $\chi\cdot f^{-1}=\chi'$.
\end{thm}
\begin{proof}
By assumption,  it is clear that $\gr H\cong\gr\bH\cong\cH$ in Theorem \ref{thm:p2m-nonNichols-1}. By Proposition \ref{pro:Hoch-fA}, there  exists some $y\in H-\fA$ such that  $\Delta(y)=y\otimes 1+g^p\otimes y+\omega_g(x)$, that is, $H\cong\bH$ as Hopf algebras. Therefore, the assertion follows by Theorem \ref{thm:liftingsofcH}, Propositions \ref{pro:non-isomorphic} and \ref{pro:isomorohism}.
\end{proof}

\section{Main results}\label{sec:p2q}
Let $\Char\K=p>0$. In this section, we introduce our main classification results. We give the complete classification of pointed Hopf algebras of dimension $p^2q$. Besides, we classify pointed Hopf algebras of dimension $p^2m$ with abelian coradicals, where $m$ is square-free and $p\nmid m$.

\subsection{Isomorphism classes of pointed Hopf algebras of dimension $p^2q$}\label{subsec:p2q}
We give a complete list of isomorphism classes of pointed Hopf algebras over $\K$ of dimension $p^2q$.
\begin{lem}\label{lem:p2q-H0-1}
Let $H$ be a pointed Hopf algebra with $\Char\K=p$ of dimension $p^2q$ and $R$ the diagram of $H$. Then $\dim H_0=pq$ or $q$.
\end{lem}
\begin{proof}
Let $R$ be the diagram of $H$ and $V:=R(1)\cong\Pp(R)$. By Nichols-Zoeller Theorem \cite{NZ},  $\dim H_0\mid\dim H$ and $\dim H=\dim R\dim H_0$. Thus we have the following possibilities:
\begin{enumerate}
  \item $\dim H_0=p^2$. Then $|\G(H)|=p^2$ and $\dim R=q$, which implies that $\dim V=1$. Let $V:=\K\{x\}$. Since $\widehat{\G(H)}=\{\epsilon\}$, $x\in V^{\epsilon}$ and hence  $c(x\otimes x)=x\otimes x$. Therefore, $R$ contains a braided Hopf algebra $\K[x]/(x^p)$ of dimension $p$, a contradiction.

  \item $\dim H_0=p$. Then $\G(H)\cong \Z_p$ and $\dim R=pq$.  It is impossible. Indeed,  if $\dim V=1$, then $\dim\BN(V)=p$ and hence $R$ is not a Nichols algebra; by Lemma \ref{lem:p2dividedimR}, $p^2\mid \dim R$, a contradiction. If $\dim V\geq 2$,  then from the proof of  \cite[Lemma 3.1]{X23}, we have $p^2\mid \dim R$, a contradiction. Indeed, there must be a  two-dimensional subobject $W\subset V$, which is of diagonal type with trivial braiding or of Jordan type, which implies that $\dim\BN(W)\mid\dim\BN(V)$.
    \item $\dim H_0=pq$. Then $|\G(H)|=pq$ and $\dim R=p$. Furthermore, $R\cong\K[x]/(x^p)$.
    \item $\dim H_0=q$. Then  $|\G(H)|=q$  and $\dim R=p^2$.
\end{enumerate}
\end{proof}

\begin{pro}\label{pro:p2q-1}
Let $H$ be a pointed Hopf algebra over $\K$ of dimension $p^2q$ whose diagram $R$ has dimension $p$. Then $H$ is isomorphic to one of the following Hopf algebras:
\begin{description}
  \item[(1)]$\K[\Z_{pq}]\otimes\K[x]/(x^p)$, with   $x\in\Pp(H)$;
  \item[(2)]$\K[\Z_{pq}]\otimes\K[x]/(x^p-x)$, with  $x\in\Pp(H)$;
  \item[(3)]$\K\langle g,x\rangle/(g^{pq}-1, gx-\xi xg,x^p)$, $\xi$ a primitive $q$th root of unity, with $g\in\G(H)$, $x\in\Pp(H)$;
  \item[(4)]$\K\langle g,x\rangle/(g^{pq}-1, gx-\xi xg,x^p-x)$, $q\mid p-1$, with $g\in\G(H)$, $x\in\Pp(H)$;

 \item[(5)]$\K\langle g,x\rangle/(g^{pq}-1, gx-\xi xg,x^p)$, with $g\in\G(H)$, $x\in\Pp_{1,g^q}(H)$;

  \item[(6)]$\K[g,x]/(g^{pq}-1, x^p)$, with $g\in\G(H)$, $x\in\Pp_{1,g}(H)$;
  \item[(7)]$\K\langle g,x\rangle/(g^{pq}-1, [g,x]-g+g^2,x^p-x)$, with $g\in\G(H)$, $x\in\Pp_{1,g}(H)$;

  \item[(8)]$\K[g,x]/(g^{pq}-1, x^p)$, with $g\in\G(H)$, $x\in\Pp_{1,g^q}(H)$;
\item[(9)]$\K\langle g,x\rangle/(g^{pq}-1, [g,x]-g+g^{q+1},x^p-q^{p-1}x)$, with $g\in\G(H)$, $x\in\Pp_{1,g^q}(H)$;

\item [(10)]$\K[g,x]/(g^{pq}-1, x^p)$, with $g\in\G(H)$, $x\in\Pp_{1,g^p}(H)$;
\item [(11)]$\K[g,x]/(g^{pq}-1, x^p-x)$, $q\mid p-1$, with $g\in\G(H)$, $x\in\Pp_{1,g^p}(H)$;
\item[(12)]$\K\langle g,x\rangle/(g^{pq}-1, [g,x]-g+g^{p+1},x^p-x)$, $q\mid p-1$, with $g\in\G(H)$, $x\in\Pp_{1,g^p}(H)$;

\item[(13)]$\K[\Z_p\rtimes \Z_q]\otimes \K[x]/(x^p) $, with $x\in\Pp(H)$;
\item[(14)] $\K[\Z_p\rtimes \Z_q]\otimes \K[x]/(x^p-x) $, with $x\in\Pp(H)$;
\item[(15)]$\K\langle g,h,x\rangle/(g^{q}-1, h^p-1,ghg^{-1}-h^t,gx-\xi xg,hx-xh,x^p)$, with $g,h\in \G(H)$, $x\in\Pp(H)$;
\item[(16)]$\K\langle g,h,x\rangle/(g^{q}-1, h^p-1,ghg^{-1}-h^t,gx-\xi xg,hx-xh,x^p-x) $, with $g,h\in \G(H)$, $x\in\Pp(H)$;
\item[(17)]$\K[\Z_q\rtimes \Z_p]\otimes \K[x]/(x^p)$, with $x\in\Pp(H)$;
\item[(18)]$\K[\Z_q\rtimes \Z_p]\otimes \K[x]/(x^p-x)$, with $x\in\Pp(H)$.
\end{description}
\end{pro}
\begin{proof}
Let $V:=R(1)$. By assumption, $R\cong\K[x]/(x^p)$ with $x\in V_{h}^{\tau}$, where $g\in\mathcal{Z}(\G(H))$ and $\tau\in\widehat{\G(H)}$ such that $\tau(h)=1$.
Furthermore, $\G(H)$ is isomorphic either to $\Z_{pq}$, $\Z_p\rtimes \Z_q$, $\Z_q\rtimes  \Z_p$. Therefore, by Theorem \ref{thm:classification-rank-one},  there exists a tuple $(\G(H),h,\tau,f)$  such that $H\cong\fA_{\G(H)}^k(h,\tau,f)$ for $k\in\I_{1,2}$ or $\fA_{\G(H)}^3(h,f)$. Now we determine isomorphism classes.

Assume that $\G(H)\cong \Z_{pq}:=\langle g\rangle$.  Observe that $\mathcal{Z}(\G(H))=\Z_{pq}$ and $\widehat{\G(H)}:=\langle \chi\rangle$, where $\chi(g)=\xi$ is a primitive $q$th root of unity.  Hence $x\in V_{g^{i}}^{\chi^j}$ for some   $i\in \I_{0,pq-1},j\in \I_{0,q-1}$ such that  $\chi^j(g^i)=\xi^{ij}=1$, which implies that $q\mid ij$, that is, $q\mid i$ or $j=0$.  If $j=0$, then by changing the generator of $\Z_{pq}$, we can take $i\in\{0,1,p,q\}$.   If $i=0$, then we can take $j\in\I_{0,1}$. If $i\neq 0,j\neq 0$, then $q\mid i$;   we can take $i=q$.  Therefore, up to isomorphism, we may restrict to the following realization of $R$:
\begin{gather*}
(i,j)=\{(0,0),\; (0,1),\; (1,0),\; (p,0),\; (q,0),\;(q,j), j\in\I_{1,q-1}\}.
\end{gather*}

If $(i,j)\in\{(0,0),\; (0,1)\}$, then $h=1$ and hence $H\cong\fA_{\Z_{pq}}^k(1,\chi^j)$ for $k\in\I_{1,2}$, which are described in $(1)$--$(4)$.

If $(i,j)\in\{(q,j),\  j\in\I_{1,q-1}\}$, then $\ord(h)=p$, $\tau\neq\epsilon$. Since $\chi^{j}(g)$ is a primitive $q$th root of unity, it follows by Remark \ref{rmk:important} (1), $f(h)=f(g^q)=0$. Hence  $H\cong\fA_{\Z_{pq}}^1(g^q,\chi^j,f)$. Observe that the value of $f(k)$ for any $k\in\Z_{pq}$ is determined by $f(g)$. Then by Proposition \ref{pro:isomorohism-rank-1}, we have $\fA_{\Z_{pq}}^1(g^q,\chi^j,f)\cong \fA_{\Z_{pq}}^1(g^q,\chi^j)$ via the linear translation $x\mapsto x+\alpha(1-g^q)$ satisfying $f(g)=\alpha(\chi^j-\epsilon)(g)$, which gives the one class descried in $(5)$.

If $(i,j)=(1,0)$, then $\tau=\epsilon$, $\ord(h)=pq$. From which, we have $h^{p-1}\neq 1$ and  $h^{p}\neq 1$. Therefore,  $H\cong\fA_{\Z_{pq}}^1(g,\epsilon)$ or $\fA_{\Z_{pq}}^3(g,f)$ with $f(g)=1$, which are described in $(6)$--$(7)$.

If $(i,j)=(q,0)$, then $\tau=\epsilon$, $\ord(h)=p$. By Remark \ref{rmk:important} (2), we have $f=0$ when $f(g^q)=0$. Consequently,  $H\cong\fA_{\Z_{pq}}^1(g^q,\epsilon)$ or $\fA_{\Z_{pq}}^3(g^q,f)$ with $f(g)=1$, which are described in $(8)$--$(9)$.

 If $(i,j)=(p,0)$, then $\ord(h)=q$, and $\tau=\epsilon$. From which, we have $p\nmid\ord(h)$ and  $h^{p}\neq 1$. Therefore, $H\cong\fA_{\Z_{pq}}^k(g^p,\epsilon)$ for $k\in\I_{1,2}$ or $\fA_{\Z_{pq}}^2(g^p,\epsilon, f)$ with $f(g)=1$, which are described in $(10)$--$(12)$. From the commutativity, the Hopf algebras described in $(11)$ and $(12)$ are non-isomorphic.

Assume that $\G(H)=\Z_p\rtimes \Z_q$.  Since the center of $\G(H)$ is trivial,  $x\in V_{1}^{\chi^i}$ for some $i\in\I_{0,q-1}$.   Up to change the character $\chi$, we take $i\in\I_{0,1}$. Therefore, $H\cong\fA_{\Z_p\rtimes \Z_q}^k(1,\chi^i)$ for $k\in\I_{1,2}$, which are described in $(13)$--$(16)$.

Assume that $\G(H)\cong\Z_q\rtimes \Z_p$. Since the center of $\G(H)$ and $\widehat{\G(H)}$ is trivial,  $x\in V_{1}^{\epsilon}$. Therefore, $H\cong\fA_{\Z_q\rtimes \Z_p}^k(1,\epsilon)$ for $k\in\I_{1,2}$, which are described in $(17)$--$(18)$.
\end{proof}

\begin{pro}\label{pro:p2q-rank-2-liftings}
Let $H$ be a pointed Hopf algebra over $\K$ of dimension $p^2q$ such that the diagram $R$ has dimension $p^2$. Then $H$ is isomorphic to one of the following Hopf algebras:
\begin{description}
  \item[(1)] $\K[\Z_q]\otimes\K[x,y]/(x^p,y^p)$, with  $x,y\in\Pp(H)$;
  \item[(2)] $\K[\Z_q]\otimes\K[x,y]/(x^p-x,y^p)$, with   $x,y\in\Pp(H)$;
  \item[(3)] $\K[\Z_q]\otimes\K[x,y]/(x^p-y,y^p)$, with  $x,y\in\Pp(H)$;
  \item[(4)] $\K[\Z_q]\otimes\K[x,y]/(x^p-x,y^p-y)$, with   $x,y\in\Pp(H)$;
  \item[(5)] $\K[\Z_q]\otimes\K\langle x,y\rangle/(x^p-x,y^p, [x,y]-y)$, with $g\in\G(H)$, $x,y\in\Pp(H)$;

  \item[(6)] $\K\langle g,x,y\rangle/(g^q-1,gx-xg, gy-\xi yg,x^p ,y^p, [x,y]),$ with $g\in\G(H)$, $x,y\in\Pp(H)$;
  \item[(7)] $\K\langle g,x,y\rangle/(g^q-1,gx-xg, gy-\xi yg,x^p-x,y^p, [x,y]),$ with $g\in\G(H)$, $x,y\in\Pp(H)$;
  \item[(8)] $\K\langle g,x,y\rangle/(g^q-1,gx-xg, gy-\xi yg,x^p ,y^p-y, [x,y]),$ $q\mid p-1$, with $g\in\G(H)$, $x,y\in\Pp(H)$;
  \item[(9)] $\K\langle g,x,y\rangle/(g^q-1,gx-xg, gy-\xi yg,x^p-x,y^p-y, [x,y]),$    $q\mid p-1$, with $g\in\G(H)$, $x,y\in\Pp(H)$;
  \item[(10)] $\K\langle g,x,y\rangle/(g^q-1,gx-xg, gy-\xi yg,x^p-x,y^p, [x,y]-y),$ with $g\in\G(H)$, $x,y\in\Pp(H)$;

  \item[(11)] $\K\langle g,x,y\rangle/(g^q-1,gx-\xi xg, gy-\xi yg,x^p,y^p, [x,y]),$ with $g\in\G(H)$, $x,y\in\Pp(H)$;
  \item[(12)] $\K\langle g,x,y\rangle/(g^q-1,gx-\xi xg, gy-\xi yg,x^p-x,y^p, [x,y]),$ $q\mid p-1$, with $g\in\G(H)$, $x,y\in\Pp(H)$;
  \item[(13)] $\K\langle g,x,y\rangle/(g^q-1,gx-\xi xg, gy-\xi yg,x^p-y,y^p, [x,y]),$ $q\mid p-1$, with $g\in\G(H)$, $x,y\in\Pp(H)$;
  \item[(14)] $\K\langle g,x,y\rangle/(g^q-1,gx-\xi xg, gy-\xi yg,x^p-x,y^p-y, [x,y]),$ $q\mid p-1$, with $g\in\G(H)$, $x,y\in\Pp(H)$;
  \item[(15)] $\K\langle g,x,y\rangle/(g^q-1,gx-\xi xg, gy-\xi^{\nu} yg,x^p,y^p, [x,y]),$ $\nu\in\I_{2,q-1}$, with $g\in\G(H)$, $x,y\in\Pp(H)$;

  \item[(16)] $\K\langle g,x,y\rangle/(g^q-1,gx-\xi xg, gy-\xi^{\nu} yg,x^p-x,y^p, [x,y]),$ $q\mid p-1$, $\nu\in\I_{2,q-1}$, with $g\in\G(H)$, $x,y\in\Pp(H)$;
  \item[(17)] $\K\langle g,x,y\rangle/(g^q-1,gx-\xi xg, gy-\xi^{\nu} yg,x^p,y^p-y, [x,y]),$ $q\mid p-1$, $\nu\in\I_{2,q-2}$, with $g\in\G(H)$, $x,y\in\Pp(H)$;
   \item[(18)] $\K\langle g,x,y\rangle/(g^q-1,gx-\xi xg, gy-\xi^{\nu} yg,x^p-y,y^p, [x,y]),$ $q\mid p-\nu$, $\nu\in\I_{2,q-1}$, with $g\in\G(H)$, $x,y\in\Pp(H)$;
    \item[(19)] $\K\langle g,x,y\rangle/(g^q-1,gx-\xi xg, gy-\xi^{\nu} yg,x^p,y^p-x, [x,y]),$ $q\mid p\nu-1$, $\nu\in\I_{2,q-2}$, with $g\in\G(H)$, $x,y\in\Pp(H)$;
    \item[(20)] $\K\langle g,x,y\rangle/(g^q-1,gx-\xi xg, gy-\xi^{\nu} yg,x^p-x,y^p-y, [x,y]),$ $q\mid p-1$, $\nu\in\I_{2,q-1}$, with $g\in\G(H)$, $x,y\in\Pp(H)$;
   \item[(21)] $\K\langle g,x,y\rangle/(g^q-1,gx-\xi xg, gy-\xi^{-1} yg,x^p-y,y^p-x, [x,y]),$ $q\nmid p-1$, $q\mid p+1$, with $g\in\G(H)$, $x,y\in\Pp(H)$;

  \item[(22)] $\K[g,x,y]/(g^q-1,x^p,y^p)$, with $g\in G(H)$, $x\in\Pp(H)$, $y\in\Pp_{1,g}(H)$;
  \item[(23)] $\K[g,x,y]/(g^q-1,x^p-x,y^p)$, with $g\in G(H)$, $x\in\Pp(H)$, $y\in\Pp_{1,g}(H)$;
  \item[(24)] $\K[g,x,y]/(g^q-1,x^p,y^p-y)$,$q\mid p-1$, with $g\in G(H)$, $x\in\Pp(H)$, $y\in\Pp_{1,g}(H)$;
  \item[(25)] $\K[g,x,y]/(g^q-1,x^p-x,y^p-y)$, $q\mid p-1$, with $g\in G(H)$, $x\in\Pp(H)$, $y\in\Pp_{1,g}(H)$;
  \item[(26)] $\K\langle g,x,y\rangle/(g^q-1,gx-xg, gy-yg,x^p-x,y^p, [x,y]-y),$ with $g\in G(H)$, $x\in\Pp(H)$, $y\in\Pp_{1,g}(H)$;
   \item[(27)] $\K\langle g,x,y\rangle/(g^q-1,gx-xg, gy-yg,x^p,y^p, [x,y]-1+g),$ with $g\in G(H)$, $x\in\Pp(H)$, $y\in\Pp_{1,g}(H)$;
  \item[(28)] $\K\langle g,x,y\rangle/(g^q-1,gx-xg, gy-yg,x^p-x,y^p, [x,y]-y-1+g),$ with $g\in G(H)$, $x\in\Pp(H)$, $y\in\Pp_{1,g}(H)$;

 \item[(29)] $\K[g, x, y]/(g^q-1, x^p, y^p)$, with $g\in G(H)$, $x,y\in\Pp_{1,g}(H)$;
    \item[(30)] $\K[g, x, y]/(g^q-1, x^p-x, y^p)$,  $q\mid p-1$, with $g\in G(H)$, $x,y\in\Pp_{1,g}(H)$;
    \item[(31)] $\K[g, x, y]/(g^q-1,  x^p-y, y^p)$,  $q\mid p-1$, with $g\in G(H)$, $x,y\in\Pp_{1,g}(H)$;
    \item[(32)] $\K[g, x, y]/(g^q-1, x^p-x, y^p-y)$,  $q\mid p-1$, with $g\in G(H)$, $x,y\in\Pp_{1,g}(H)$;
    \item[(33)] $\K\langle g, x, y\rangle/(g^q-1, gx-xg,  gy-yg, x^p, y^p,  [x, y]-1+g^2), q\neq 2$,
    with $g\in G(H)$, $x,y\in\Pp_{1,g}(H)$.

  \item[(34)] $\K[g,x,y]/(g^q-1,x^p,y^p)$, with $g\in G(H)$, $x\in\Pp_{1,g}(H)$, $y\in\Pp_{1,g^{\nu}}(H), \nu\in\I_{2,q-1}$;
  \item[(35)] $\K[g,x,y]/(g^q-1,x^p-x,y^p)$, $q\mid p-1$, with $g\in G(H)$, $x\in\Pp_{1,g}(H)$, $y\in\Pp_{1,g^{\nu}}(H), \nu\in\I_{2,q-1}$;
  \item[(36)] $\K[g,x,y]/(g^q-1,x^p,y^p-y)$, $q\mid p-1$, with $g\in G(H)$, $x\in\Pp_{1,g}(H)$, $y\in\Pp_{1,g^{\nu}}(H), \nu\in\I_{2,q-2}$;
  \item[(37)] $\K[g,x,y]/(g^q-1,x^p-y,y^p)$, $q\mid p-\nu$, with $g\in G(H)$, $x\in\Pp_{1,g}(H)$, $y\in\Pp_{1,g^{\nu}}(H), \nu\in\I_{2,q-1}$;
  \item[(38)] $\K[g,x,y]/(g^q-1,x^p,y^p-x)$, $q\mid p\nu-1$, with $g\in G(H)$, $x\in\Pp_{1,g}(H)$, $y\in\Pp_{1,g^{\nu}}(H), \nu\in\I_{2,q-2}$;

        \item[(39)] $\K[g,x,y]/(g^q-1,x^p-x,y^p-y)$, $q\mid p-1$, with $g\in G(H)$, $x\in\Pp_{1,g}(H)$, $y\in\Pp_{1,g^{\nu}}(H), \nu\in\I_{2,q-1}$;
  \item[(40)] $\K\langle g,x,y\rangle/(g^q-1,gx-xg, gy-yg,x^p,y^p, [x,y]-1+g^{\nu+1})$,    with $g\in G(H)$, $x\in\Pp_{1,g}(H)$, $y\in\Pp_{1,g^{\nu}}(H), \nu\in\I_{2,q-2}$;
  \item[(41)] $\K[ g,x,y]/(g^q-1,x^p-y,y^p-x)$,  $q\nmid p-1$, $q\mid p+1$,   with $g\in G(H)$, $x\in\Pp_{1,g}(H)$, $y\in\Pp_{1,g^{-1}}(H)$.


\end{description}
\end{pro}

\begin{proof}
By assumption, $\G(H)\cong\Z_q$ with generator $g$.  Let  $\widehat{\G(H)}\cong \langle\chi\rangle$, such that $\chi(g)=\xi$ is a primitive $q$th root of unity. Then by Lemma \ref{lem:p2m-H0-2},  $R\cong\BN(V,\D)$ for some QPYD-datum $\D:=\D(\Z_q, \chi^i,\chi^l, g^j,g^k)$ with $i,j,k,l\in\I_{0,q-1}$. Since $\chi^i(g^j)=1=\chi^l(g^k)=\chi^i(g^k)\chi^l(g^j)$, it follows that
  $q\mid ij$, $q\mid kl$ and $q\mid kj+li$. Therefore, $ij=0$, $kl=0$ and $kj+li=0$. Observe that $\Aut(\Z_q)\cong\Z_{q-1}$. Let $\Omega:=\{(i,l,j,k)\mid ij=0,kl=0,kj+li=0,i,j,j,l\in\I_{0,q-1}\}$. Then consider the action of $\Aut(\Z_q)\times S_2$ on $\Omega$ defined by Proposition \ref{pro:isomorphism-rank-2},
   we restrict to consider the cases:
\begin{enumerate}
  \item[(1)] $V=V_{1}^{\epsilon}$.
  \item[(2)]$V=V_{1}^{\epsilon}\oplus V_{1}^{\chi}$.
  \item[(3)]$V=V_{1}^{\chi}\oplus V_{1}^{\chi^\nu},\ \nu\in\I_{1,q-1}$.

  \item[(4)] $V=V_{1}^{\epsilon}\oplus V_{g}^{\epsilon}.$
  \item[(5)]$V= V_{g}^{\epsilon}\oplus V_{g^{\nu}}^{\epsilon},\ \nu\in\I_{1,q-1}.$
\end{enumerate}

\textbf{Case (1) }   By Proposition \ref{pro:liftings-dH-dimP(dH)=2}, $H\cong\dH^k(\Z_q,\epsilon,\epsilon,1,1)$ for $k\in\I_{1,5}$, which are the non-isomorphic classes  described in $(1)$--$(5)$.

  \textbf{Case (2) }  By Proposition \ref{pro:liftings-dH-dimP(dH)=2}, $H\cong\dH^k(\Z_q,\chi,\epsilon,1,1)$ for $k\in\{1,2,4\}$ or  $H\cong\dH^k(\Z_q,\epsilon,\chi,1,1)$ for $k\in\{1,2,4,5\}$. It is clear that $\dH^k(\Z_q,\epsilon,\chi,1,1)\cong \dH^k(\Z_q,\chi,\epsilon,1,1)$ for $k\in\{1,4\}$ by exchanging $x$ with $y$. Hence we obtain five classes, which are described in $(6)$-$(10)$.

  We claim that the Hopf algebras described in $(6)$-$(10)$ are pairwise non-isomorphic. Indeed, by Proposition \ref{pro:non-isomorphic-rank-2}, it remains to show that the Hopf algebras described in $(7)$-$(8)$ are non-isomorphic.   Denote by $\cH_{7}$ and $\cH_{8}$  the Hopf algebras described in $(7)$--$(8)$ respectively. If there is an isomorphism of Hopf algebras $\phi: \cH_{7}\rightarrow\cH_{8}$, then by Proposition \ref{pro:R11-4.3.3} $\phi(\Pp(\cH_{7}))=\Pp(\cH_{8})$ and hence we have
 \begin{align*}
 \phi(x)=\alpha_1x^{\prime}+\alpha_2y^{\prime}, \quad \phi(y)=\beta_1x^{\prime}+\beta_2y^{\prime},\quad \alpha_1\beta_2-\alpha_2\beta_1\neq 0.
 \end{align*}
 From $gx=xg$, $gy=\xi yg$ in $\cH_{7}$, we deduce that $\alpha_2=0=\beta_1$. From $x^p=x$ in $\cH_{7}$, we have that $\alpha_1=0$, a contradiction. Hence $\cH_{7}\not\cong\cH_{8}$.

\textbf{Case (3).} Suppose that  $\nu=1$. Then by Proposition \ref{pro:liftings-dH-dimP(dH)=2}, $H\cong\dH^k(\Z_q,\chi,\chi,1,1)$ for $k\in\I_{1,4}$, which are described in $(11)$--$(14)$.

   Suppose that $\nu\neq 1$. Then by Proposition \ref{pro:liftings-dH-dimP(dH)=2}, $H\cong\dH^k(\Z_q,\chi,\chi^{\nu},1,1)$ or $\dH^k(\Z_q,\chi^{\nu},\chi,1,1)$ for $k\in\{1,2,3,4,8\}$. It is clear that $\dH^k(\Z_q,\chi,\chi^{\nu},1,1)\cong \dH^k(\Z_q,\chi^{\nu},\chi,1,1)$ for $k\in\{1,4,8\}$ by exchanging $x$ with $y$. Furthermore, if $\dH\cong\dH^8(\Z_q,\chi^{\nu},\chi,1,1)$, then from the conditions: $\chi^{p\nu}=\chi$, $\chi^p=\chi^{\nu}$ and $\chi^{\nu}\neq\chi$, we see that $q\nmid p-1$, $q\mid p+1$ and $\nu=q-1$.  Hence we obtain seven classes, which are described in $(15)$-$(21)$.

     Similar to Case $(2)$, the Hopf algebras described in $(15)$-$(21)$ are pairwise non-isomorphic, except for the following two cases:
     \begin{itemize}
       \item the Hopf algebras described in $(16)$ and $(17)$ are isomorphic if and only if $\nu=q-1$. Indeed, if there is a Hopf algebra isomorphism $\phi$  then by Proposition \ref{pro:R11-4.3.3}, there are  $\alpha_1,\alpha_2,\beta_1,\beta_2$ such that
         \begin{align*}
         \phi(x)=\alpha_1x^{\prime}+\alpha_2y^{\prime},\quad \phi(y)=\beta_1x^{\prime}+\beta_2y^{\prime}.
         \end{align*}
         Applying $\phi$ to the relations $x^p-x=0,y^p=0,[x,y]=0$, it follows that $\alpha_1=\beta_2=0$~and~$\alpha_2^p-\alpha_2=0$. Then applying $\phi$ to the rest of the defining relations, we have $\xi^{\nu^2-1}=1$, that is, $\nu=q-1$. If $\nu=q-1$, then by swapping~$x$~and~$y$, the Hopf algebras described in $(16)$ and $(17)$ are isomorphic.
       \item if $q\mid p-\nu$~and~$q\mid p\nu-1$, the Hopf algebras described in $(18)$ and $(19)$ are isomorphic if and only if  $\nu= q-1$.
     \end{itemize}

\textbf{Case (4).} By Proposition \ref{pro:liftings-dimP(dH)=1}, $H\cong\dH^k(\Z_q,\epsilon,\epsilon,g,1)$ for $k\in\I_{1,7}-\{3,5,7\}$ or $H\cong\dH^k(\Z_q,\epsilon,\epsilon,1,g)$ for $k\in\I_{1,7}-\{3\}$. It is clear that $\dH^k(\Z_q,\epsilon,\chi,g,1)\cong \dH^k(\Z_q,\chi,\epsilon,1,g)$ for $k\in\{1,4,6\}$ by exchanging $x$ with $y$. Hence we obtain seven classes, which are described in $(22)$--$(28)$.

It is clear that Hopf algebras described in $(23)$--$(24)$ are pairwise non-isomorphic. Then by Proposition \ref{pro:non-isomorphic-rank-2}, the Hopf algebras described in $(22)$--$(28)$ are pairwise non-isomorphic.

\textbf{Case (5).}
Suppose that $\nu=1$. Then by Proposition \ref{pro:liftings-dimP1g=2}, $H\cong\dH^k(\Z_q,\epsilon,\epsilon,g,g)$ for $k\in\I_{1,6}-\{5\}$, which are described in $(29)$-$(33)$.

Suppose that $\nu\neq  1$. Then  by Proposition \ref{pro:liftings-dimP1g=1}, $H\cong\dH^k(\Z_q,\epsilon,\epsilon,g,g^{\nu})$ or $H\cong\dH^k(\Z_q,\epsilon,\epsilon,g^{\nu},g)$ for $k\in\I_{1,8}-\{5,7\}$. It is clear that $\dH^k(\Z_q,\epsilon,\epsilon,g,g^{\nu})\cong \dH^k(\Z_q,\epsilon,\epsilon,g^{\nu},g)$ for $k\in\{1,4,6,8\}$ by exchanging $x$ with $y$. Furthermore, if $\dH\cong\dH^8(\Z_q,\epsilon,\epsilon,g,g^{\nu})$, then from the conditions: $g^{p\nu}=g$, $g^p=g^{\nu}$ and $g^{\nu}\neq g$, we see that $q\nmid p-1$, $\nu=q-1$ and $q\mid p+1$. Hence we obtain the classes described in $(34)$--$(41)$.

 Following the same lines of the last case,   one can show that the Hopf algebras described in $(35)$ and $(36)$ are isomorphic if $\nu=q-1$; if $q\mid p-\nu$ and $q\mid p\nu-1$, then the Hopf algebras described in $(37)$ and $(38)$ are isomorphic if and only if $\nu= q-1$; otherwise the Hopf algebras described in $(34)$--$(41)$ are pairwise non-isomorphic.
 \end{proof}
\begin{rmk}
In Proposition \ref{pro:p2q-rank-2-liftings}, the Hopf algebras described in $(26)$--$(27)$, $(33)$ and $(40)$ are noncommutative and noncocommutative, which constitute new examples of non-commutative non-cocommutative pointed Hopf algebras. By Remark \ref{rmk:dHk-1-5-Radfordbiproduct}, the rest are Radford biproducts of restricted universal enveloping algebras of dimension $p^2$ by $\K[\Z_q]$.
\end{rmk}

\begin{pro}\label{pro:p2q-2}
Let $H$ be a pointed Hopf algebra over $\K$ of dimension $p^2q$ whose diagram $R$ is not a Nichols algebra. Then $H$ is isomorphic to one of the following Hopf algebras
\begin{itemize}
  \item[(1) ]$ \K[\Z_q]\otimes \K[x,y]/(x^p,y^p)$, with $g=1$;
  \item[(2) ]$ \K[\Z_q]\otimes \K[x,y]/(x^p,y^p-x)$, with $g=1$;
  \item[(3) ]$ \K[\Z_q]\otimes \K[x,y]/(x^p-x,y^p-y)$, with $g=1$;
  \item[(4) ]$\K\langle h,x,y\rangle/(h^q-1, hx-\xi xh, hy-\xi^pyh,[x,y], x^p,y^p)$, with $g=1$;
  \item[(5) ]$\K\langle h,x,y\rangle/(h^q-1, hx-\xi xh, hy-\xi^pyh,[x,y], x^p,y^p-x)$, $q\mid p-1$, with $g=1$;
  \item[(6) ]$\K\langle h,x,y\rangle/(h^q-1, hx-\xi xh, hy-\xi^pyh,[x,y], x^p-x,y^p-y)$, $q\mid p-1$, with $g=1$;

  \item[(7) ]$\K[h,x,y]/(h^q-1,x^p,y^p)$, with $g=h$;
  \item[(8) ]$\K[h,x,y]/(h^q-1,x^p,y^p-x)$, $q\mid p-1$, with $g=h$;
   \item[(9) ]$\K[h,x,y]/(h^q-1,x^p-x,y^p-y)$, $q\mid p-1$, with $g=h$;
  \item[(10) ]$\K\langle h,x,y\rangle/(h^q-1, [h,x], [h,y], [x,y]-1+h^{p+1},x^p, y^p+(1-h^{p+1})^{p-1}x-x)$, $q\mid p-1$, $q\nmid p+1$, with $g=h$;
 \end{itemize}
  where $\xi$ is a primitive $q$th root of unity, $h\in\G(H)$, $x\in\Pp_{1,g}(H)$ and $\Delta(y)=y\otimes 1+g^p\otimes y+\omega_g(x)$.

\end{pro}
\begin{proof}
By assumption and Lemma \ref{lem:p2q-H0-1}, $\G(H)\cong \Z_q$ with generator $h$. Then $\dim\BN(R(1))=p$ and hence $\dim R(1)=1$. Let $V:=R(1)=\K\{x\}$.
   Then $\K\{x\}\in{}^{\G(H)}_{\G(H)}\mathcal{YD}$ by
 \begin{itemize}
 \item $x\in V^{\chi^i}_{h^j}$ for $i,j\in\I_{0,q-1}$, where $\chi\in \widehat{\G}$ satisfying $\chi(h)=\xi$.
 \end{itemize}
 Hence $\xi^{ij}=1$, that is, $q\mid ij$, which implies that $i=0$ or $j=0$. Up to change the character $\chi$ and the element $h$
  we may restrict to consider the cases: $(i,j)=(0,0),(1,0),(0,1)$.
Then by Theorem \ref{thm:H-group-algebra-of-dim-m}, $H$ is isomorphic to $\fH_{\Z_q}^k(1,\epsilon)$, $\fH_{\Z_q}^k(1,\chi)$ for $k\in\I_{1,3}$ or $\fH_{\Z_q}^k(g,\epsilon)$ for $k\in\I_{1,4}$.
\end{proof}
\begin{rmk}
By Propositions \ref{pro:non-isomorphic}, \ref{pro:isomorohism} and  \ref{pro:p2q-2},  there are 10 isomorphism classes of pointed Hopf algebras of dimension $p^2q$  whose  diagrams are not Nichols algebras. By Remark \ref{rmk:fHk-1-3-Radfordbiproduct}, the classes described in $(1)$-$(9)$ are Radford biproducts of connected Hopf algebras of dimension $p^2$ \cite{W1} by $\K[\Z_q]$; the class  described in $(10)$ is the unique non-commutative non-cocommutative pointed Hopf algebra of dimension $p^2q$ whose diagram is not a Nichols algebra, which constitute new examples of non-commutative non-cocommutative pointed Hopf algebras.
\end{rmk}

\begin{thm}\label{thm:dimp2q-completeclassification}
Let $H$ be a pointed Hopf algebra over $\K$ of dimension $p^2q$. If $H$ is generated by group-like elements and skew-primitive elements, then $H$ is one of the algebras listed in Propositions \ref{pro:p2q-1} and \ref{pro:p2q-rank-2-liftings}; otherwise $H$ is one of the algebras listed in Proposition \ref{pro:p2q-2}.
\end{thm}
\begin{proof}
By Lemma \ref{lem:p2q-H0-1}, the diagram of $H$ has dimension $p$ or $p^2$. Therefore, it follows by Propositions \ref{pro:p2q-1}, \ref{pro:p2q-rank-2-liftings} and \ref{pro:p2q-2}.
\end{proof}
\subsection{On pointed Hopf algebras of dimension $p^2m$}
Let $m$ be square-free and $p\nmid m$. Now we first give a classification of pointed Hopf algebras of dimension $p^2m$ with abelian coradicals whose diagrams are Nichols algebras.
\begin{lem}\label{lem:p2m-H0-2}
Suppose that $\Char\K=p$, $m$ is square-free and $p\nmid m$. Let $H$ be a pointed Hopf algebra over $\K$  of dimension $p^2m$ whose diagram is a Nichols algebra. Then  $\dim R=p$ or $p^2$. Furthermore,
\begin{itemize}
  \item [(1)] If $\dim R=p$, then $R\cong\K[x]/(x^p)$ with $x\in R(1)_g^{\chi}$ satisfying $\chi(g)=1$ for some $g\in\mathcal{Z}(\G(H))$ and $\chi\in\widehat{\G(H)}$;
  \item [(2)] If $\dim R=p^2$, then $R\cong\BN(V,\D)$ for some QPYD-datum $\D:=\D(\G(H),\chi_1,\chi_2,g_1,g_2)$.
\end{itemize}

\end{lem}
\begin{proof}
We first claim that $p\mid\dim R$. Let $\dim H_0=p^kn$ with $p\nmid n$ for $k\in\I_{0,2}$. Then $\dim R=p^{2-k}m/n$. By assumption, there is an element $x\in R(1)_g^{\chi}$ for some YD-pair $(g,\chi)$ such that $c(x\otimes x)=\chi(g)x\otimes x$. Then $\BN(\K\{x\})$ is a Hopf subalgebra of $R$ in ${}_{\G(H)}^{\G(H)}\mathcal{YD}$ and hence $\dim\BN(\K\{x\})\mid \dim R$.  Suppose that $\chi(g)\neq 1$. Then $\chi(g)=\xi^{i}$ for $i\in\I_{1,n-1}$, where $\xi$ is a primitive $n$th root of unity and hence $\ord(\xi^i)\mid n$, which implies that $\ord(\xi^i)\nmid m/n$. On the other hand, $\dim\BN(\K\{x\})=\ord(\xi^i)\mid m/n$, a contradiction. Consequently, $\chi(g)=1$ and hence $p=\dim\BN(\K\{x\})\mid \dim R$.

Since  $p\mid\dim R$, it follows that the order of $\G(H)$ is square-free and  $\G(H)\cong\Z_{p^{k}n}$ for $k\in\I_{0,1}$.

 We claim that $p^2\mid\dim R$ when $\dim R\neq p$. Assume that $\dim R\neq p$. If $\K\{x\}$ is  a subobject of some non-simple indecomposable object $\mathcal{V}$, then  by \cite[Example 2.2]{X23} $p^2\mid\dim R$, which implies that $p\nmid\dim H_0$ and so ${}_{H_0}^{H_0}\mathcal{YD}$ are semisimple, a contradiction. Then there is an element $y\in R(1)-\K\{x\}$  such that $y\in R(1)_{g'}^{\chi'}$. Furthermore, $\chi'(g')=1$. Let $W:=\K\{x,y\}$. Then
$\BN(W)$ is of diagonal type whose braiding $(q_{i,j})_{i,j\in\I_{1,2}}$ is given by $q_{i,i}=1$. Then by  \cite[Theorem 5.1]{WH}, $\BN(W)$ is finite-dimensional if and only if $q_{1,2}q_{2,1}=1$. Hence $\BN(W)=\K[x,y]/(x^p,y^p)$ and then $p^2\mid\dim R$.

If $\dim R\neq p$, then $p^2\mid\dim R$ and hence $\G(H)\cong\Z_n$. Similarly, if $\dim R\not\in\{p,p^2\}$, then $p^3\mid\dim R$, a contradiction. Consequently, $\dim R=p$ or $p^2$.
\end{proof}

\begin{lem}\label{lem:dimH0-dimp2m-nonN}
Suppose that $\Char\K=p$, $m$ is square-free and $p\nmid m$. Let $H$ be a pointed Hopf algebra over $\K$  of dimension $p^2m$ whose diagram is not a Nichols algebra. Then  $\dim R=p^2$.
\end{lem}
\begin{proof}
By Lemma \ref{lem:p2m-H0-2}, $\BN(R(1))\cong \K[x]/(x^p)$, where $R(1):=\K\{x\}$.  Then by Corollary \ref{cor:p2dividedimR}, $p^2\mid\dim R$.
If $\dim R\neq p^2$, then by \cite[Lemma 3.2]{X22}, there are a braided Hopf subalgebra of dimension $p^3$, which is isomorphic to $\K[x,y,x]/(x^p,y^p,z^p)$ with
\begin{align*}
\Delta_R(x)&=x\otimes 1+1\otimes x,\quad
\Delta_R(y)=y\otimes 1+1\otimes y+\omega_0(x),\\
\Delta_R(z)&=z\otimes 1+1\otimes z+\omega_0(x)(y\otimes 1+1\otimes y)^{p-1}+\omega_0(y),
\end{align*}
which implies that $p^3\mid\dim R$, a contradiction. Consequently, we have $\dim R=p^2$.
\end{proof}
\begin{thm}\label{thm:p2m-withabeliancoradical}
Suppose that $\Char\K=p$, $m$ is square-free and $p\nmid m$. Let $H$ be a pointed Hopf algebra over $\K$  of dimension $p^2m$ with abelian coradical.  Then $H$ is isomorphic to one of the following Hopf algebras:
\begin{itemize}
\item[(i)] $\fA^k_{\G(H)}(g,\chi,f)$ for $k\in\I_{1,2}$ and $\fA^3_{\G(H)}(g,f)$;
\item[(ii)]$\dH^k(\D)$ for $k\in\I_{1,8}$;
\item[(iii)] $\fH^k_{\G(H)}(g,\chi)$ for $k\in\I_{1,4}$.
\end{itemize}
Furthermore,
\begin{itemize}
\item $\fA_G^1(g,\chi,f)\cong\fA_G^1(g',\chi',f')$  if and only if there exists $F\in\Aut(G)$ such that $F(g)=g'$,  $\chi\cdot F^{-1}=\chi'$ and $\alpha f'F-f+\beta(\chi-\epsilon)=0$ for some $\alpha\in\K^{\times},\beta\in\K$ satisfying $\beta(1-g^p)=0$.
\item $\fA_G^2(g,\chi,f)\cong\fA_G^2(g',\chi',f')$   if and only if there exists $F\in\Aut(G)$ such that $F(g)=g'$,  $\chi\cdot F^{-1}=\chi'$ and $\alpha f'F-f+\beta(\chi-\epsilon)=0$ for some $\alpha\in\K^{\times},\beta\in\K$ satisfying $\alpha^p=\alpha$ and $(\beta^p-\beta)(1-g)=0$.
    \item $\fA_{\G(H)}^3(g,\eta)\cong\fA_{\G(H)}^3(g',\eta')$ if and only if there exists $F\in\Aut(\G(H))$ such that $F(g)=g'$ and $f\cdot F^{-1}=f'$;
    \item  $\dH^k(\D)\cong\dH^k(\D')$ for $k\in\I_{1,8}$  if and only if  there are $F\in\Aut(\G(H))$ and $\sigma\in S_2$ such that $F(g_{\sigma(i)})=g_i^{\prime}$ and $\chi_{\sigma(i)}\cdot F^{-1}=\chi_{i}^{\prime}$;
    \item $\fH_{\G(H)}^i(g,\chi)\cong\fH_{\G(H)}^i(g',\chi')$ for $i\in\I_{1,4}$ if and only if there exists $F\in\Aut(\G(H))$ such that $F(g)=g'$ and $\chi\cdot F^{-1}=\chi'$.
\end{itemize}
\end{thm}
\begin{proof}
If the diagram is a Nichols algebra, then it follows by Lemma \ref{lem:p2m-H0-2}, Theorems \ref{thm:classification-rank-one}, \ref{thm:liftingsofdH}; otherwise it follows by Lemma \ref{lem:dimH0-dimp2m-nonN}  and Theorem \ref{thm:H-group-algebra-of-dim-m}.
\end{proof}

\section*{Acknowledgements}
The author is grateful to Prof. V. C. Nguyen for her suggestions and comments on earlier draft of this article and to Profs. Quanshui Wu and Xingting Wang so much for the  help and  encouragement during his visiting at the Shanghai Center for Mathematical Sciences. The author would like to thank  the referee for many helpful comments and suggestions that improved the exposition.


\begin{thebibliography}{50}
\bibitem{AA17}N. Andruskiewitsch and I. E. Angiono, On finite dimensional Nichols algebras of diagonal type, Bulletin of Mathematical Sciences 7 (3) (2017), 353--573.
\bibitem{AAH19} N. Andruskiewitsch,  I. Angiono and I. Heckenberger, \emph{Examples of finite-dimensional pointed Hopf algebras in positive characteristic}. Preprint: arXiv:1905.03074.


\bibitem{ABFF}N. Andruskiewitsch, D. Bagio, S. D. Flora and D. Flres, \emph{Examples of finite-dimensional pointed Hopf algebras in characteristic 2}, Glasgow Math. J. \textbf{64} (2022), 65--78.

\bibitem{AP}N. Andruskiewitsch and H. Pena Pollastri,  \emph{On the restricted Jordan plane in odd characteristic}, J. Algebra Appl.  \textbf{20} (1) (2021), 2140012
\bibitem{AG99} N. Andruskiewitsch and  M. Gra\~{n}a, \emph{Braided Hopf algebras over non abelian finite groups}, Bol. Acad. Nac. Cienc. Cordoba \textbf{63} (1999), 46--78.
\bibitem{AN01} N. Andruskiewitsch and S. Natale, \emph{Counting arguments for Hopf algebras of low dimension}, Tsukuba Math J. \textbf{25} (1) (2001), 187--201.
\bibitem{AS98b} N. Andruskiewitsch and H. J.~Schneider, \emph{Lifting of quantum linear spaces and pointed Hopf algebras of order $p^3$}, J.~Algebra \textbf{209} (1998), 658--691.
\bibitem{AS02} N. Andruskiewitsch and H. J.~Schneider, \emph{Pointed Hopf algebras}, New directions in Hopf algebras, 1--68, Math.~Sci.~Res.~Inst.~Publ., 43, Cambridge Univ. Press, Cambridge, 2002.
\bibitem{AI19}I. Angiono and  A. G. Iglesias, \emph{Liftings of Nichols algebras of diagonal type II: all liftings are cocycle deformations}. Sel. Math. New Ser. \textbf{25} (5) (2019). https://doi.org/10.1007/s00029-019-0452-4
\bibitem{BG13} M. Beattie and G. A. Garcia, \emph{Classifying Hopf algebras of a given dimension}, Contemp. Math. \textbf{585} (2013), 125--152.
\bibitem{B}G. Bergman, \emph{The diamond lemma for ring theory}, Adv. Math. \textbf{29} (1978), 178--218.

\bibitem{CLW}C. Cibils, A. Lauve and S. Witherspoon, \emph{Hopf quivers and Nichols algebras in positive characteristic}, Proc. Amer. Math. Soc. \textbf{137} (2009), 4029--4041.
\bibitem{G} M. Gra\~{n}a, \emph{Freeness theorem for Nichols algebras}, J. Algebra \textbf{231} (1) (2000), 235--257.

\bibitem{Hen95}G. Henderson, \emph{Low dimensional cocommutative connected Hopf algebras}, J. Pure Appl. Algebra \textbf{102} (1995), 173--193.
\bibitem{HW07}N. Hu and X. Wang,   \emph{Quantizations of generalized-Witt algebra and of Jacobson-Witt algebra in the modular case}, J. Algebra \textbf{312} (2) (2007),  902--929.
\bibitem{HW11}N. Hu and X. Wang,  \emph{Twists and quantizations of Cartan type S Lie algebras}, J. Pure Appl. Algebra \textbf{215} (6) (2011),  1205--1222.
\bibitem{J}N. Jacobson, \emph{Lie Algebras}, Dover Publications Inc., New York, 1979.
\bibitem{Ja}J. Jantzen, Representations of algebraic groups. Second edition. Mathematical Surveys and Monographs 107. American Mathematical Society, Providence, RI, 2003.
\bibitem{NW18} S. H.~Ng and X. Wang, \emph{Hopf algebras of prime dimension in positive characteristic}, Bull. London Math. Soc. \textbf{51} (3) (2019), 459--465.
\bibitem{NWW1}V. C. Nguyen, L. Wang and X. Wang, \emph{Classification of connected Hopf algebras of dimension $p^3$ I}, J. Algebra \textbf{424} (2015), 473--505.


\bibitem{NWW2}V. C. Nguyen, L. Wang and X. Wang, \emph{Primitive deformations of quantum p-groups},
Algebr. Represent. Theor. \textbf{22} (2019), 837--865.
\bibitem{NW}V. C. Nguyen and X. Wang, \emph{Pointed $p^3$-dimensional Hopf algebras in positive characteristic}, Algebra Colloquium \textbf{25} (3) (2018), 399--436.
\bibitem{NZ}W. D. Nichols and M. B. Zoeller, \emph{A Hopf algebra freeness theorem}, Amer. J. of Math. \textbf{111} (2) (1989), 381--385.
%
\bibitem{R85} D. E. Radford, \emph{The structure of Hopf algebras with a projection}, J. Algebra \textbf{92} (2) (1985), 322--347.
\bibitem{R99}D. E. Radford, \emph{Finite-dimensional simple-pointed Hopf algebras}, J. Algebra \textbf{211} (1999), 686--710.
\bibitem{R11}D. E. Radford, \emph{Hopf algebras}, Series on Knots and Everything, \textbf{49}, World Scientific Publishing
     Co. Pte. Ltd., Singapore, 2012.
\bibitem{Rosso}M. Rosso, Quantum groups and quantum shuffles. Inventiones Mathematicae 133(2) (1998), 399--416.
\bibitem{Scha01} B. Scharfschwerdt, The Nichols Zoeller Theorem for Hopf algebras in the category of Yetter-Drinfeld modules,  Commu.
Algebra  \textbf{29}  (2001), 2481--2487.
\bibitem{S}S. Scherotzke, \emph{Classification of pointed rank one Hopf algebras}, J. Algebra \textbf{319} (2008), 2889--2912.
\bibitem{St}D. \c{S}tefan, Hochschild cohomology on Hopf Galois extensions, J. Pure Appl. Algebra 103(2) (1995), 221--233.
\bibitem{SO}D. \c{S}tefan and F. van Oystaeyen, \emph{Hochschild cohomology and the coradical filtration of pointed coalgebras: applications}, J. Algebra \textbf{210} (1998), 535--556.

\bibitem{TH16}Z. Tong and N. Hu, \emph{Modular quantizations of Lie algebras of Cartan type K via Drinfeld twists of Jordanian type}, J. Algebra \textbf{450} (2016), 102--151.
\bibitem{THW15}Z. Tong,  N. Hu and X. Wang, \emph{Modular quantizations of Lie algebras of Cartan type H via Drinfel'd twists}, Lie algebras and related topics, 173--206, Contemp. Math. 652, Amer. Math. Soc. Providence, RI, 2015.
\bibitem{WZZ}D. G. Wang, J. J. Zhang and G. Zhuang, \emph{Primitive cohomology of Hopf algebras}, J. Algebra \textbf{464} (2016), 36--96.
\bibitem{WJrank3}J. Wang, \emph{Rank three Nichols algebras of diagonal type over arbitrary fields}, Isr. J. Math., 218 (2017), 1--26.

\bibitem{WJrank4}J. Wang, \emph{Rank 4 finite-dimensional Nichols algebras of diagonal type in positive characteristic}, J. Algebra, \textbf{559} (2020), 547--579.
\bibitem{WH}J. Wang and I. Heckenberger, \emph{Rank 2 Nichols algebras of diagonal type over fields of positive
characteristic}, SIGMA Symmetry Integrability Geom. Methods Appl. \textbf{011} (2015), 24 pages.


\bibitem{WW}L. Wang and X. Wang, \emph{Classification of pointed Hopf algebras of dimension $p^2$ over any algebraically
     closed field}, Algebr. Represent. Theory \textbf{17} (2014), 1267--1276.
\bibitem{W1}X. Wang, \emph{Connected Hopf algebras of dimension $p^2$}, J. Algebra \textbf{391} (2013), 93--113.

\bibitem{X23}R. Xiong, On non-connected pointed Hopf algebras of dimension 16 in characteristic $2$, J. Algebra Appl.  (2023), 2350214 (52 pages). https://doi.org/10.1142/S0219498823502146
\bibitem{X22}R. Xiong, On  pointed Hopf algebras of dimension 16 whose diagrams are not Nichols algebras. Submitted for publication.
\end{thebibliography}
\end{document}